\RequirePackage{fix-cm}
\documentclass[smallextended]{svjour3}       
\smartqed  

\usepackage{fullpage}

\usepackage{amsfonts}
\usepackage{amsmath}
\usepackage{amssymb}

\usepackage[dvips]{graphicx}
\usepackage{color}
\usepackage{algorithmic}
\usepackage{algorithm}
\usepackage{multirow}

\usepackage{comment}
\usepackage{stackrel}
\usepackage{url}
\usepackage{subfloat}
\usepackage{subfig}
\newtheorem{assumption}[theorem]{Assumption}

\newcommand{\R}{\mathbf{R}}
\newcommand{\E}{\mathbf{E}}

\newcommand{\eqdef}{:=}

\newcommand{\Lmax}{\Lambda_{\max}}

\newcommand{\Lip}{L}

\newcommand{\Prob}{\mathbf{P}}
\newcommand{\prox}{{\rm{prox}}}

\DeclareMathOperator*{\argmin}{arg\,min}

\usepackage{mathtools}
\DeclarePairedDelimiter{\ceil}{\lceil}{\rceil}

\def\mplus{\mathrel{%
  \ooalign{\raise.29ex\hbox{$\scriptscriptstyle\mathbf{+}$}\cr}}}

 \def\mminus{\mathrel{%
  \ooalign{\raise.29ex\hbox{$\scriptscriptstyle\mathbf{-}$}\cr}}}

\title{A Flexible Coordinate Descent Method}
\author{Kimon Fountoulakis and Rachael Tappenden}

\author{
        Kimon~Fountoulakis\thanks{K. Fountoulakis is with the International Computer Science Institute, Department of Statistics, University of California Berkeley, 1947 Center St, Ste. 600, Berkeley, CA 94704, USA. e-mail: kfount@berkeley.edu.} \and
        Rachael~Tappenden\thanks{R. Tappenden (corresponding author) is with the School of Mathematics and Statistics, University of Canterbury, Private Bag 8400, Christchurch 8041, NZ. e-mail: rachael.tappenden@canterbury.ac.nz.}
}

\begin{document}

\maketitle
\begin{abstract}
We present a novel randomized block coordinate descent method for the minimization of a convex composite objective function. The method uses (approximate) partial second-order (curvature) information, so that the algorithm performance is more robust when applied to highly nonseparable or ill conditioned problems. We call the method Flexible Coordinate Descent (FCD). At each iteration of FCD, a block of coordinates is sampled randomly, a quadratic model is formed about that block and the model is minimized \emph{approximately/inexactly} to determine the search direction. An inexpensive line search is then employed to ensure a monotonic decrease in the objective function and acceptance of large step sizes. We present several high probability iteration complexity results to show that convergence of FCD is guaranteed theoretically. Finally, we present numerical results on large-scale problems to demonstrate the practical performance of the method.
\end{abstract}


\section{Introduction}
In this work we are interested in solving the following convex composite optimization problem
\begin{equation}
\label{Def_F}
  \min_{x \in \R^N}  F(x) \eqdef f(x) + \Psi(x),
\end{equation}
where $f(x)$ is a smooth convex function and $\Psi(x)$ is a (possibly) nonsmooth, (coordinate) separable, real valued convex function. Problems of the form \eqref{Def_F} arise in many important scientific fields, and applications include machine learning \cite{yuanho}, regression \cite{IEEEhowto:Tibshirani} and compressed sensing \cite{IEEEhowto:CandesRombergTao,Candes06,Donoho06}. Often the term $f(x)$ is a data fidelity term, and the term $\Psi(x)$ represents some kind of regularization.

Frequently, problems of the form \eqref{Def_F} are large-scale, i.e., the size of $N$ is of the order of a million or a billion. Large-scale problems impose restrictions on the types of methods that can be employed for the solution of \eqref{Def_F}. In particular, the methods should have low per iteration computational cost, otherwise completing even a single iteration of the method might require unreasonable time. The methods must also rely only on simple operations such as matrix vector products, and ideally, they should offer fast progress towards optimality.

First order methods, and in particular randomized coordinate descent methods, have found great success in this area because they take advantage of the underlying problem structure (separability and block structure), and satisfy the requirements of low computational cost and low storage requirements. For example, in \cite{Richtarik14} the authors show that their randomized coordinate descent method was able to solve sparse problems with millions of variables in a reasonable amount of time.

Unfortunately, randomized coordinate descent methods have two significant drawbacks. Firstly, due to their coordinate nature, coordinate descent methods are usually efficient only on problems with a high degree of partial separability,\footnote{See \cite[Equation~(2)]{Richtarik15} or \cite[Definition~13]{Tappenden15} for a precise definition of `partial separability'.} and performance suffers when there is a high dependency between variables. Secondly, as first-order methods, coordinate descent methods do not usually capture essential curvature information and have been shown to struggle on complicated sparse problems \cite{Fountoulakis15,l1regSCfg}.

The purpose of this work is to overcome these drawbacks by equipping a randomized block coordinate descent method with approximate partial second-order information. In particular, at every iteration of FCD a search direction is obtained by solving a local quadratic model \textit{approximately}. The quadratic model is defined by a matrix representing approximate second order information, so that essential curvature information is incorporated. The model need only be minimized \emph{approximately} to give an \emph{inexact} search direction, so that the process is efficient in practice. (The termination condition for the quadratic subproblem is inspired by \cite{sqa}, and we discuss this in more detail later.) A line search is then employed along this inexact direction, in order to guarantee a monotonic decrease in the objective function and large step-sizes.

FCD is computationally efficient. At each iteration of FCD, a block of coordinates is randomly selected, which is inexpensive. The method allows the use of inexact search directions, i.e., the quadratic model is minimized approximately.
Intuitively, it is expected that this will lead to a reduction in algorithm running time, compared with minimizing the model exactly at every iteration.
Also, the line search step depends on a subset of the coordinates only (corresponding to the randomly selected block of coordinates), so it is much cheaper compared with working in the full dimensional space. We note that the per iteration computational cost of the FCD method may be slightly higher than other randomized coordinate descent methods. This is due to the incorporation of the matrix representing partial second order information --- the matrix gives rise to a quadratic model that is (in general) nonseparable, so the associated subproblem may not have a closed form solution. However, in the numerical results presented later in this work, we show that the method is more robust (and efficient) in practice, and often FCD will require fewer iterations to reach optimality, compared with other state-of-the-art methods. i.e., we show that FCD is able to solve difficult problems, on which other coordinate descent methods may struggle.

The FCD method is supported by theoretical convergence guarantees in the form of high probability iteration complexity results. These iteration complexity results provide an explicit expression for the number of iterations $k$, which guarantees that, for any given error tolerance $\epsilon>0$, and confidence level $\rho \in (0,1)$, the probability that FCD achieves an $\epsilon$-optimal solution, exceeds $1-\rho$, i.e., $\mathbf{P}(F(x_k)-F^* \leq \epsilon)\geq 1-\rho$.

\subsection{Literature review}

Coordinate descent methods are some of the oldest and simplest iterative methods, and they are often better known in the literature under various names such as Jacobi methods, or Gauss-Seidel methods, among others. It has been observed that these methods may suffer from poor practical performance, particularly on ill-conditioned problems, and until recently, often higher order methods have been favoured by the optimization community. However, as we enter the era of big data, coordinate descent methods are coming back into favour, because of their ability to provide approximate solutions to real world instances of very large/huge scale problems in a reasonable amount of time.

Currently, state-of-the-art randomized coordinate descent methods include that of Richt\`{a}rik and Tak\`{a}\v{c} \cite{Richtarik14}, where the method can be applied to unconstrained convex composite optimization problems of the form \eqref{Def_F}. The algorithm is supported by theoretical convergence guarantees in the form of high probability iteration complexity results, and \cite{Richtarik14} also reports very impressive practical performance on highly separable large scale problems. Their work has also been extended to the parallel case \cite{Richtarik15}, to include acceleration techniques \cite{Fercoq13}, and to include the use of inexact updates \cite{Tappenden13}.

Other important works on coordinate descent methods include that of Tseng and Yun in \cite{tsengyun,Tseng09,Tseng10}, Nesterov \cite{Nesterov12} for huge-scale problems, work in \cite{Lu13} and \cite{Tappenden15} that improves the complexity analysis of \cite{Richtarik14} and \cite{Richtarik15} respectively, coordinate descent methods for group lasso problems \cite{Qin13,Simon12} or for problems with general regularizers \cite{ShalevSchwartz13,Wright12}, a coordinate descent method that uses a sequential active set strategy \cite{desantis14}, and coordinate descent for constrained optimization problems \cite{Necoara14}.

Unfortunately, on ill-conditioned problems, or problems that are highly nonseparable, first order methods can display very poor practical performance, and this has prompted the study of methods that employ second order information. To this end, recently there has been a flurry of research on Newton-type methods for problems of the general form \eqref{Def_F}, or in the special case where $\Psi(x)=\|x\|_1$. For example, Karimi and Vavasis \cite{Karimi14} have developed a proximal quasi-Newton method for $l_1$-regularized least squares problems, Lee, Sun and Saunders \cite{Lee12,Lee13} have proposed a family of Newton-type methods for solving problems of the form \eqref{Def_F} and Scheinberg and Tang \cite{Scheinberg13} present iteration complexity results for a proximal Newton-type method. Moreover, the authors in \cite{sqa} extended standard inexact Newton-type methods to the case of minimization of a composite objective involving a smooth convex term plus an $l_1$-regularizer term.

Facchinei and coauthors in \cite{facchinei14a,facchinei14b,facchinei14} have also made a significant contribution to the literature on randomized coordinate descent methods. Two new algorithms are introduced in those works, namely FLEXA \cite{facchinei14b,facchinei14} and HyFLEXA \cite{facchinei14a}. Both methods can be applied to problems of the form \eqref{Def_F} where $f$ is allowed to be nonconvex, and for HyFLEXA, $\Psi$ is allowed to be nonseparable. We stress that nonconvexity of $f$, and nonseparability of $\Psi$ are \emph{outside the scope of this work}. Nevertheless, these methods are pertinent to the current work, so we discuss them now.

For FLEXA and HyFLEXA, (as is standard in the current literature), the block structure of the problem is fixed throughout the algorithm: the data vector $x \in \R^N$ is partitioned into $n(\leq N)$ blocks, $x = (x^1,\dots,x^n)\in \R^N$. At each iteration, one block, all blocks, or some number in between, are selected in such a way that some specified error bound is satisfied. Next, a local \emph{convex} model is formed and that model is approximately minimized to find a search direction. Finally, a convex combination of the current point, and the search direction is taken, to give a new point, and the process is repeated. The model they propose is \emph{block separable}, so the subproblems for each block are independent. (See (3) in \cite{facchinei14} for the local/block quadratic model, or F1--F3 in \cite{facchinei14a} for a description of the strongly convex local/block model). This has the advantage that the subproblem for each block can be solved/minimized in parallel, but has the disadvantage that no interaction between different blocks is captured by either algorithm. We also note that, while these algorithms are equipped with global convergence guarantees, iteration complexity results have not been established for either method, which is a drawback.

We make the following two comments about FLEXA and HyFLEXA. The first is that, although, as presented in \cite{facchinei14} and \cite{facchinei14a}, coordinate selection must respect the block structure of the problem, if the regularizer is assumed to be coordinate separable (which is the assumption we make in this work) then the convergence analysis in \cite{facchinei14} and \cite{facchinei14a} can be adapted to hold in the `non-fixed block' setting. The second comment is that, while the works in \cite{facchinei14} and \cite{facchinei14a} only present global convergence guarantees, a variation of the method, presented in \cite{Razaviyayn14}, is equipped with iteration complexity results. The algorithm in this paper incorporates partial (possibly inexact) second order information, allows inexact updates with a verifiable stopping condition, and is supported by iteration complexity guarantees.

Another important method that has recently been proposed is a randomized coordinate decent method by Qu et al. \cite{Qu15}. One of the most notable features of their algorithm is that the block structure is \emph{not fixed throughout the algorithm}. That is, at each iteration of their algorithm, a random subset of coordinates $S_k$ is selected from the set of coordinates $\{1,\dots,N\}$. Then, a quadratic model is minimized exactly to obtain an update, and a new point is generated by applying the update to the current point. Unfortunately, this algorithm is only applicable to strongly convex smooth functions, or strongly convex empirical risk minimization problems, but is \emph{not suitable for general convex problems of the form \eqref{Def_F}}. Moreover, the matrices used to define their quadratic model must obey several (strong) assumptions, and the quadratic model must be minimized \emph{exactly}. These restrictions may make the algorithm inconvenient from a practical perspective.

The purpose of this work is to build upon the positive features of some existing coordinate descent methods and create a general flexible framework that can be used to solve any problem of the form \eqref{Def_F}. We will adopt some of the ideas in \cite{sqa,facchinei14,Qu15} (including variable block structure; the incorporation of partial approximate second order information; the practicality of approximate solves and inexact updates; cheap line search) and create a new flexible randomized coordinate descent framework that is not only computationally practical, but is also supported by strong theoretical guarantees.

\subsection{Contributions}
In this section we list several of the contributions of this paper, which correspond to the central ideas of our FCD algorithm.
\begin{enumerate}
\item \textbf{Blocks can vary throughout iterations.}
Existing randomized coordinate descent methods initially partition the data $x\in \R^N$ into $n(\leq N)$ blocks ($x = (x^1,\dots,x^n)\in \R^N$) and at each iteration, one/all/several of the blocks are selected according to some probability distribution, and those blocks are then updated. For those methods, the block partition is \emph{fixed throughout the entire algorithm.} So, for example, coordinates in block $x^1$ will always be updated all together as a group, independently of any other block $x^i$ $i \neq 1$. One of the main contributions of this work is that we allow the blocks to \emph{vary throughout the algorithm}. i.e., the method is not restricted to a fixed block structure. No partition of $x$ need be initialized. Rather, when FCD is initialized, a parameter $1\leq \tau \leq N$ is chosen and then, at each iteration of FCD, a subset $S \subseteq \{1,\dots, N\}$ of size $|S| = \tau $ is randomly selected/generated, and the coordinates of $x$ corresponding to the indices in $S$ are updated.
%
To the best of our knowledge, the only other paper that allows for this random mixing of coordinates/varying blocks strategy, is the work by Qu et al. \cite{Qu15}.\footnote{An earlier version of this work, which, to the best of our knowledge, was the first to propose varying random block selection, is cited by \cite{Qu15}.} This variable block structure is crucial with regards to our next major contribution.
\begin{remark}
  We note that the algorithms of Tseng and Yun \cite{tsengyun,Tseng09,Tseng10} also allow a certain amount of coordinate mixing. However, their algorithms enforce deterministic rules for coordinate subset selection (either a `generalized Gauss-Seidel', or Gauss-Southwell rule), which is different from our randomized selection rule. This difference is important because their deterministic strategy can be expensive/inconvenient to implement in practice, whereas our random scheme is simple and cheap to implement.
\end{remark}

 \item \textbf{Model: Incorporation of partial second order information.} FCD uses a quadratic model to determine the search direction. The model is defined by any
     symmetric positive definite matrix $H^S(x_k)$, which depends on both the subset of coordinates $S$, and also on the current point $x_k$. i.e., $H^S(x_k)$ is not fixed; rather, it can change/vary throughout the algorithm. We stress that this is a key advantage over most existing methods, which enforce the symmetric positive definite matrix defining their quadratic model to be fixed with respect to the fixed block structure, and/or iteration counter. To the best of our knowledge, the only works that allow the matrix $H^S(x_k)$ to vary with respect to the subset $S$ is this one (FCD), and the work by Qu et al. \cite{Qu15}. A crucial observation is that, if $H^S(x_k)$ approximates the Hessian, our approach allows \emph{every element of the Hessian to be accessed}. On the other hand, other existing methods can only access blocks along the diagonal of (some perturbation of) the Hessian, (including \cite{facchinei14a,facchinei14b,facchinei14,Richtarik15,Richtarik14,Tappenden15}.)

     Furthermore, the only restriction we make on the choice of the matrix $H^S(x_k)$ is that it be symmetric and positive definite.
     This is a much weaker assumption than made by Qu et al. \cite{Qu15}, who insist that the matrix defining their quadratic model be a principle submatrix of some fixed/predefined $N \times N$ matrix $M$ say, and the large matrix $M$ must be symmetric and \emph{sufficiently positive definite}. (See Section 2.1 in \cite{Qu15}.)

\item \textbf{Inexact search directions.} To ensure that FCD is computationally practical, it is imperative that the iterates are inexpensive, and this is achieved through the use of \emph{inexact} updates. That is, the model can be \emph{approximately} minimized, giving \emph{inexact search directions}. This strategy makes intuitive sense because, if the current point is far from the true solution, it may be wasteful/unnecessary to exactly minimize the model. Any algorithm can be used to approximately minimize the model, and if $H^S(x_k)$ approximates the Hessian, then the search direction obtained by minimizing the model is an approximate Newton-type direction. Importantly, the stopping conditions for the inner solve are \emph{easy to verify}; they depend upon quantities that are easy/inexpensive to obtain, or may be available as a byproduct of the inner search direction solver. {\color{black} Block coordinate descent methods that use inexact search directions are uncommon in the literature because it is often notoriously difficult to check the conditions that determine whether an inexact search direction is `acceptable'. A positive feature of our algorithm is that it is computationally practical to identify and verify the inexact directions used in FCD.}
    \begin{remark}
      The precise form of our inexactness termination criterion is important because, not only is it computationally practical (i.e., easy to verify), it also allowed us to derive iteration complexity results for FCD (see point 5). We considered several other termination criterion formulations, but we were unable to obtain corresponding iteration complexity results (although global convergence results were possible). Currently, the majority of randomized CD methods require exact solves for their inner subproblems and we believe that a major reason for this is because iteration complexity results are significantly easier to obtain in the exact case. Coming up with practical inexact termination criteria for randomized CD methods that also allow the derivation of iteration complexity results seems to be a major hurdle in this area, although some progress is being made, e.g., \cite{Cassioli13,Devolder13,Devolder11,Tappenden13}.
    \end{remark}

  \item \textbf{Line search.} FCD includes a line search to ensure a monotonic decrease of the objective function as iterates progress. (The line search is needed because we only make weak assumptions on the matrix $H^S(x_k)$.)
  The line search is inexpensive to perform because, at each iteration, \emph{it depends on a subset of coordinates $S$ only}.
  One of the major advantages of incorporating second-order information combined with a line search is to allow, in practice, the selection of \emph{large step sizes} (close to one).
  This is because unit step sizes can substantially improve the practical efficiency of a method. In fact, for all experiments that we performed,
  unit step sizes were accepted by the line search for the majority of the iterations. {\color{black} A commonly held view in this field is that function evaluations costs are unacceptably high, and so a line-search should not be used. Instead, block coordinate descent methods almost always use a fixed step size. However, FCD does indeed use a line search, and the numerical experiments presented in Section~\ref{S_Numerical} show that our algorithm is competitive with, and often outperforms, other state-of-the-art block coordinate descent methods that use a fixed step size. Thus, we would argue that the previously mentioned reservations regarding the use of a line search are not necessarily well founded.}
  \item \textbf{Convergence theory.} We provide a complete convergence theory for FCD. In particular, we provide high probability iteration complexity guarantees for the algorithm in both the convex and strongly convex cases, and in the cases of both inexact or exact search directions. Our results show that FCD converges at a sublinear rate for convex functions, and at a linear rate for strongly convex functions. {\color{black} Complexity theory for block coordinate descent methods that use inexact updates is also uncommon, which is another strength of FCD.}
\end{enumerate}

\subsection{Paper Outline}

The paper is organized as follows. In Section \ref{Section_Preliminaries} we introduce the notation and definitions that are used throughout this paper, including the definition of the quadratic model, and stationarity conditions. A thorough description of the FCD algorithm is presented in Section \ref{S_Algorithm}, including how the blocks are selected/sampled at each iteration (Section~\ref{S_blockstructure}), a description of the search direction, several suggestions for selecting the matrices $H^S(x_k)$, and analysis of the subproblem/model termination conditions (Section~\ref{SS_searchdirectionHessian}),   a definition of the line search (Section~\ref{subsec:pract_line}) and we also present several concrete problem examples (Section~\ref{SS_concrete_examples}). In Section~\ref{S_GlobalConvergence} we show that the line search in FCD will always be satisfied for some positive step $\alpha$, and that the objective function value is guaranteed to decrease at every iteration. Section~\ref{S_technicalresults} introduces several technical results, and in Section~\ref{S_Complexity} we present our main iteration complexity results. In Section~\ref{S_ComplexitySmooth} we give a simplified analysis of FCD in the smooth case. Finally, several numerical experiments are presented in Section \ref{S_Numerical}, which show that the algorithm performs very well in practice, even on ill-conditioned problems.

\section{Preliminaries}
\label{Section_Preliminaries}
Here we introduce the notation and definitions that are used in this paper, and we also present some important technical results. Throughout the paper we denote the standard Euclidean norm by $\|\cdot\| \equiv \sqrt{\langle \cdot, \cdot \rangle}$
and the ellipsoidal norm by $\|\cdot\|_A \equiv \sqrt{\langle \cdot, A\cdot \rangle}$, where $A$ is a symmetric positive definite matrix. Furthermore, $\R^{N}_{++}$ denotes a vector in $\R^N$ with (strictly) positive components.

\paragraph{Subgradient and subdifferential.}
For a function $\Phi: \R^N \to \R$ the elements $s\in\R^N$ that satisfy
$$
     \Phi(y) \geq \Phi(x) + \langle s,y-x \rangle, \quad \forall y\in\R^N,
$$
are called the subgradients of $\Phi$ at point $x$. In words, all elements defining a linear function that supports $\Phi$ at a point $x$ are subgradients. The set of all $s$ at a point $x$ is called the subdifferential of $\Phi$ and it is denoted by $\partial \Phi(x)$.

\paragraph{Convexity.}
\label{S_StrongConvex}
A function $\Phi: \R^N \to \R$ is strongly convex with convexity parameter $\mu_{\Phi} > 0$ if
\begin{equation}
\label{strongly_convex_1}
     \Phi(y) \geq \Phi(x) + \langle s,y-x \rangle + \tfrac{\mu_{\Phi}}{2}\|y-x\|^2, \quad \forall x,y \in \R^N,
\end{equation}
where $s\in\partial \Phi(x)$.
If $\mu_\Phi = 0$. then function $\Phi$ is said to be convex.

\paragraph{Convex conjugate and proximal mapping.}
The proximal mapping of a convex function $\Psi$ at the point $x \in \R^N$ is defined as follows:
\begin{equation}
\label{Def_prox}
  \prox_\Psi (x) \eqdef \arg \min_{y\in \R^N} \Psi(y) + \tfrac12 \|y-x\|^2.
\end{equation}
Furthermore, for a convex function $\Phi: \R^N \to \R$, it's convex conjugate $\Phi^*$ is defined as
$ \Phi^*(y) \eqdef \sup_{u\in\mathbb{R}^N} \langle u,y \rangle  - \Phi(u).$
   From Chapter $1$ of \cite{Rockafellar06}, we see that $\prox_{\Psi}(\cdot)$, and it's convex conjugate $\prox_{\Psi^*}(\cdot)$, are nonexpansive:
  \begin{equation}
  \label{Eq_proxnonexpansive}
  \|\prox_{\Psi}(y) - \prox_{\Psi}(x)\| \le \|y -x\|, \quad  \mbox{and} \quad \|\prox_{\Psi^*}(y) - \prox_{\Psi^*}(x)\| \le \|y -x\|.
  \end{equation}

\paragraph{Coordinate decomposition of $\R^N$.}
\label{S_Block_structure}
Let $U \in \mathbb{R}^{N \times  N}$ be a column permutation of the $N \times N$ identity matrix and further let $U = [U_1,U_2,\dots,U_N]$ be a decomposition of $U$ into $N$ submatrices (column vectors), where $U_i \in \R^{N \times 1}$. It is clear that any vector $x \in \mathbb{R}^N$ can be written uniquely as
$x = \sum_{i=1}^N U_ix^i,$ where $x^i \in \R$ denotes the $i$th coordinate of $x$.

Define $[N]\eqdef \{1,\dots,N\}$. Then we let $S\subseteq [N]$ and $U_S \in \mathbb{R}^{N\times |S|}$ be the collection of columns from matrix $U$ that have column indices in the set $S$.
We denote the vector
\begin{equation}\label{U_S}x^S \eqdef \sum_{i\in S} U_i^Tx = U_S^T x.\end{equation}

\paragraph{Coordinate decomposition of $\Psi$.}\label{S_Psi}
The function $\Psi:\R^N \to \R$ is assumed to be coordinate separable. That is, we assume that $\Psi(x)$ can be decomposed as:
\begin{equation}\label{S2_separable_psi}
\Psi(x) = \sum_{i=1}^N \Psi_i (x^i),
\end{equation}
where the functions $\Psi_i: \R \to \R$ are convex. We denote by $\Psi_S: \R^{|S|} \to \R$, where $S\subseteq [N]$, the function
\begin{equation}\label{PsiS}
  \Psi_S(x^S) = \sum_{i\in S} \Psi_i (x^i),
\end{equation}
where $x^S \in \R^{|S|}$ is the collection of components from $x$ that have indices in set $S$.
The following relationship will be used repeatedly in this work. For $x \in \R^N$, index set $S \subseteq [N]$, and $t, \hat t\in \R^{|S|}$, it holds that:
\begin{eqnarray}\label{Eq_PsivsPsii}
  \Psi(x+U_St^S) - \Psi(x + U_S \hat t^S) \overset{\eqref{S2_separable_psi}+\eqref{PsiS}}{=}  \Psi_S(x^S+t^S)-\Psi_S(x^S + \hat t^S).
\end{eqnarray}
Clearly, when $\hat t^S\equiv 0$, we have the special case $\Psi(x+U_St^S) - \Psi(x) =  \Psi_S(x^S+t^S)-\Psi_S(x^S)$.

\paragraph{Subset Lipschitz continuity of $f$.}

Throughout the paper we assume that the gradient of $f$ is Lipschitz, uniformly in $x$, for all subsets $S\subseteq [N]$. This means that, for all $x \in \R^N$, $S\subseteq [N]$ and $t^S \in \R^{|S|}$ we have
\begin{equation}
\label{S2_Lipschitz}
     \| \nabla_S f(x + U_St^S) - \nabla_S f(x) \| \leq \Lip_S \|t^S\|,
\end{equation}
where $ \nabla_S f(x)  \overset{\eqref{U_S}}{=} U_S^T\nabla f(x)$. An important consequence of \eqref{S2_Lipschitz} is the following standard inequality \cite[p.57]{Nesterov04}:
\begin{equation}
\label{S2_upperbound}
     f(x+ U_St^S) \leq f(x) + \langle \nabla_S f(x), t^S \rangle+ \tfrac{\Lip_S}{2}\|t^S\|^2.
\end{equation}

\paragraph{Radius of the Levelset.}
Let $X_*$ denote the set of optimal solutions of \eqref{Def_F}, and let $x_*$ be any member of that set. We define
\begin{equation}\label{Def_Rlevelset}
  \mathcal{R}(x) = \max_y \max_{x_* \in X_*} \{\|y-x_*\| \; : \: F(y) \leq F(x)\},
\end{equation}
which is a measure of the size of the level set of $F$ given by $x$. In this work we make the standard assumption that $\mathcal{R}(x_0)$ is bounded at the initial iterate $x_0$.
We also define a scaled version of \eqref{Def_Rlevelset}
\begin{equation}\label{Def_Rlevelset_2}
  \mathcal{R}_w(x) = \max_y \max_{x_* \in X_*} \{\|y-x_*\|_w \; : \: F(y) \leq F(x)\},
\end{equation}
where $w\in\R_{++}^N$ and $\|u\|_w = \left(\sum_{i=1}^N w_iu_i^2\right)^{1/2}$.

\subsection{Piecewise Quadratic Model}

For fixed $x \in\R^N$, we define a piecewise quadratic approximation of $F$ around the point $(x + t)\,\in \R^N$ as follows:
$F(x + t)\approx f(x) + Q(x;t),$
where
\begin{equation}
\label{Def_Q}
Q(x;t) \eqdef  \langle \nabla f(x), t \rangle + \tfrac12 \|t\|_{H(x)}^2 + \Psi(x + t)
\end{equation}
and $H(x)$ is \emph{any} symmetric positive definite matrix.
We also define
\begin{equation}\label{Def_Qi}
  Q_S(x,t^S) \eqdef \langle \nabla_S f(x), t^S \rangle + \tfrac12 \|t^S\|_{H^S(x)}^2 + \Psi_S(x^S + t^S),
\end{equation}
and $H^S(x) \in \R^{|S| \times |S|}$ is a square submatrix (principal minor) of $H(x)$ that corresponds to the selection of columns
and rows from $H(x)$ with column and row indices in set $S$. Notice that $Q_S(x,t^S)$ is the quadratic model for the collection of coordinates in $S$.

Similarly to \eqref{Eq_PsivsPsii}, we have the following important relationship. For $x \in \R^N$, index set $S \subseteq [N]$, and $t, \hat t\in \R^{|S|}$, it holds that:
\begin{eqnarray}\label{Eq_QvsQS}
\notag
  Q(x;U_St^S) - Q(x;U_S \hat t^S) &\overset{\eqref{Def_Q}}{=}& \langle \nabla f(x), U_St^S \rangle + \tfrac12 \|U_St^S\|_{H(x)}^2 + \Psi(x + U_St^S)\\
  \notag
  &&- \langle \nabla f(x), U_S\hat t^S \rangle - \tfrac12 \|U_S\hat t^S\|_{H(x)}^2 - \Psi(x + U_S \hat t^S)\\
  \notag
  &\overset{\eqref{Eq_PsivsPsii}}{=}& \langle \nabla_S f(x), t^S\rangle -\langle \nabla_S f(x),\hat t^S \rangle + \tfrac12 \|t^S\|_{H^S(x)}^2 - \tfrac12 \|\hat t^S\|_{H^S(x)}^2\\
  \notag
  &&+ \Psi_S(x^S + t^S)  - \Psi_S(x^S + \hat t^S)\\
  &\overset{\eqref{Def_Qi}}{=}& Q_S(x,t^S) - Q_S(x,\hat t^S)
\end{eqnarray}

\subsection{Stationarity conditions.}
The following theorem gives the equivalence of some stationarity conditions of problem \eqref{Def_F}.
\begin{theorem}
\label{thm:optimality_conditions}
The following are equivalent first order optimality conditions for problem \eqref{Def_F}, Section $2$ in \cite{PB13}.
\begin{enumerate}
\item[(i)] $\nabla f(x) + s=0$ and $s\in \partial \Psi(x)$,
\item[(ii)] $-\nabla f(x)\in \partial \Psi(x)$,
\item[(iii)] $\nabla f(x) + \prox_{\Psi^*}\left(x -  \nabla f(x) \right) = 0$,
\item[(iv)] $x = \prox_{\Psi}\left(x - \nabla f(x) \right)$.
\end{enumerate}
\end{theorem}
Let us define the continuous function
\begin{equation}\label{defeq:1}
g(x;t) \eqdef \nabla f(x) + H(x) t + \prox_{\Psi^*}\big(x + t -  \nabla f(x) - H(x) t \big).
\end{equation}
By Theorem \ref{thm:optimality_conditions}, the points that satisfy $g(x;0) = \nabla f(x) + \prox_{\Psi^*}\left(x -  \nabla f(x)\right)=0$ are stationary points for problem \eqref{Def_F}. Hence, $g(x;0)$ is a continuous measure of the distance from the set of stationary points of problem \eqref{Def_F}.
Furthermore, we define
\begin{equation}
\label{defeq:2}
g_S(x;t^S) \eqdef \nabla_S f(x) + H^S (x)t^S + \prox_{\Psi^*_S}\big(x^S + t^S - \nabla_S f(x) - H^S(x)t^S \big).
\end{equation}

\section{The Algorithm}
\label{S_Algorithm}
In this section we  present the Flexible Coordinate Descent (FCD) algorithm for solving problems of the form \eqref{Def_F}. There are three key steps in the algorithm: (Step $4$) the coordinates are sampled randomly; (Step $5$) the model \eqref{Def_Qi} is solved approximately until rigorously defined stopping conditions are satisfied; (Step $6$) a line search is performed to find a step size that ensures a sufficient reduction in the objective value. Once these key steps have been performed, the current point $x_k$ is updated to give a new point $x_{k+1}$, and the process is repeated.

In the algorithm description, and later in the paper, we will use the following definition for the set $\Omega$:
\begin{equation}\label{Omega}
  \Omega \eqdef \partial Q_S(x_k;t_k^S).
\end{equation}
We are now ready to present pseudocode for the FCD algorithm, while a thorough description of each of the key steps in the algorithm will follow in the rest of this section.
\begin{algorithm}[H]
\begin{algorithmic}[1]
\vspace{0.1cm}
\STATE \textbf{Input} Choose $x_0\in \R^N$, $\theta \in (0,1)$ and $1\leq \tau \leq N$.
\FOR{$k=1,2,\dots$} \vspace{2mm}
\STATE Sample a subset of coordinates $S\subseteq [N]$, of size $|S| = \tau$, with probability $\mathbf{P}(S) >0$\\ \vspace{2mm}
\STATE If $g_S(x_k;0)=0$ go to Step $3$; else select $\eta_k^S \in [0,1)$ and approximately solve
\begin{equation}\label{eq_subproblem}
t_k^S \approx \argmin_{t^S} Q_S(x_k;t^S),
\end{equation}
until the following stopping conditions hold:
\begin{equation}\label{Def_stoppingconditions_1}
Q(x_k;U_St^S_k) < Q(x_k;0),
\end{equation}
and
\begin{equation}\label{Def_stoppingconditions_2}
\inf_{u\in \Omega}\|u - g_S(x_k;t^S_k)\|^2 \le (\eta_k^S \|g_S(x_k;0)\|)^2 -\|g_S(x_k;t^S_k)\|^2.
\end{equation}
\STATE Perform a backtracking line search along the direction $t_k^S$ starting from $\alpha=1$. i.e., find $\alpha\in(0,1]$ such that
\begin{equation}\label{Def_linesearch}
F(x_k) - F(x_k+\alpha U_St_k^S) \ge \theta \left(\ell(x_k;0) - \ell(x_k;\alpha U_St_k^S)\right),
\end{equation}
where
\begin{equation}\label{Def_lossfunctionli}
\ell(x_k; t) \eqdef f(x_k) + \langle \nabla f(x_k), t \rangle + \Psi(x_k + t).
\end{equation}
\STATE Update $x_{k+1} = x_k + \alpha U_S t_k^S$
\ENDFOR
\end{algorithmic}
\caption{Flexible Coordinate Descent (FCD)}
\label{FCD}
\end{algorithm}

\subsection{Selection of coordinates (Step $\boldsymbol 3$)} \label{S_blockstructure}

One of the central ideas of this algorithm is that the selection of coordinates to be updated at each iteration is chosen \emph{randomly}. This allows the coordinates to be selected very quickly and cheaply. For all the results in this work, we use a `$\tau$-nice sampling' \cite{Richtarik15}. This sampling is described briefly below, and a full description can be found in \cite{Richtarik15}.

Let $2^{[N]}$ denote the set of all subsets of $[N]$, which includes the empty set $\emptyset$ and $[N]$ itself.
At every iteration of FCD, a subset is sampled from $2^{[N]}$. The probability of a set $S\in2^{[N]}$ being selected is denoted by $\Prob(S)$, and we are interested in probability distributions with $\Prob(S)>0$ $\forall S\in 2^{[N]}\backslash \emptyset$ and $\Prob(\emptyset) = 0$.
i.e., all non-empty sets have positive probability of being selected.

In what follows we describe a (specific instance of a) uniform probability distribution (called a $\tau$-nice sampling). In particular, the uniform probability distribution is constructed in such a way that subsets in $2^{[N]}$ with the same cardinality have equal probability of being selected, i.e., $\Prob(S) = \Prob(S^{'})>0$ $\forall S,S^{'}\in 2^{[N]}$ with $|S|=|S^{'}|$.
\begin{definition}[$\tau$-nice sampling]\label{Def_Sampling}
  Given an integer $1\leq \tau \le N$, a set $S$ with $|S|=\tau$ is selected uniformly
  with probability one, and the sampling is uniquely characterised by its density function:
\begin{equation}
\label{prob_dist_uniform}
\Prob(S) = \frac{1}{\binom{N}{\tau}} \ \forall S\in2^{[N]} \mbox{ with } |S|=\tau,
\quad
 \Prob(S) = 0 \ \forall S\in2^{[N]} \mbox{ with } |S|\neq\tau.
 \end{equation}
\end{definition}
\begin{assumption}\label{Assume_Sampling}
  In Step 3 of FCD (Algorithm~\ref{FCD}), we assume that the sampling procedure used to generate the subset of coordinates $S$ in that given in Definition~\ref{Def_Sampling}.
\end{assumption}
Below is a technical result (see \cite{Richtarik15}) that uses the probability distribution \eqref{prob_dist_uniform} and is used in the worst-case iteration complexity results of this paper. Let $\theta_i$ with $i=1,2,\dots,N$ be some constants. Then
\begin{equation}
 \label{eq:tech_prob}
 \E\Big[\sum_{i\in S} \theta_i\Big] = \frac{\tau}{N} \sum_{i=1}^N \theta_i.
 \end{equation}

\subsection{The search direction and Hessian approximation (Step $\boldsymbol 4$)}\label{SS_searchdirectionHessian}

The search direction (Step~4 of FCD) is determined as follows. Given a set of coordinates $S$, FCD forms a model for $S$ \eqref{Def_Qi}, and minimizes the model approximately until the stopping conditions \eqref{Def_stoppingconditions_1} and \eqref{Def_stoppingconditions_2} are satisfied, giving an `inexact' search direction. We also describe the importance of the choice of the matrix $H$, which defines the model and is an approximate second order information term. Henceforth, we will use the shorthand $H_k^S \equiv H^S(x_k)$.

\subsubsection{The search direction}

The subproblem \eqref{eq_subproblem}, (where $Q_S(x_k;t^S)$ is defined in \eqref{Def_Qi}) is approximately solved, and the search direction $t_k^S$ is accepted when the stopping conditions \eqref{Def_stoppingconditions_1} and \eqref{Def_stoppingconditions_2} are satisfied, for some $\eta_k^S \in [0,1)$. Notice that
\begin{align}\label{eq:1}
\notag
  Q(x;U_S t^S) - Q(x;0) &\overset{\eqref{Eq_QvsQS}}{=} Q_S(x;t^S) - Q_S(x;0)\\
  &\overset{\eqref{Def_Qi}}{=} \langle \nabla_S f(x), t^S \rangle + \tfrac12 \|t^S\|_{H^S}^2 + \Psi_S(x^S + t^S) - \Psi_S(x^S).
\end{align}
Hence, from \eqref{eq:1}, the stopping conditions \eqref{Def_stoppingconditions_1} and \eqref{Def_stoppingconditions_2} depend on subset $S$ only, and are therefore inexpensive to verify, meaning that they are \emph{implementable}.
\begin{remark}\label{Remark_tineq0}\quad
\begin{itemize}
  \item[(i)] At some iteration $k$, it is possible that $g_S(x_k;0)=0$. In this case, it is easy to verify that the optimal solution of subproblem \eqref{eq_subproblem} is $t^S_k=0$. Therefore, before calculating $t^S_k$ we check first if condition $g_S(x_k;0)=0$ is satisfied.
  \item[(ii)] Notice that, unless at optimality (i.e., $g(x_k;0)=0$), there will always exist a subset $S$ such that $g_S(x_k;0)\neq 0$, which implies that $t^S_k \neq 0$. Hence, FCD will not stagnate.
\end{itemize}
\end{remark}

\subsubsection{The Hessian approximation}
\label{S_HessianApprox}
One of the most important features of this method is that the quadratic model \eqref{Def_Qi} incorporates second order information in the form of a symmetric positive definite matrix $H_k^S$. This is key because, depending upon the choice of $H_k^S$, it makes the method robust. Moreover, at each iteration, the user has complete freedom over the choice of $H_k^S\succ0$.

We now provide a few suggestions for the choice of $H_k^S$. (This list is not intended to be exhaustive.) Notice that in each case there is a trade off between a matrix that is inexpensive to work with, and one that is a more accurate representation of the true block Hessian.
\begin{enumerate}
\item \textbf{Identity.} Here, $H_k^S = I$, so that \emph{no second order information is employed by the method.}
\item \textbf{Multiple of the identity.} One could set $H_k^S = \nu I$, for some $\nu>0$. In particular, a popular choice is $\nu = L_S$, where $L_S$ is the Lipschitz constant defined in \eqref{S2_Lipschitz}.
  \item \textbf{Diagonal matrix.} When $H_k^S$ is a diagonal matrix, $H_k^S$ and it's inverse are inexpensive to work with. In particular, $H_k^S = \text{diag}(\nabla_S^2 f(x_k))$ captures partial curvature information, and is an effective choice in practice, particularly when $\text{diag}(\nabla^2 f(x))$ is a good approximation to $\nabla^2 f(x)$. Moreover, if $f$ is quadratic, then $\nabla^2 f(x_k)$ is constant for all $k$, so $\text{diag}(\nabla^2 f(x))$ can be computed and stored at the start of the algorithm and elements can be accessed throughout the algorithm as necessary.
  \item \textbf{Principal minor of the Hessian.} When $H_k^S = \nabla_S^2 f(x_k)$, second order information is incorporated into FCD, which can be very beneficial for ill-conditioned problems. However, this choice of $H_k^S$ is usually more computationally expensive to work with.
  In practice, $\nabla_S^2 f(x_k)$ may be used in a matrix-free way and not explicitly stored. For example, there may be an analytic formula for performing matrix-vector products with $\nabla_S^2 f(x_k)$, or techniques from automatic differentiation could be employed, see \cite[Section $7$]{IEEEhowto:wrightbook2}.
  \item \textbf{Quasi-Newton approximation.} One may choose $H_k^S$ to be an approximation to $\nabla_S^2 f(x_k)$ based on a limited-memory BFGS update scheme, see \cite[Section $8$]{IEEEhowto:wrightbook2}.
  This approach might be suitable in cases where the problem is ill-conditioned, but performing matrix-vector products with $\nabla_S^2 f(x_k)$ is expensive.
\end{enumerate}

\begin{remark}We make the following remarks regarding the choice of $H_k^S$.
\begin{enumerate}
  \item If any of the matrices mentioned above are not positive definite, then they can be adjusted to make them so. For example, if $H_k^S$ is diagonal, any zero that appears on the diagonal can be replaced with a positive number. Moreover, if $\nabla_S^2 f(x_k)$ is not positive definite, a multiple of the identity can be added to it.
  \item If $H_k^S$ is chosen to be a diagonal matrix, then the quadratic model \eqref{Def_Qi} is \emph{separable}, so the subproblems for each coordinate in $S$ are independent. Moreover, in some cases this gives rise to a \emph{closed form solution} for the search direction $t_k^S$. (For example, if $\Psi(x) \equiv \|x\|_1$, then soft thresholding may be used.)
\end{enumerate}
\end{remark}

An advantage of the FCD algorithm (if Option 4 is used for $H_k^S$) is that \emph{all elements of the Hessian can be accessed.} This is because the set of coordinates can change at every iteration, and so too can matrix $H_k^S$. This makes FCD extremely \emph{flexible} and is particularly advantageous when there are large off diagonal elements in the Hessian. The importance of incorporation of information from off-diagonal elements is demonstrated in numerical experiments in Section \ref{S_Numerical}.

\subsubsection{Termination criteria for model minimization}

In Step 4 of FCD, the termination criteria \eqref{Def_stoppingconditions_1} and \eqref{Def_stoppingconditions_2} are used to determine whether an acceptable search direction has been found. For composite functions of the form \eqref{Def_F}, \emph{both} conditions are required to ensure that $t_S^k$ is a descent direction. Moreover, \eqref{Def_stoppingconditions_1} is important for intuitively obvious reasons: the model should be decreased at each iteration. We will now attempt to explain/justify the use of the second termination condition \eqref{Def_stoppingconditions_2}.

The parameter $\eta_k^S$ plays a key role; it determines how `inexactly' the quadratic model \eqref{eq_subproblem} may be solved (or equivalently, how `inexact' the search direction $t_k^S$ is allowed to be). If one sets $\eta_k^S = 0$, then the model \eqref{eq_subproblem} is minimized exactly, leading to exact search directions.
On the other hand, if $\eta_k^S > 0$ then the model is approximately minimized, leading to inexact search directions.

To obtain iteration complexity guarantees for FCD, the termination conditions must be chosen carefully. In particular, it becomes obvious in Lemma~\ref{lem:bound_Q}, that \eqref{Def_stoppingconditions_1} and \eqref{Def_stoppingconditions_2} are the appropriate conditions. We now proceed to show that the conditions are mathematically sound.
By rearranging \eqref{Def_stoppingconditions_2} we obtain
\begin{equation}\label{Def_stoppingconditions_2_rewrite}
\inf_{u\in \Omega}\|u - g_S(x_k;t^S_k)\|^2+\|g_S(x_k;t^S_k)\|^2 \le (\eta_k^S \|g_S(x_k;0)\|)^2.
\end{equation}
The left hand side in \eqref{Def_stoppingconditions_2_rewrite} is a continuous function because the `$\inf$' operator is an orthonormal projection
onto a closed subspace $\Omega:=\nabla_S f(x_k) + \partial \Psi_S(x_k^S  + t_k^S)$, see \eqref{Omega}, and $\|g_S(x_k;t^S_k)\|^2$ is continuous as a function of $t^S_k$. Furthermore, there exists
a point $t^S_k$ such that \eqref{Def_stoppingconditions_2_rewrite} is satisfied, and this is the minimizer of \eqref{eq_subproblem}. In particular,
if $t_k^S$ is the minimizer of \eqref{eq_subproblem}, then from Theorem \ref{thm:optimality_conditions} we have that $g_S(x_k;t^S_k)=0\in \Omega$.
Hence, the left hand side in \eqref{Def_stoppingconditions_2_rewrite} is zero and the condition is satisfied for any $\eta^k\in[0,1)$ and $g_S(x_k;0)$. Since the right hand side in \eqref{Def_stoppingconditions_2_rewrite} is a non-negative constant,
the left hand side
is continuous and the minimizer of \eqref{eq_subproblem} satisfies \eqref{Def_stoppingconditions_2_rewrite}, then there exists a closed ball centered at the minimizer such that within this ball \eqref{Def_stoppingconditions_2_rewrite}
is satisfied. Therefore, any convergent method which can be used to minimize \eqref{eq_subproblem} will eventually satisfy the termination condition \eqref{Def_stoppingconditions_2_rewrite} without solving \eqref{eq_subproblem} exactly,
unless the right hand side in \eqref{Def_stoppingconditions_2_rewrite} is zero.

We stress that this projection operation is inexpensive for the applications that
we consider, e.g., when $\Psi$ is the $\ell_1$-norm or the $\ell_2$-norm squared, or elastic net regularization, which is a combinations of these two.
In particular, the subdifferential $\Omega$ has an `analytic' form, which also allows for a closed form solution of the left hand side of the termination criterion. Below we provide an example for the case where $\Psi=\tau\|x\|_1$.
Using the definition of $\Omega$ we have that the left hand side problem in \eqref{Def_stoppingconditions_2_rewrite} $\inf_{u\in \Omega}\|u - g_S(x_k;t^S_k)\|^2$
is equivalent to
$$
\inf_{u - \nabla_S f(x_k) - H^S (x_k)t^S_k\in \partial \Psi_S^*(x^S_k + t^S_k)}\|u - g_S(x_k;t^S_k)\|^2.
$$
By making a change of coordinates from $u$ to $v$ as $u:=\nabla_S f(x_k) + H^S (x_k)t^S_k + v$, we rewrite the previous problem as
$$
\inf_{v\in \partial \Psi_S^*(x^S_k + t^S_k)}\|v + \nabla_S f(x_k) + H^S (x_k)t^S_k - g_S(x_k;t^S_k)\|^2.
$$
Using the definition of $g_S(x_k;t^S_k)$ from \eqref{defeq:2} we have that the previous problem is equivalent to
\begin{equation}\label{eq:sol_term_cond}
\inf_{v\in \partial \Psi_S^*(x^S_k + t^S_k)}\|v -\prox_{\Psi^*_S}\big(x^S_k + t^S_k - \nabla_S f(x_k) - H^S(x_k)t^S_k \big) \|^2.
\end{equation}
The subdifferential $\partial \Psi_S^*(x^S_k + t^S_k)$ has the following coordinate-wise form
\begin{eqnarray*}
\partial \Psi_i^*(x^i_k + t^i_k)
     =
  \begin{cases}
     \tau & \text{if}\;\; x^i_k + t^i_k > 0,\\
     [-\tau,\tau] & \text{if}\;\; x^i_k + t^i_k = 0,\\
     -\tau & \text{if}\;\; x^i_k + t^i_k < 0.
  \end{cases}
  \quad \forall i\in S
\end{eqnarray*}
Using the above analytic form of the subdifferential it is easy to see that problem \eqref{eq:sol_term_cond} has a closed form solution that is inexpensive to compute.

\begin{remark}
We stress that if the regularizer is the $\ell_1$-norm, or the squared $\ell_2$-norm, which are two very popular regularizers, our inexactness criterion \eqref{Def_stoppingconditions_2_rewrite} is inexpensive to implement in practice. However, for other general regularizers, this may not be the case, and it may not be possible to compute \eqref{Def_stoppingconditions_2_rewrite} analytically. In such a case, it may be necessary to compute the minimizer of the quadratic model exactly, or to use some other iterative method, both of which come with an associated computational cost.
\end{remark}

\subsection{The line search (Step $\boldsymbol 5$)} \label{subsec:pract_line}

The stopping conditions \eqref{Def_stoppingconditions_1} and \eqref{Def_stoppingconditions_2} ensure that $t_k^S$ is a descent direction, but if the full step $x_k + U_St_k^S$ is taken, a reduction in the function value \eqref{Def_F}  is not guaranteed. To this end, we include a line search step in our algorithm in order to guarantee a monotonic decrease in the function $F$. Essentially, the line search guarantees the sufficient decrease of $F$ at every iteration, where sufficient decrease is measured by the loss function \eqref{Def_lossfunctionli}.

In particular, for fixed $\theta \in (0,1)$, we require that for some $\alpha \in (0,1]$, \eqref{Def_linesearch} is satisfied.
(In Lemma \ref{lem:Fdec} we prove the existence of such an $\alpha$.)
Notice that
\begin{eqnarray}
\notag
  \ell(x;U_St^S)-\ell(x;0) &\overset{\eqref{Def_lossfunctionli}}{=}& \langle \nabla_S f(x),t^S \rangle + \Psi(x + U_St^S) - \Psi(x)\\
  \label{Eq_differenceloss}
  &\overset{\eqref{Eq_PsivsPsii}}{=}& \langle \nabla_S f(x),t^S\rangle  + \Psi_S(x^S + t^S) - \Psi_S(x^S),
\end{eqnarray}
which shows that the calculation of the right hand side of \eqref{Def_linesearch} only depends upon block $S$, so it is inexpensive. Moreover, the line search condition (Step 5) involves the difference between function values $F(x_k) - F(x_k + \alpha U_S t_k^S)$. Fortunately, while function values can be expensive to compute, computing the difference between function values at successive iterates need not be. (See Section~\ref{SS_concrete_examples}, and the numerical results in Section~\ref{S_Numerical} and Table~\ref{LSsmall}.)

\subsection{Concrete examples} \label{SS_concrete_examples}
Now we give two examples to show that the difference between function values at successive iterates can be computed efficiently, which demonstrates that the line-search is implementable.
The first example considers a convex composite function, where the smooth term is a quadratic loss term and the second example considers a logistic regression problem.

\paragraph{Quadratic loss plus regularization example.}
Suppose that
$f(x)  = \tfrac12\|Ax-b\|^2$ and $\Psi(x)$ is not equivalent to zero,
where $A\in\mathbb{R}^{m\times N}$, $b\in\mathbb{R}^m$ and $x \in \R^N$.
Then
\begin{eqnarray}
\label{Eq_Fvaldiff}
  F(x_k) - F(x_k+\alpha U_S t_k^S) &\overset{\eqref{Eq_PsivsPsii}}{=}& f(x_k) + \Psi_S(x^S)  - f(x_k+\alpha U_S t_k^S)  - \Psi_S(x_k^S+\alpha t_k^S) \\ \notag
  &=& \Psi_S(x^S)- \alpha \langle \nabla_S f(x), t_k^S \rangle - \tfrac{\alpha^2}2\|A_it_k^S\|_2^2  - \Psi_S(x_k^S+\alpha t_k^S).
\end{eqnarray}
The calculation $F(x_k) - F(x_k+\alpha U_S t_k^S)$, as a function of $\alpha$, only depends on subset $S$. Hence, it is inexpensive.
Moreover, often the quantities in \eqref{Eq_Fvaldiff} are already needed in the computation of the search direction $t$, so regarding the line search step, they essentially come ``for free''.

\paragraph{Logistic regression example.}
Suppose that
$f(x) \equiv  \sum_{j=1}^m \log(1 + e^{-b_j a_j^T x})$ and $\Psi(x)$ is not equivalent to zero,
where $a_j^T$ is the $j$th row of a matrix $A\in\mathbb{R}^{m\times N}$
and $b_j$ is the $j$th component of vector $b\in\mathbb{R}^m$. As before, we need to evaluate \eqref{Eq_Fvaldiff}. Let us split calculation of $F(x_k) - F(x_k+\alpha U_S t_k^S) $ in parts. The first part $\Psi_S(x^S)- \Psi_S(x_k^S+\alpha t_k^S) $ is inexpensive, since it depends only on subset $S$. The second part $f(x_k) - f(x_k+\alpha U_S t_k^S)$ is more expensive because
is depends upon the logarithm.
In this case, one can calculate $f(x_0)$ \textit{once} at the beginning of the algorithm and then update $f(x_k+\alpha U_S t_k^S)$ $\forall k\ge1$ less expensively.
In particular, let us assume that the following terms:
\begin{equation}
\label{inner_prod_log}
e^{-b_j a_j^Tx_0 } \quad \forall j \quad \mbox{and} \quad f(x_0)=\sum_{j=1}^m \log(1 + e^{-b_j a_j^Tx_0}),
\end{equation}
 are calculated once and stored in memory.
Then, at iteration $1$,  the calculation of
$f(x_0 + \alpha U_S t^S_0) = \sum_{j=1}^m \log(1 + e^{-b_j a_j^Tx_0}e^{-\alpha b_j a_j^T(U_S t^S_0)})$
is required for different values of $\alpha$ by the backtracking line search algorithm.
The most demanding task in calculating $f(x_0 + \alpha U_S t^S_0)$ is the calculation of the products $b_j a_j^T(U_S t^S_0)$ $\forall j$ \textit{once}, which is inexpensive since $\forall j$ this operation
depends only on subset $S$. Having $b_j a_j^T(U_S t^S_0)$ $\forall j$ and \eqref{inner_prod_log} calculation of $f(x_0) - f(x_0 + \alpha U_S t^S_0)$ for different values of $\alpha$ is inexpensive.
At the end of the process, $f(x_1)$ and $e^{-b_j a_j^T x_1}$ $\forall j $ are given for free, and the process can be repeated for the calculation of $f(x_1) - f(x_1 + \alpha U_S t^S_1)$ etc.

\begin{remark}
  The examples above show that, for quadratic loss functions, and logistic regression problems, computing function values is relatively inexpensive. Thus, for problems with this, or similiar structure, FCD is competitive, even though it requires function evaluations for the line search. (The competitiveness of FCD is confirmed in our numerical experiments in Section~\ref{S_Numerical}.) However, we stress that, for other applications with more general objective functions, it may not be possible to perform function evaluations efficiently, in which case FCD \emph{may not} be a suitable algorithm, and a different algorithm could be used.
\end{remark}


\section{Bounded step-size and monotonic decrease of $F$} \label{S_GlobalConvergence}
Throughout this section we denote $H^S_k \equiv H^S(x_k)$. The following assumptions are made about $H_k^S$ and $f$.
Assumption \ref{Assump_HisPD} explains that the Hessian approximation $H_k^S$ must be positive definite for all subsets of coordinates $S$ at all iterations $k$.
Assumption \ref{Assump_fisLipschitz} explains that the gradient of $f$ must be Lipschitz for all subsets $S$.

\begin{assumption}\label{Assump_HisPD}
  There exist $0 < \lambda_S \leq \Lambda_S$, such that the sequence of symmetric $\{H_k^S\}_{k\geq 0}$ satisfies:
  \begin{equation}\label{Assumption_lambdai}
    0 < \lambda_S \leq \lambda_{\min}(H_k^S) \quad \text{and} \quad \lambda_{\max}(H_k^S) \leq \Lambda_S, \quad \text{for all } S\subseteq [N].
  \end{equation}
\end{assumption}
\begin{assumption}\label{Assump_fisLipschitz}
  The function $f$ is smooth and satisfies \eqref{S2_Lipschitz} for all $S\subseteq [N]$.
\end{assumption}
The next assumption regards the relationship between the parameters introduced in Assumptions~\ref{Assump_HisPD} and \ref{Assump_fisLipschitz}.
\begin{assumption}\label{Assump_LambdaLrelation}
  The relation $\lambda_S \leq L_S$ holds for all $S \subseteq[N]$.
\end{assumption}
The following lemma shows that if $t^S_k$ is nonzero, then $F$ is decreased. The proof closely follows that of \cite[Theorem $3.1$]{sqa}.
\begin{lemma}
\label{lem:Fdec}
  Let Assumptions \ref{Assume_Sampling}, \ref{Assump_HisPD}, \ref{Assump_fisLipschitz} and \ref{Assump_LambdaLrelation} hold, and let $\theta\in(0,1)$. Given $x_k$, let $S$ and $t^{S}_k \neq 0$ be generated by FCD (Algorithm~\ref{FCD}) with $\eta_k^S\in[0,1)$. Then Step 6 of FCD will accept a step-size $\alpha$ that satisfies
  \begin{equation}\label{Eq_alphamin}
    \alpha \geq \alpha_S, \quad \text{where} \quad \alpha_S\eqdef(1-\theta)\tfrac{\lambda_S}{2L_S}>0.
  \end{equation}
  Furthermore,
  \begin{equation}\label{Eq_F_ubont}
    F(x_k) - F(x_k + \alpha U_St^{S}_k)  > \theta(1-\theta)\tfrac{\lambda_S^2}{4L_S}\|t^S_k\|^2.
  \end{equation}
\end{lemma}
\begin{proof}
From \eqref{Def_stoppingconditions_1}, we have
  $
    0 > Q(x_k ; U_St_k^S) - Q(x_k;0) \overset{\eqref{Def_Q}+\eqref{Def_lossfunctionli}}{=} \ell(x_k;U_St_k^S) - \ell(x_k;0) + \tfrac12\|t_k^S\|_{H_k^S}^2.
  $
  Rearranging gives
  \begin{equation}\label{in:1}
    \ell(x_k;0) - \ell(x_k; U_St_k^S) > \tfrac12\|t_k^S\|_{H_k^S}^2 \overset{\eqref{Assumption_lambdai}}{\geq} \tfrac12\lambda_S\|t_k^S\|^2.
  \end{equation}
  Denote $x_{k+1} = x_k + \alpha U_St_k^S$. By Assumption \ref{Assump_fisLipschitz}, for all $\alpha \in (0,1]$, we have
  \begin{equation*}
    F(x_{k+1}) \leq f(x_k) + \alpha \langle \nabla_S f(x_k), t_k^S \rangle + \tfrac{\Lip_S}{2}\alpha^2 \|t_k^S\|^2 + \Psi(x_k+\alpha U_St_k^S).
  \end{equation*}
  Adding $\Psi(x_k)$ to both sides of the above and rearranging gives
\begin{eqnarray}
  \label{in:3}
  \notag
    F(x_k) - F(x_{k+1}) &\ge& -\alpha \langle \nabla_S f(x_k), t_k^S \rangle - \tfrac{\Lip_S}{2}\alpha^2 \|t_k^S\|^2  - \Psi(x_k+\alpha U_St_k^S) + \Psi(x_k) \\
    &\overset{\eqref{Eq_differenceloss}}{=}& \ell(x_k;0) - \ell(x_k ; \alpha U_St_k^S) - \tfrac{\Lip_S}{2} \alpha^2 \|t_k^S\|^2.
  \end{eqnarray}
 By convexity of $\Psi(x)$ we have that
\begin{equation}\label{Eq_ell_convexity}
  \ell(x_k;0) - \ell(x_k;\alpha U_St_k^S)\ge \alpha(\ell(x_k;0) - \ell(x_k;U_St_k^S)).
\end{equation}
  Then
  \begin{eqnarray}\label{in:10}
    F(x_k) - F(x_{k+1}) &-& \theta (\ell(x_k;0) - \ell(x_k ; \alpha U_St_k^S)) \nonumber \\
    &\overset{\eqref{in:3}}{\geq}& (1-\theta) \big(\ell(x_k;0) - \ell(x_k ; \alpha U_St^S)\big) - \tfrac{\Lip_S}{2}\alpha^2 \|t_k^S\|^2  \nonumber \\
    &\overset{\eqref{Eq_ell_convexity}}{\geq}& \alpha(1-\theta) \big(\ell(x_k;0) - \ell(x_k ;  U_St_k^S)\big) - \tfrac{\Lip_S}{2}\alpha^2 \|t_k^S\|^2  \nonumber \\
    & \stackrel{\eqref{in:1}}{>}&  \tfrac12\big(\alpha(1-\theta)\lambda_S \|t_k^S\|^2 - \Lip_S\alpha^2 \|t_k^S\|^2\big) \nonumber \\
    & =& \tfrac{\alpha}{2}\big((1-\theta)\lambda_S - \Lip_S\alpha\big) \|t_k^S\|^2  \geq  0,
  \end{eqnarray}
  if $(1-\theta)\lambda_S - \Lip_S\alpha \geq 0$. Therefore, we observe that if $\alpha$ satisfies $0 < \alpha \le (1-\theta)\frac{\lambda_S}{L_S}$,
  then $\alpha$ also satisfies the backtracking line search step of FCD (i.e., a function value reduction is achieved). Suppose that any $\alpha$ that is rejected by the line search is halved for the next line search trial. Then, it is guaranteed that the $\alpha$ that is accepted satisfies \eqref{Eq_alphamin}.

   We now show that \eqref{Eq_F_ubont} holds. By the previous arguments, the line search condition \eqref{Def_linesearch} is guaranteed to be satisfied for some step size $\alpha$. Then,
\begin{eqnarray}
  \notag
    F(x_k) - F(x_{k+1}) &\overset{\eqref{Def_linesearch}+\eqref{Eq_ell_convexity}}{\geq}&
    \theta \alpha(\ell(x_k;0)- \ell(x_k; U_St_k^S)) \\
    \label{FvsH}
    &\overset{\eqref{in:1}}{>}& \tfrac{\alpha\theta\lambda_S}{2}\|t_k^S\|^2 \geq \tfrac{\alpha_S\theta\lambda_S}{2}\|t_k^S\|^2
  \end{eqnarray}
  Using the definition of $\alpha_S$ in \eqref{Eq_alphamin} gives the result.
\end{proof}

The following lemma bounds the norm of the direction $t_k^S$ in terms of $g_S(x_k,0)$. This proof closely follows that of \cite[Theorem $3.1$]{sqa}.
\begin{lemma}\label{lem:Fdec2}
Let Assumptions \ref{Assume_Sampling}, \ref{Assump_HisPD}, \ref{Assump_fisLipschitz} and \ref{Assump_LambdaLrelation} hold, and let $\theta\in(0,1)$. Given $x_k$, let $S$, $t^{S}_k\neq 0$ and $\alpha$ be generated by FCD (Algorithm~\ref{FCD}) with $\eta_k^S\in[0,1)$. Then
\begin{equation}\label{in:8}
  \|t_k^S\| \ge \gamma_S \|g_S(x_k;0)\|, \quad \mbox{where} \quad \gamma_s \eqdef \tfrac{1-\eta_k^S}{1 + 2\Lambda_S}.
\end{equation}
Moreover,
\begin{equation}\label{Eq_Fubong}
  F(x_k) - F(x_k + {\alpha}U_St_k^S) > \theta(1-\theta)\tfrac{\lambda_S^2}{4L_S} \gamma_S^2 \|g_S(x_k;0)\|^2.
\end{equation}
\end{lemma}
\begin{proof}
From the termination condition \eqref{Def_stoppingconditions_2} we have that $\|g_S(x_k;t^S_k)\| \le \eta_k^S \|g_S(x_k;0)\|.$
  Using the reverse triangular inequality and the fact that $\prox_{\Psi_S^*}(\cdot)$ is nonexpansive, we have that
  \begin{eqnarray*}
  (1-\eta_k^S)\|g_S(x_k;0)\| & \le& \|g_S(x_k;0)\| - \|g_S(x_k;t^S_k)\| \\
  & \le& \|g_S(x_k;t_k^S) - g_S(x_k;0) \| \\
  & \stackrel{\eqref{defeq:2}}{=} & \|H_k^S t_k^S + \prox_{\Psi_S^*}\big(x_k^S + t_k^S - (\nabla_S f(x_k) + H_k^S t_k^S)  \big)  \\
  & &- \prox_{\Psi_S^*}\big(x_k^S - \nabla_S f(x_k) \big)\| \\
  & \overset{\eqref{Eq_proxnonexpansive}}{\le} & \|H_k^S t_k^S\| + \| (I-  H_k^S)t_k^S \| \\
  & \le & (1 + 2\|H_k^S\| )\|t_k^S\| \\
  & \overset{\eqref{Assumption_lambdai}}{\le} & (1 + 2\Lambda_S)\|t_k^S\|.
  \end{eqnarray*}
  Rearranging gives \eqref{in:8}, and combining \eqref{in:8} with \eqref{Eq_F_ubont} gives \eqref{Eq_Fubong}.
\end{proof}


\section{Technical results}\label{S_technicalresults}
In this section we will present several technical results that will be necessary when establishing our main iteration complexity results for FCD. The first result (Theorem~\ref{PeterThm}) is taken from \cite{Richtarik14} and will be used to finish all our iteration complexity results for FCD. The second result (Lemma~\ref{lem:bound_Q}) provides an upper bound on the decrease in the model that is obtained when an inexact search direction is used in FCD. The final two results provide upper bounds on the difference of function values $F(x_{k+1})- F(x_k)$ (Lemma~\ref{lem:another_intermediate}) and the expected distance to the solution $\mathbf{E}[F(x_{k+1})-F^*]$ (Lemma~\ref{Lemma_fandr}).

The following Theorem will be used at the end of all our proofs to obtain iteration complexity results for FCD. (It will be used with $\xi_k = F(x_k) - F^*$.)
\begin{theorem}[Theorem 1 in \cite{Richtarik14}]\label{PeterThm}
  Fix $\xi_0\in \R^N$ and let $\{\xi_k\}_{k \geq 0}$ be a sequence of random vectors in $\R^N$ with $\xi_{k+1}$ depending on $\xi_k$ only.
  Let $\{\xi_k\}_{k\ge 0}$ be a nonnegative non increasing sequence of random variables and assume that it has one of the following properties:
  \begin{itemize}
    \item[(i)] $\E[\xi_{k+1}|\xi_k] \leq \xi_k - \frac{\xi_k^2}{c_1}$ for all $k$, where $c_1>0$ is a constant,
    \item[(ii)] $\E[\xi_{k+1}|\xi_k] \leq (1-\frac{1}{c_2})\xi_k$ for all $k$ such that $\xi_k\geq \epsilon$, where $c_2>1$ is a constant.
  \end{itemize}
  Choose accuracy level $0<\epsilon<\xi_0$, confidence level $\rho \in (0,1)$.
  If (i) holds and we choose $\epsilon <c_1$ and
  $
    K \geq \tfrac{c_1}{\epsilon}\left(1 + \ln \tfrac{1}{\rho}\right) +2 - \tfrac{c_1}{\xi_0},
  $
  or if property (ii) holds and we choose
  $
    K\geq c_2 \ln\tfrac{\xi_0}{\epsilon \rho},
  $
  then $\mathbf{P}(\xi_K \leq \epsilon )\geq 1-\rho$.
\end{theorem}

The following result gives an upper bound on the decrease in the quadratic model when an inexact update is used in FCD.
\begin{lemma}\label{lem:bound_Q}
Let Assumptions \ref{Assume_Sampling}, \ref{Assump_HisPD}, and \ref{Assump_fisLipschitz} hold, let $\theta\in(0,1)$, and let $t_*^S$ denote the exact minimizer of subproblem \eqref{eq_subproblem}. Given $x_k$, let $S$, $t^{S}_k\neq 0$ and $\alpha$ be generated by FCD (Algorithm~\ref{FCD}) with $\eta_k^S\in[0,1)$, and let $x_{k+1} = x_k + \alpha U_St_k^S $. Then 
\begin{align*}
Q(x_k; U_S t_k^S) - Q(x_k;0) & \le Q(x_k; U_S t_*^S) - Q(x_k;0) + \frac{2(\eta_k^S)^2L_S}{\theta(1-\theta)\lambda_S^3\gamma_S^2}(F(x_k) - F(x_{k+1})),
\end{align*}
where $\eta_k^S$, $\theta$ and $\gamma_S$ are defined in \eqref{Def_stoppingconditions_2},  \eqref{Def_linesearch} and \eqref{in:8}, respectively.
\end{lemma}
\begin{proof}
  We have that
\begin{align}
  \label{bound_Q}
  Q(x_k; U_S t_k^S) - Q(x_k;0) = \left(Q(x_k; U_S t_k^S) - Q(x_k; U_S t_*^S)\right) + \left(Q(x_k; U_S t_*^S) - Q(x_k;0)\right).
  \end{align}
 The term $Q(x_k; U_S t_k^S) - Q(x_k; U_S t_*^S)$ can be bounded using strong convexity of $Q_S$ (see \cite[p.460]{bookboyd}). That is, for every $u \in \partial Q_S(x_k;t_k^S)$:
 \begin{align}
 Q(x_k; U_S t_k^S) - Q(x_k; U_S t_*^S)
 &\overset{\eqref{Eq_QvsQS}}{=} Q_S(x_k; t_k^S) - Q_S(x_k; t_*^S)
 \le \tfrac{1}{2\lambda_S}\|u \|^2\\
 \notag
 &\le \tfrac{1}{2\lambda_S}\|u + g_S(x_k;t^S_k) - g_S(x_k;t^S_k)\|^2 \\\notag
 							& \le \tfrac{1}{2\lambda_S}\Big(\|u - g_S(x_k;t^S_k)\|^2 + \|g_S(x_k;t^S_k)\|^2\Big).
 \end{align}
Therefore, using termination condition \eqref{Def_stoppingconditions_2}, where $\Omega$ is defined in \eqref{Omega}, we have that
 \begin{align}
 \notag
  Q(x_k; U_S t_k^S) - Q(x_k; U_S t_*^S) &\le \tfrac{1}{2\lambda_S}\Big(\inf_{u\in \Omega}\|u - g_S(x_k;t^S_k)\|^2+  \|g_S(x_k;t^S_k)\|^2\Big) \\\notag
    							&\le \tfrac{1}{2\lambda_S} (\eta_k^S \|g_S(x_k;0)\|)^2.
 \end{align}
Using Lemma \ref{lem:Fdec2} in the previous inequality, followed by \eqref{bound_Q}, gives the result.
\end{proof}

In the next Lemma we establish an upper bound on the difference between consecutive function values obtained using FCD.
 \begin{lemma}\label{lem:another_intermediate}
Let Assumptions \ref{Assume_Sampling}, \ref{Assump_HisPD}, \ref{Assump_fisLipschitz} and \ref{Assump_LambdaLrelation} hold. Let $\theta\in(0,1)$, fix $\tilde \eta \in [0,1)$, and let $t_*^S$ denote the exact minimizer of subproblem \eqref{eq_subproblem}. Given $x_k$, let $S$, $t^{S}_k\neq 0$ and $\alpha$ be generated by FCD (Algorithm~\ref{FCD}) with $\eta_k^S\in[0,\tilde \eta]$, and let $x_{k+1} = x_k + \alpha U_St_k^S $. Then
   \begin{equation}\label{Fub}
     F(x_{k+1}) - F(x_k) \overset{}{\leq} \chi(\tilde \eta)\big(Q(x_k; U_S t_*^S) - Q(x_k;0)\big),
   \end{equation}
  where
  \begin{equation}\label{chi_defs}
  \chi(\tilde \eta) = \min_{S \in 2^{[N]} \ : \ |S|=\tau}\; \frac{\theta(1-\theta)\lambda_S^3\gamma_S^2}{2L_S\Big(\tilde \eta^2 + \lambda_S\gamma_S^2(L_S - (1-\theta)\lambda_S)\Big)}.
  \end{equation}
  \end{lemma}
  \begin{proof}
  Using block Lipschitz continuity of $f$ \eqref{S2_upperbound}, and the definition of $Q$ \eqref{Def_Q} we get
 \begin{eqnarray*}
  \notag
    F(x_{k+1}) &\overset{}{\leq}& F(x_k) + (Q(x_k;\alpha U_S t_k^S) - Q(x_k;0))-  \tfrac{\alpha^2}2 \|t_k^S\|_{H_k^S}^2 + \tfrac{\alpha^2\Lip_S}{2} \|t_k^S\|^2\\
    \notag
    &\overset{\eqref{Assumption_lambdai}}{\leq}& F(x_k) + (Q(x_k;\alpha U_S t_k^S) - Q(x_k;0))+ \tfrac{\alpha^2}2 \left(\tfrac{\Lip_S}{\lambda_S}-1 \right)\|t_k^S\|_{H_k^S}^2.
  \end{eqnarray*}
Using convexity of $Q$ with respect to $\alpha$ gives
  \begin{align}\label{eq:345453}
  \notag
  F(x_{k+1}) &\le F(x_k) + \alpha(Q(x_k; U_S t_k^S) - Q(x_k;0))+ \tfrac{\alpha^2}2 \left(\tfrac{\Lip_S}{\lambda_S}-1 \right)\|t_k^S\|_{H_k^S}^2\\
  &\le F(x_k) + \alpha_{S}(Q(x_k; U_S t_k^S) - Q(x_k;0))+ \tfrac{\alpha^2}2 \left(\tfrac{\Lip_S}{\lambda_S}-1 \right)\|t_k^S\|_{H_k^S}^2,
  \end{align}
  where the second inequality holds because $Q(x_k; U_S t_k^S) - Q(x_k;0) < 0$ \eqref{Def_stoppingconditions_1}, and $\alpha \ge \alpha_S$ \eqref{Eq_alphamin}.
  By rearranging \eqref{FvsH} we get $\frac{\alpha}{2}\|t_k^S\|_{H_k^S}^2 \le \frac{1}{\theta}(F(x_k) - F(x_{k+1})).$  Using this in \eqref{eq:345453} gives
$$
  F(x_{k+1})  \overset{}{\leq} F(x_k) + \alpha_S(Q(x_k; U_S t_k^S) - Q(x_k;0)) + \tfrac{\alpha}{\theta}\left(\tfrac{\Lip_S}{\lambda_S}-1\right)(F(x_k) - F(x_{k+1}))
  $$
  By Assumption~\ref{Assump_LambdaLrelation}, $L_S \ge \lambda_S$, so we can replace $\alpha \in (0,1]$ with $1$ to get
  \begin{eqnarray}\label{FvsQF}
    F(x_{k+1})  \overset{}{\leq} F(x_k) + \alpha_S(Q(x_k; U_S t_k^S) - Q(x_k;0)) + \tfrac{1}{\theta}\left(\tfrac{\Lip_S}{\lambda_S}-1\right)(F(x_k) - F(x_{k+1})).
  \end{eqnarray}
  Using Lemma \ref{lem:bound_Q} and the definition of $\tilde \eta$ we get
  \begin{align*}\label{Rearrange}
  \notag
  F(x_{k+1})  & \overset{}{\leq} F(x_k) + \alpha_S(Q(x_k; U_S t_*^S) - Q(x_k;0))
  		   +  \tfrac{1}{\theta}\left(\tfrac{2\alpha_S\tilde{\eta}^2L_S}{(1-\theta)\lambda_S^3\gamma_S^2} +\tfrac{\Lip_S}{\lambda_S}-1\right)(F(x_k) - F(x_{k+1}))\\
  & \overset{\eqref{Eq_alphamin}}{\leq}F(x_k) + \alpha_S(Q(x_k; U_S t_*^S) - Q(x_k;0))
  		   + \tfrac{1}{\theta} \left(\tfrac{\tilde{\eta}^2}{\lambda_S^2\gamma_S^2} + \tfrac{\Lip_S}{\lambda_S}-1\right)(F(x_k) - F(x_{k+1})).
  \end{align*}
  Rearranging the previous, using the definitions of $\chi(\tilde \eta)$ and $\alpha_S$, gives the result.
  \end{proof}

\begin{remark}
  Notice that when exact directions are used in FCD (i.e., $\tilde \eta = 0$), we have
  \begin{equation}\label{eq_chi_simplified_exact}
    \chi(0) = \min_{S \in 2^{[N]} \ : \ |S|=\tau}\; \frac{\theta(1-\theta)\lambda_S^2}{2L_S^2 - 2(1-\theta)\lambda_SL_S}
  \end{equation}
  i.e., in general $\chi(\tilde \eta)$ depends cubically on $\lambda_S$ \eqref{chi_defs}, but in the case of exact directions, $\chi(0)$ only depends upon $\lambda_S^2$. However, if $L_S = \lambda_S$, but inexact directions are used ($\tilde \eta \neq 0$), then
  \begin{equation}\label{eq_chi_simplified_inexact}
    \chi(\tilde \eta) = \min_{S \in 2^{[N]} \ : \ |S|=\tau}\; \frac{1-\theta}{2}\frac{\theta\lambda_S^2\gamma_S^2}{\tilde \eta^2 + \theta\lambda_S^2\gamma_S^2},
  \end{equation}
  which again shows a dependence upon $\lambda_S^2$. Moreover, if $L_S = \lambda_S$, then $\chi(0) = (1-\theta)/2$ (i.e., constant). This is summarized in the following table.
  \begin{table}[H]\centering
    \begin{tabular}{|c | c| c|}
    \hline
    & Inexact Directions ($\tilde \eta\neq 0$) & Exact Directions ($\tilde \eta= 0$)\\
    \hline
      $\lambda_S \neq L_S$ & cubic & squared\\
      \hline
      $\lambda_S = L_S$ & squared & constant\\
      \hline
    \end{tabular}
    \caption{Dependence of $\chi(\tilde \eta)$ \eqref{chi_defs} upon $\lambda_S$ (Assumption~\ref{Assump_HisPD}).}
    \label{Table_chi}
  \end{table}

\end{remark}
The final result in this section gives an upper bound on the expected distance between the current function value and the optimal value.
\begin{lemma}\label{Lemma_fandr}
Let Assumptions \ref{Assume_Sampling}, \ref{Assump_HisPD}, \ref{Assump_fisLipschitz} and \ref{Assump_LambdaLrelation} hold. Let $\theta\in(0,1)$, fix $\tilde \eta \in [0,1)$, and let $\mu_f \geq 0$ be the strong convexity constant of $f$. Given $x_k$, let $S$, $t^{S}_k\neq 0$ and $\alpha$ be generated by FCD (Algorithm~\ref{FCD}) with $\eta_k^S\in[0,\tilde \eta]$, and let $x_{k+1} = x_k + \alpha U_St_k^S $. Then
\begin{equation}\label{Lemma18}
     \E[F(x_{k+1})- F^*|x_k]
     \leq \left(1- \tfrac{\tau\chi(\tilde \eta)\zeta}{N}\right)(F(x_k) - F^* ) - \left(\mu_f - \Lmax\zeta \right)\tfrac{\tau\chi(\tilde \eta)\zeta}{2N} \|x_k - x_*\|^2,
  \end{equation}
  where $\zeta \in [0,1]$, $\chi(\tilde \eta)$ is defined in \eqref{chi_defs} and
  \begin{equation}\label{chi_lambda_max}
   \Lmax \eqdef \max_{S \in 2^{[N]} \ : \ |S|=\tau} \Lambda_S.
  \end{equation}
\end{lemma}
\begin{proof}
The assumptions of Lemma \ref{lem:another_intermediate} are satisfied, so \eqref{Fub} holds.
  Notice that $t_*^S$ is the vector that makes the difference $Q(x_k; U_S t_*^S) - Q(x_k;0)$, as small/negative as possible
  over the subspace defined by the columns of $U_S$. Therefore, for any vector $y_k^S \neq t_*^S$,
 $
    Q(x_k; U_S t_*^S) - Q(x_k;0) \leq Q(x_k; U_S y_k^S) - Q(x_k;0).
 $
  Choosing $y_k ^S = \zeta(x_*^S - x_k^S)$ for any $\zeta \in [0,1]$ gives
   \begin{eqnarray*}
    F(x_{k+1}) - F(x_k) &\overset{}{\leq}& \chi(\tilde \eta)(Q(x_k; \zeta U_S (x_*^S - x_k^S)) - Q(x_k;0)) \\
    &\overset{\eqref{Def_Q}}{\leq}& \chi(\tilde \eta)\Big(\zeta \langle \nabla f(x_k),U_S(x_*^S - x_k^S)\rangle  + \frac{\zeta^2}{2}\|x_*^S - x_k^S\|_{H_k^S}^2 \\
    &+& \Psi_S(x_k^S + \zeta (x_*^S - x_k^S)) - \Psi_S(x_k^S)\Big).
    \end{eqnarray*}
    Using convexity of $\Psi$, \eqref{Assumption_lambdai} and the definition $\Lmax$ \eqref{chi_lambda_max} we get
    \begin{align}\label{Fdiffupper}
     F(x_{k+1}) - F(x_k) & \leq \chi(\tilde \eta)\zeta \Big(\langle \nabla f(x_k),U_S(x_*^S - x_k^S)\rangle  + \Psi_S(x_*^S) - \Psi_S(x_k^S) \Big) \\
                 \notag  &+ \tfrac{\chi(\tilde \eta)\Lmax\zeta^2}{2}\|x_*^S - x_k^S\|^2.
  \end{align}
  Taking expectation of \eqref{Fdiffupper} and using convexity of $f$ gives
  \begin{eqnarray*}
     \E[F(x_{k+1}) - F(x_k)|x_k] &\overset{\eqref{eq:tech_prob}}{\leq}&\tfrac{\tau\chi(\tilde \eta)\zeta}{N} \Big( \langle \nabla f(x_k),x_* - x_k \rangle  + \Psi(x_*) - \Psi(x_k) \Big) \\
     &+& \tfrac{\tau \chi\Lmax}{N}\tfrac{\zeta^2}{2}\|x_k - x_*\|^2\\
     &\overset{}{\leq}& - \tfrac{\tau\chi(\tilde \eta)\zeta}{N}\left(F(x_k) - F^* \right) - \tfrac{\tau\chi(\tilde \eta)\zeta}{N} \tfrac{\mu_f}{2}\|x_k - x_*\|_{2}^2\\
     &+& \tfrac{\tau \chi(\tilde \eta)\Lmax}{N}\tfrac{\zeta^2}{2}\|x_k - x_*\|^2.
  \end{eqnarray*}
  Rearranging and subtracting $F^*$ from both sides gives the result.
  \end{proof}

\section{Iteration complexity results for inexact FCD}\label{S_Complexity}
In this section we present iteration complexity results for FCD applied to problem \eqref{Def_F} when an inexact update is used.

\subsection{Iteration complexity in the convex case}\label{S_Complexity_C}
Here we present iteration complexity results for FCD in the convex case.

\begin{theorem}\label{ItcomplexCInexact}
Let $F$ be the convex function defined in \eqref{Def_F} and let Assumptions \ref{Assume_Sampling}, \ref{Assump_HisPD}, \ref{Assump_fisLipschitz} and \ref{Assump_LambdaLrelation} hold. Choose an initial point $x_0\in \R^N$, a target confidence $\rho \in (0,1)$, fix $\tilde{\eta} \in [0,1)$, and recall that $\chi(\tilde \eta)$ is defined in \eqref{chi_defs}. Let the target accuracy $\epsilon$ and iteration counter $K$ be chosen in either of the following two ways:
\begin{itemize}
  \item[(i)] Let $m_1\eqdef \max\{\mathcal{R}^2(x_0),F(x_0) - F^* \}$, $0<\epsilon < F(x_0) - F^*$ and
    \begin{equation}
        K \geq \frac{2 N }{\tau \chi(\tilde \eta)}\frac{m_1}{\epsilon}\left(1 + \ln \frac{1}{\rho}\right) +2 - \frac{2 N}{\tau \chi(\tilde \eta)}\frac{m_1}{(F(x_0) - F^*)},
    \end{equation}
    \item[(ii)] Let $0<\epsilon < \min \{\mathcal{R}^2(x_0),F(x_0) - F^*\}$ and
    \begin{equation}
    K\geq \frac{2 N }{\tau \chi(\tilde \eta)}\frac{\mathcal{R}^2(x_0)}{\epsilon} \ln\frac{F(x_0) - F^*}{\epsilon \rho}.
  \end{equation}
\end{itemize}
If $\{x_k\}_{k\geq 0}$ are the random points generated by FCD (Algorithm~\ref{FCD}), with $\eta_k^S \in [0,\tilde{\eta}]$ for all $k$ and all $S\in2^{[N]}$, then $\Prob(F(x_K) - F^* \leq \epsilon) \geq 1 - \rho$.
\end{theorem}
\begin{proof}
By Lemma~\ref{lem:Fdec} we have that $F(x_k) \leq F(x_0)$ for all $k$, so $\|x_k - x_*\| \leq \mathcal{R}(x_0)$ for all $x_* \in X^*$, where $\mathcal{R}(x_0)$ is defined in \eqref{Def_Rlevelset}.
Then, using Lemma~\ref{Lemma_fandr} with $\mu_f = 0$ gives
\begin{eqnarray}\label{Eq_findzeta}
     \E[F(x_{k+1}) - F^*|x_k]
     \leq
     \left(1- \tfrac{\tau\chi(\tilde \eta)\zeta}{N}\right)\left(F(x_k) - F^* \right) + \tfrac{\tau\chi(\tilde \eta)\Lmax\zeta^2}{2N} \mathcal{R}^2(x_0).
  \end{eqnarray}
Minimizing the right hand side of the previous with respect to $\zeta$ gives:
$
    \zeta = \min \left\{\frac{F(x_k) - F^*}{\Lmax\mathcal{R}^2(x_0)},1\right\}.
$
Then, \eqref{Eq_findzeta} becomes
\begin{eqnarray*}
\E[F(x_{k+1}) - F^*|x_k]
     \leq
  \begin{cases}
    \left(1- \tfrac{\tau\chi(\tilde \eta)}{2N}\tfrac{F(x_k) - F^*}{\Lmax\mathcal{R}^2(x_0)}\right) (F(x_k) - F^*) & \text{if}\;\; \tfrac{F(x_k)-F^*}{\Lmax\mathcal{R}^2(x_0)} \leq 1\\
    \left(1- \frac{\tau\chi(\tilde \eta)}{2N}\right)(F(x_k) - F^*) & \text{otherwise},
  \end{cases}
\end{eqnarray*}
or alternatively,
\begin{eqnarray}\label{Eq_overapprox}
  \E[F(x_{k+1})-F^*|x_k]
     \leq \left(1 - \tfrac{\tau\chi(\tilde \eta)}{2N}\min\left\{\tfrac{F(x_k) - F^* }{\Lmax\mathcal{R}^2(x_0)}, 1\right\}\right)(F(x_k) - F^* ).
\end{eqnarray}
It is enough to study \eqref{Eq_overapprox} under the condition $F(x_k) - F^* \leq \Lmax\mathcal{R}^2(x_0)$.\footnote{Notice that, at any particular iteration of FCD, either of the two options inside the `min' in \eqref{Eq_overapprox} could be the smallest. However, $\Lmax\mathcal{R}^2(x_0)$ is fixed throughout the iterations, and $F(x_k) - F^*$ is becoming smaller as the iterations progress by Lemma~\ref{lem:Fdec}. Therefore, eventually it will be the case that the first term inside the `min' in \eqref{Eq_overapprox} will be the smallest, i.e., $F(x_k) - F^* \leq \Lmax\mathcal{R}^2(x_0)$.} Hence, we have the following two cases.
\begin{itemize}
  \item[(i)]
      Letting $c_1 = \frac{2 N}{\tau \chi(\tilde \eta)}\max\{\Lmax\mathcal{R}^2(x_0),F(x_0) - F^* \}$, \eqref{Eq_overapprox} becomes
\begin{eqnarray*}
  \E[F(x_{k+1})-F^*|x_k]&\leq&\left(1 - \tfrac{\tau\chi(\tilde \eta)}{2N}\tfrac{F(x_k) - F^* }{\Lmax\mathcal{R}^2(x_0)}\right)(F(x_k) - F^* )\\
     &\leq& \left(1 - \tfrac{F(x_k) - F^* }{c_1}\right)(F(x_k) - F^* ).
\end{eqnarray*}
Notice that for this choice of $c_1$, $\epsilon < F(x_0)-F^* < c_1$, so it suffices to apply Theorem~\ref{PeterThm}(i) and the result follows.
\item[(ii)]
Now we choose $c_2 = \frac{2 N }{\tau\chi(\tilde \eta)}\frac{\mathcal{R}^2(x_0)}{\epsilon}> 1$. Observe that, whenever $\epsilon \leq F(x_k) - F^* $, by \eqref{Eq_overapprox} we have
$
       \E[F(x_{k+1})-F^*|x_k]
     \leq (1 - \tfrac{1}{c_2})(F(x_k) - F^* ).
$
    It remains to apply Theorem~\ref{PeterThm}(ii), and the result follows.
\end{itemize}
\end{proof}

\subsection{Iteration complexity in the strongly convex case}

In this section we establish iteration complexity results for FCD when the objective function $F$, and the smooth function $f$, are strongly convex, with strong convexity parameters $\mu_f>0$ and $\mu_F>0$, respectively. Moreover, $\mu_f \leq \mu_F$.

Before presenting our iteration complexity results, we make the following assumption. It explains that at least one of the chosen matrices $\{H_k^S\}_{k\geq 0}$ has an eigenvalue greater than $\mu_f$.
\begin{assumption}\label{Assume_Lambdamuf}
  We assume that $\Lambda_{\max}\geq \mu_f> 0 $, where $\Lambda_{\max}$ is defined in \eqref{chi_lambda_max}.
\end{assumption}

We also define the following variable, which appears in the complexity results:
\begin{equation}\label{delta}
  \delta \eqdef
   \begin{cases}
     \frac{\mu_f + \mu_F}{4 \Lambda_{\max}}\Big(1 + \frac{\mu_f}{\mu_F}\Big) & \text{if}\;\; \mu_F + \mu_f < 2\Lambda_{\max}\\
     1 - \frac{\Lambda_{\max}-\mu_f}{\mu_F} & \text{otherwise}.
   \end{cases}
\end{equation}
It is straightforward to show that $\delta \in (0,1]$. In particular, if $\mu_F + \mu_f < 2\Lambda_{\max}$, the conclusion follows because $\mu_F \geq \mu_f \Rightarrow 1+\frac{\mu_f}{\mu_F} \leq 2$. On the other hand, if $ 2\Lambda_{\max} \leq \mu_F+\mu_f$, then by inspection, we have $(\mu_{F}+\mu_f-\Lambda_{\max})/\mu_F \in (0,1]$.


\begin{theorem}\label{ItcomplexSCInexact}
  Let $f$ and $F$ be the strongly convex functions defined in \eqref{Def_F} with $\mu_f >0$ and $\mu_F >0$ respectively, and let Assumptions \ref{Assume_Sampling}, \ref{Assump_HisPD}, \ref{Assump_fisLipschitz}, \ref{Assump_LambdaLrelation} and \ref{Assume_Lambdamuf} hold.
 Choose an initial point $x_0\in \R^N$, a target confidence $\rho \in (0,1)$, a target accuracy $\epsilon >0$, and fix $\tilde{\eta} \in [0,1)$. Let $\chi(\tilde \eta)$ and $\delta$ be defined in \eqref{chi_defs} and \eqref{delta} respectively, and let
  \begin{equation}
    K \geq \frac{N}{\tau \chi(\tilde \eta) \delta} \ln \left(\frac{F(x_0)-F^*}{\epsilon \rho}\right).
  \end{equation}
If $\{x_k\}_{k\geq 0}$ are the random points generated by FCD (Algorithm~\ref{FCD}), with $\eta_k^S \in [0,\tilde{\eta}]$ for all $k$ and all $S\in2^{[N]}$, then $\Prob(F(x_K) - F^* \leq \epsilon) \geq 1 - \rho$.
\end{theorem}

\begin{proof}
  Define $\xi_k\eqdef F(x_k)-F^*$. From Lemma~\ref{Lemma_fandr} we have that
  \begin{eqnarray}\label{Eq_zetaSC}
  \notag
     \E[\xi_{k+1}|x_k]
     &\overset{}{\leq}& \left(1- \tfrac{\tau\chi(\tilde \eta)\zeta}{N} \right)\xi_k + \left(\Lmax\zeta -\mu_f \right)\tfrac{\tau\chi(\tilde \eta)\zeta}{2N} \|x_k - x_*\|^2\\
     &\overset{\eqref{strongly_convex_1}}{\leq}& \left(1- \tfrac{\tau\chi(\tilde \eta)\zeta}{N} \right)\xi_k +\left(\Lmax\zeta -\mu_f \right) \tfrac{\tau\chi(\tilde \eta)\zeta}{N\mu_F}\xi_k.
  \end{eqnarray}
  Notice that \eqref{strongly_convex_1} can only be applied in \eqref{Eq_zetaSC} when $\mu_f \leq \Lmax\zeta$. But, differentiating the right hand side of \eqref{Eq_zetaSC} with respect to $\zeta$ gives $\zeta = \min\{(\mu_F + \mu_f)/2\Lmax,1\}$.\footnote{We could simply have set $\zeta = 1$ initially, so that $\mu_f \leq \Lmax\zeta$ is satisfied by Assumption~\ref{Assume_Lambdamuf}. However, we obtain a better complexity result by taking a smaller $\zeta$.} Then
  \begin{eqnarray*}
\E[\xi_{k+1}|x_k]
     \leq
  \begin{cases}
    \left(1- \frac{\tau\chi(\tilde \eta)}{N}\tfrac{(\mu_F + \mu_f)}{4\Lmax}\Big(1 + \frac{\mu_f}{\mu_F}\Big)\right)\xi_k & \text{if}\;\; \tfrac{\mu_F + \mu_f}{2\Lambda_{\max}} < 1\\
    \left(1- \frac{\tau\chi(\tilde \eta)}{N}\Big(1 - \frac{\Lmax-\mu_f}{\mu_F}\Big) \right)\xi_k & \text{otherwise.}
  \end{cases}
\end{eqnarray*}
Now, using \eqref{delta}, we can write $\E[\xi_{k+1}|x_k]  \leq \left(1 - \tfrac{\tau \chi(\tilde \eta) \delta}{N}\right) \xi_k.$
The result follows by applying Theorem~\ref{PeterThm}(ii) with $c_2 =  \frac{N}{\tau \chi(\tilde \eta) \delta}>1$.
\end{proof}



\section{Iteration complexity of FCD applied to smooth functions}\label{S_ComplexitySmooth}

The iteration complexity results presented in Section~\ref{S_Complexity} simplify in the smooth case. Therefore, in this section we assume that $\Psi = 0$ so that $F \equiv f$. In the smooth case, the quadratic model becomes
\begin{equation}\label{QSsmooth}
  Q(x_k;U_St_k^S) \equiv Q_S(x_k;t_k^S) = \langle \nabla_S f(x_k),t_k^S \rangle + \tfrac12  \langle H_k^S t_k^S ,t_k^S \rangle.
\end{equation}
Notice that the stationarity conditions also simplify in the following way
\begin{equation}\label{gS_smooth}
  g_S(x_k;t_k^S) = \nabla_S f(x_k) + H_k^S t_k^S,
\end{equation}
so that minimizing the (smooth) quadratic model \eqref{QSsmooth}, is equivalent to solving the linear system
\begin{equation}\label{smoothlinsystem}
  H_k^S t_k^S = - \nabla_S f(x_k).
\end{equation}

The matrix $H_k^S$ is always symmetric and positive definite (by definition), so an obvious approach is to solve \eqref{smoothlinsystem} using the Conjugate Gradient method (CG). It is possible to use other iterative methods to approximately minimize \eqref{QSsmooth}/solve \eqref{smoothlinsystem}. However, we will see that using CG gives better iteration complexity guarantees, so, for the results presented in this section, we assume that CG is always used to solve \eqref{smoothlinsystem}.

We have the following theorem, which allows us to determine the value of the function $Q_S(x_k;U_St_k^S)$ when an inexact direction $t_k^S$ is used.
\begin{lemma}[Lemma 7 in \cite{l1regSCfg}]\label{Lemma_Kimon}
  Let $Az=b$, where $A$ is a symmetric and positive definite matrix. Furthermore, let us assume that this system is solved using CG approximately; CG is terminated prematurely at the $j$th iteration. Then if CG is initialized with the zero solution, the approximate solution $z_j$ satisfies $z_j^TAz_j = z_j^Tb$. 
\end{lemma}
Further, note that for smooth functions, the FCD stopping condition \eqref{Def_stoppingconditions_2} becomes
\begin{equation}\label{termination_SCsmooth}
  \frac{\|g_S(x;t_k^S) \|_2}{\|g_S(x;0) \|_2} \equiv \frac{\|\nabla_S f(x_k) + H_k^S t_k^S\|_2}{\|\nabla_S f(x_k)\|_2} \leq \eta_k^S.
\end{equation}
Therefore, in this section we will make the following assumption.
\begin{assumption}\label{Assume_tksmooth}
  We assume that if FCD (Algorithm~\ref{FCD}) is applied to problem \eqref{Def_F} with $\Psi = 0$, then, for all $k$ and all $ S\in2^{[N]}$, the inexact search direction $t_k^S$ is generated by applying $j$ iterations of CG to \eqref{smoothlinsystem}, initialized with the zero vector, until the stopping condition \eqref{termination_SCsmooth} is satisfied, where $\eta_k^S\in[0,1)$ .
\end{assumption}


Let $t_k^S$ be the inexact solution obtained by applying CG to \eqref{smoothlinsystem} starting from the zero vector, and terminating according to \eqref{termination_SCsmooth}, (i.e., let Assumption~\ref{Assume_tksmooth} hold). Then, by Lemma~\ref{Lemma_Kimon}
\begin{equation}\label{Kimonequality}
  \langle H_k^S t_k^S ,t_k^S \rangle = - \langle \nabla_S f(x_k),t_k^S \rangle.
\end{equation}
so that
\begin{equation}\label{QSsmooth_eval}
  Q_S(x_k;t_k^S) = \tfrac12\langle \nabla_S f(x_k),t_k^S \rangle \equiv -\tfrac12\langle H_k^S t_k^S ,t_k^S \rangle <0.
\end{equation}
Therefore, combining \eqref{QSsmooth_eval} with the fact that $Q(x_k;0) = 0$ in the smooth case, the stopping condition $Q(x_k;U_St_k^S) \;(= Q_S(x_k;t_k^S)) < Q(x_k;0)$ (i.e., \eqref{Def_stoppingconditions_1}) is always satisfied.

\subsection{Technical Results when $\Psi=0$}

We begin by presenting several technical results that will be needed when establishing iteration complexity results for FCD in the smooth ($\Psi=0$) case. The first result is analogous to Lemma~\ref{lem:Fdec2} for smooth functions, while the second is similar to Lemma~\ref{lem:Fdec}.
\begin{lemma}\label{Lemmagvstsmooth}
Let Assumptions \ref{Assume_Sampling}, \ref{Assump_HisPD}, \ref{Assump_fisLipschitz}, \ref{Assump_LambdaLrelation} and \ref{Assume_tksmooth} hold. Given $x_k$, let $S$, and $t^{S}_k\neq 0$ be generated by FCD (Algorithm~\ref{FCD}), applied to problem \eqref{Def_F} with $\Psi = 0$, and with $\eta_k^S\in[0,1)$. Then
\begin{equation}\label{gvstsmooth}
  \tfrac{1-\eta_k^S}{\Lambda_S} \|\nabla_S f(x_k)\|_{2}
  \leq \|t_k^S\|_{2}.
\end{equation}
\end{lemma}
\begin{proof}
This proof closely follows that of \cite[Theorem $3.1$]{sqa}, and of Lemma~\ref{lem:Fdec2}. The termination condition \eqref{termination_SCsmooth} is $ \|g_S(x_k;t^S_k)\|_2 \le \eta_k^S \|g_S(x_k;0)\|_2$. Then
  \begin{eqnarray*}
  (1-\eta_k^S)\|g_S(x_k;0)\|_2 &\le& \|g_S(x_k;0)\|_2 - \|g_S(x_k;t^S_k)\|_2 \\
   &\le& \|g_S(x_k;t_k^S) - g_S(x_k;0) \|_2 \\
  & \overset{\eqref{gS_smooth}}{=} & \|H_k^S t_k^S\|_2 \leq \Lambda_S \|t_k^S\|_2. 
  \end{eqnarray*}
  It remains to recall that $\nabla_S f(x_k) \equiv g_S(x_k;0)$.
\end{proof}

\begin{lemma}\label{iterationcomplexitysmoothSC}
Let Assumptions \ref{Assume_Sampling}, \ref{Assump_HisPD}, \ref{Assump_fisLipschitz}, \ref{Assump_LambdaLrelation} and \ref{Assume_tksmooth} hold, let $\theta\in(0,1/2)$, and fix $\tilde \eta \in [0,1)$.
Given $x_k$, let $S$, $t^{S}_k\neq 0$ be generated by FCD (Algorithm~\ref{FCD}) applied to problem \eqref{Def_F} with $\Psi = 0$, and with $\eta_k^S\in[0,\tilde \eta]$. Then Step 6 of FCD will accept a step-size $\alpha$ that satisfies
  \begin{eqnarray}\label{Eq_alphaminSC}
    \alpha \geq \tfrac{\theta}2 \tfrac{\lambda_S}{L_S}>0.
  \end{eqnarray}
  Moreover,
  \begin{eqnarray}\label{fdiffSC}
  f(x_k) - f(x_{k} + \alpha U_St_k^S) \geq \tfrac{\vartheta(\tilde \eta)}{2}\|\nabla_S f(x_k)\|_{2}^2,
\end{eqnarray}
where
\begin{eqnarray}\label{Def_vartheta}
  \vartheta(\tilde \eta) \eqdef \min_{S:|S|=\tau}  \frac{ \theta \lambda_S^2}{L_S}\frac{(1-\tilde\eta)^2}{\Lambda_S^2}.
\end{eqnarray}
\end{lemma}
\begin{proof}
The proof follows that of Lemma 9 in \cite{l1regSCfg}. By Lipschitz continuity of the gradient of $f$:
\begin{eqnarray*}
  f(x_{k} + \alpha U_St_k^S) &\leq& f(x_k) + \alpha\langle \nabla_S f(x_k), t_k^S \rangle + \tfrac{L_S \alpha^2}2 \|t_k^S\|_2^2\\
  &\overset{\eqref{Kimonequality}}{\leq}& f(x_k) - \alpha \|t_k^S\|_{H_k^S}^2 + \tfrac{L_S \alpha^2}{2\lambda_S} \|t_k^S\|_{H_k^S}^2.
\end{eqnarray*}
The right hand side of the above inequality is minimized for $\alpha^* = \lambda_S/L_S$, which gives
$
  f(x_{k} + \alpha^* U_St_k^S) \le  f(x_k) - \frac{\lambda_S}{2L_S}\|t_k^S\|_{H_k^S}^2.
$
Observe that this step-size satisfies the backtracking line-search condition \eqref{Def_linesearch} because
$
  f(x_{k} + \alpha^* U_St_k^S) \leq f(x_k) - \tfrac12\tfrac{\lambda_S}{L_S}\|t_k^S\|_{H_k^S}^2 < f(x_k) - \theta \tfrac{\lambda_S}{L_S}\|t_k^S\|_{H_k^S}^2.
$
Suppose that any $\alpha$ that is rejected by the line search is halved for the next line search trial. Then, it is guaranteed that the $\alpha$ that is accepted in Step 6 of FCD, satisfies \eqref{Eq_alphaminSC}.
This, combined with Lemma~\ref{Lemmagvstsmooth}, results in the following decrease in the objective function:
\begin{eqnarray}\label{fdiffSCv2}
  f(x_k) - f(x_{k} + \alpha U_St_k^S)
  \geq  \tfrac{\theta \lambda_S}{2L_S}\|t_k^S\|_{H_k^S}^2
  \geq  \tfrac{\theta \lambda_S^2}{2L_S}\|t_k^S\|_{2}^2
  \geq  \tfrac{\theta \lambda_S^2}{2L_S} \tfrac{(1-\eta_k^S)^2}{\Lambda_S^2} \|\nabla_S f(x_k)\|_{2}^2.
\end{eqnarray}
Using the definition \eqref{Def_vartheta} in \eqref{fdiffSCv2} gives \eqref{fdiffSC}.
\end{proof}

\begin{remark}
  The quantity $\vartheta(\tilde \eta)$ in \eqref{Def_vartheta} will appear in the complexity result for FCD in the smooth case. Notice that, while $\chi(\tilde \eta)$ depends upon $\lambda_S^3$, $\vartheta(\tilde \eta)$ only depends upon $\lambda_S^2$. Interestingly, the dependence of $\vartheta(\tilde \eta)$ on $\lambda_S$ is unaffected by an inexact search direction.
  \begin{table}[H]\centering
    \begin{tabular}{|c |c |c|}
    \hline
    & $\lambda_S = L_S$ & $\lambda_S = \Lambda_S$\\
    \hline
      $\vartheta(\tilde \eta)$ & $\min_{S:|S|=\tau} \frac{\theta(1-\tilde\eta)^2 L_S}{\Lambda_S^2}$ & $\min_{S:|S|=\tau}  \frac{ \theta (1-\tilde\eta)^2}{L_S}$\\
      \hline
    \end{tabular}
  \end{table}
\end{remark}

\subsection{Iteration Complexity in the Convex Smooth Case}

We are now ready to present iteration complexity results for FCD applied to smooth convex functions of the form \eqref{Def_F} when $\Psi=0$.
\begin{theorem}\label{ItcomplexCInexactSmooth}
Let Assumptions \ref{Assume_Sampling}, \ref{Assump_HisPD}, \ref{Assump_fisLipschitz}, \ref{Assump_LambdaLrelation} and \ref{Assume_tksmooth} hold. Choose an initial point $x_0\in \R^N$, a target confidence $\rho \in (0,1)$, accuracy $0 <\epsilon<\max\{\mathcal{R}^2(x_0),f(x_0) - f^* \}$, fix $\tilde{\eta} \in [0,1)$. Let $\vartheta(\tilde \eta)$ be as defined in \eqref{Def_vartheta}, with $\theta \in (0,1/2)$ and let
    \begin{equation}
        K \geq \frac{2 N \mathcal{R}^2(x_0) }{\tau \vartheta(\eta)\epsilon} \left(1 + \ln \frac{1}{\rho}\right) + 2 - \frac{2 N \mathcal{R}^2(x_0) }{\tau \vartheta(\eta)(f(x_0) - f^*)}.
    \end{equation}
If $\{x_k\}_{k\geq 0}$ are the random points generated by FCD (Algorithm~\ref{FCD}) applied to problem \eqref{Def_F} with $\Psi=0$, where $\eta_k^S\in[0,1)$ satisfies $\eta_k^S/(1-\eta_k^S)^2 \leq \eta<1$ for all $ k$ and all $ S\in2^{[N]}$, then $\Prob(F(x_K) - F^* \leq \epsilon) \geq 1 - \rho$.
\end{theorem}
\begin{proof}
  By Lemma~\ref{iterationcomplexitysmoothSC} $f(x_k)\leq f(x_0)$ for all $k$, and by convexity of $f$, we have
  $f(x_k)-f^*\leq \max_{x^*\in X^*} \langle \nabla f(x_k),x_k-x^* \rangle \leq \|\nabla f(x_k)\|_2 \mathcal{R}(x_0)$. Taking expectation of \eqref{fdiffSCv2}, and applying the previous, gives:
  \begin{eqnarray*}
  f(x_k) - \E[f(x_{k} + \alpha U_St_k^S)|x_k]
  \geq  \tfrac{\vartheta(\tilde \eta)\tau}{2N} \| f(x_k)\|_{2}^2
  \geq \tfrac{\vartheta(\tilde \eta)\tau}{2N} \left( \tfrac{f(x_k)-f^*}{\mathcal{R}(x_0)}\right)^2.
\end{eqnarray*}
Rearranging and applying Theorem~\ref{PeterThm}(i) with $c_1= 2N \mathcal{R}^2(x_0)/\tau \vartheta(\tilde \eta)$, gives the result.
\end{proof}

\subsection{Iteration Complexity in the Strongly Convex Smooth Case}
We now present iteration complexity results for FCD applied to smooth, strongly convex functions of the form \eqref{Def_F}, when $\Psi=0$.

\begin{theorem}\label{iterationcomplexity_SCS}
  Let $\Psi=0$, let $f=F$ be the strongly convex function defined in \eqref{Def_F} with $\mu_f >0$ and let Assumptions \ref{Assume_Sampling}, \ref{Assump_HisPD}, \ref{Assump_fisLipschitz}, \ref{Assump_LambdaLrelation} and \ref{Assume_tksmooth} hold.
 Choose an initial point $x_0\in \R^N$, a target confidence $\rho \in (0,1)$, a target accuracy $\epsilon >0$, and fix $\tilde{\eta} \in [0,1)$. Let $\vartheta(\tilde \eta)$ be as defined in \eqref{Def_vartheta}, with $\theta \in (0,1/2)$ and let
  \begin{equation}
    K \geq \frac{N}{\tau \vartheta(\tilde \eta) \mu_f} \ln \left(\frac{F(x_0)-F^*}{\epsilon \rho}\right).
  \end{equation}
If $\{x_k\}_{k\geq 0}$ are the random points generated by FCD (Algorithm~\ref{FCD}) applied to problem \eqref{Def_F} with $\Psi=0$, where $\eta_k^S\in[0,\tilde \eta]$, then $\Prob(F(x_K) - F^* \leq \epsilon) \geq 1 - \rho$.
\end{theorem}
\begin{proof}
Notice that, by strong convexity, $f(x_k)-f^* \leq \frac1{2\mu_f}\|\nabla f(x_k)\|_2^2$, (see, e.g., \cite[p.460]{bookboyd}). Taking the expectation of \eqref{fdiffSC}, and using the previous, gives:
  \begin{eqnarray*}
 \E[ f(x_k) - f(x_{k} + \alpha_S U_St_k^S)|x_k] \geq \tfrac{\tau \vartheta(\tilde \eta)}{2N}\|\nabla f(x_k)\|_{2}^2\geq \tfrac{\tau\vartheta(\tilde \eta) \mu_f}{N}(f(x_k)-f^*).
\end{eqnarray*}
Rearranging, and applying Theorem~\ref{PeterThm}(ii) with $c_2 = N/(\mu_f\tau\vartheta(\tilde \eta))>1$ gives the result.
\end{proof}


\section{Discussion/Comparison}

In this section we discuss the complexity results derived in this paper, and compare them with the current state of the art results.


\subsection{General Framework}

The FCD framework has been designed to be extremely \emph{flexible}. Moreover, this framework generalizes several existing algorithms.

To see this, suppose that FCD is set-up in the following way. Choose $|S|=\tau=1$, $H_k^i=L_i$ (where $L_i$ is the Lipschitz constant for the $i$th coordinate), suppose coordinate $i$ is selected with uniform probability and suppose that \eqref{eq_subproblem} is solved exactly at every iteration. Then the update generated by FCD is given by
\begin{eqnarray*}
  t_k^i = \arg \min_{t\in\R} \{\langle \nabla_i f(x), t \rangle + \tfrac{L_i}2 t^2 + \Psi_i(x^i + t)\}.
\end{eqnarray*}
However, this is equivalent to the update generated by UCDC (see Algorithm 1 and equation (17)) in \cite{Richtarik14}. So FCD generalizes UCDC, (and accepts a step-size $\alpha=1$ for all iterations in this case, i.e., no line-search is needed by FCD.)

Now suppose that we assume that $\Psi=0$, and that $f$ is strongly convex. For simplicity, let us assume that $f$ is quadratic. Moreover, suppose that FCD is set-up in the following way. Choose $|S|=\tau=1$, $H_k^i=\nabla_{ii}^2 f$ (for a quadratic function the Hessian $\nabla^2 f(\cdot)$ is independent of $x_k$), suppose coordinate $i$ is selected with uniform probability and suppose that \eqref{eq_subproblem} is solved exactly at every iteration. Then the update generated by FCD is $t_k^i = - (\nabla_{i} f(x_k))/(\nabla_{ii}^2 f)$. However, this is equivalent to the update generated by SDNA in \cite{Qu15} (Methods 1--3 in \cite{Qu15} are equivalent when $\tau=1$). So FCD generalizes SDNA, (and also accepts a step-size $\alpha=1$ for all iterations in this case, i.e., no line-search is needed by FCD.)

Clearly, a central advantage of FCD is it's flexibility. It recovers several existing state-of-the-art algorithms, depending on the particular choice of algorithm parameters. We stress that, for UCDC, one must always compute and use the Lipschitz constants, while for FCD, $H_k^S$ can be any positive semi-definite matrix. Further, SDNA only applies to problems that are smooth and strongly convex, whereas FCD can be applied to general convex problems of the form \eqref{Def_F}.

\subsection{Practicality vs Iteration Complexity}

We remark that it is expected that the complexity results for FCD may be slightly worse than for other existing methods. There are several reasons for this:
\begin{enumerate}
\item FCD is a general framework that enables substantial user flexibility, 
    as well as computational practicality. The user may choose the block size (i.e., the number of coordinates to be updated at each iteration) the matrix $H_k^S$. Therefore, the user has full flexibility regarding how \emph{expensive} each iteration of FCD is. For example, choosing $\tau$ small, and $H_k^S$ to be diagonal, will lead to cheap iterations. Choosing a larger block size and/or a full matrix $H_k^S$, will lead to iterations that are more expensive. However, it is anticipated that fewer iterations may be needed when $H_k^S$ is a good approximation to a principal minor of the Hessian. Because of the user friendly, general and flexible framework, the complexity rates for FCD  may be slightly worse than for methods that apply to specific problem instances, with specially tuned parameters.
\item Another difference between FCD, and other current methods, is that many methods do not incorporate any second order information, or they enforce that $H_k^S = L_S I$.
    The restrictions made in other methods are usually imposed to ensure that their surrogate function/model, overapproximates $F(x_k + U_St_k^S)$, so that minimizing their model will result in a reduction in the original function value.
    The assumption that FCD makes regarding $H_k^S$ is much weaker (\emph{any} symmetric positive definite $H_k^S$ is allowed), which means that FCD requires a line search, to guarantee a reduction in the objective function value at every iteration, and ultimately, convergence. We prefer an algorithm with fewer assumptions, so that it is \emph{widely applicable} and \emph{practical}, but we pay the price of a (potentially) slightly weaker complexity result.
\item FCD is one of the (very) few coordinate descent methods that allow inexact updates. It is expected that inexact updates will lead to cheaper iterations (it is intuitive that solving a subproblem inexactly, will usually require less computational effort than solving it exactly), although this may result in slightly more iterations compared with an `exact' method. We stress that, despite the potential increase in the number of iterations, it is still expected that inexact updates will result in a lower computational run time overall.
\end{enumerate}
We stress that, despite the slightly weaker iteration complexity results for FCD, the algorithm works extremely well in practice, and often outperforms algorithms with `better' complexity rates, (see the numerical experiments in Section~\ref{S_Numerical}). Moreover, the user has some control over the complexity results for FCD through their choice of parameters. For example, $\lambda_S$ and $\Lambda_S$ (Assumption~\ref{Assump_HisPD}) are user defined parameters, which appear in the complexity rates for FCD (see \eqref{chi_defs} and \eqref{Def_vartheta}), so the complexity rate for FCD can be larger or smaller, depending on how the user chooses these parameters.

\subsection{Comparison of complexity results}

We will now compare the complexity results obtained in this work for FCD, with those of UCDC in \cite{Richtarik14} and PCDM \cite{Richtarik15}.

\subsubsection{Comparison of complexity results with UCDC \cite{Richtarik14}}\label{CompareUCDC}

UCDC is a \emph{serial} method that allows \emph{exact updates only}, so in this section we assume that $|S|=\tau=1$, and that FCD uses only exact updates ($\tilde \eta \equiv 0$). This means that we have $L_i$ for $i=1,\dots,N$, and $H_k^i$ is a scalar, for all $k$. The complexity results in \cite{Richtarik14} use `scaled' quantities that depend on the vector of Lipschitz constants ($L=(L_1,\dots,L_N)$) while the quantities in this work depend on the `unscaled' 2-norm (i.e., $e = (1,\dots,1)$).\footnote{Using the notation of \cite{Richtarik14}, we have $n=N$. Further, although UCDC allows arbitrary probabilities in the smooth case, we restrict our attention to uniform probabilities only, so $N$ appears in the C-S and SC-S results in Table~\ref{Table_Comparison}.} Moreover, in \cite{Richtarik14}, the authors suppose that strong convexity can come from $f$ or $\Psi$, or both. Here, to allow easy comparison, we suppose that $\mu_{\Psi}=0$, which means that, in this section, if strong convexity is assumed, then $\mu_f=\mu_F$. To this end, define the following constants, which depend a `scaling' vector $w \in \R_{++}^N$:
\begin{eqnarray}
\label{c1}
    c_1(w)= \tfrac{2N}{\epsilon} \max\{{\cal R}_{w}^2(x_0),F(x_0)-F^*\}, \text{ and }
    c_2(w)= \tfrac{2N {\cal R}_{w}^2(x_0)}{\epsilon}.
  \end{eqnarray}
Table~\ref{Table_Comparison} presents the complexity results obtained for UCDC \cite{Richtarik14}, and the complexity results obtained in this work for FCD.
The following notation is used in the table. Constants $\xi_0 = F(x_0)-F^*$, $\mu_{\phi}(w)$ denotes the strong convexity parameter of function $\phi$ with respect to a `$w$-weighted' norm, $\mu = (\mu_f + \mu_\Psi)/(1 + \mu_\Psi)$ and ${\cal R}_{w}(x_0)$ is defined in \eqref{Def_Rlevelset_2}. 

\begin{table}[h!]\centering
\begin{tabular}{|c| c| c| c|}
\hline
$F$ & UCDC \cite{Richtarik14} & FCD [this paper] & Theorem \\
\hline
C-N
& $\displaystyle c_1(L) \left(1+\log\frac1{\rho} - \frac{\epsilon}{\xi_0}\right)+2$
& $\displaystyle \frac{c_1(e)}{\chi(0)}\left(1+\log\frac1{\rho} - \frac{\epsilon}{\xi_0}\right)+2$
& \ref{ItcomplexCInexact}(i)\\
\cline{2-4}
C-N & $\displaystyle c_2(L) \log\left(\frac{\xi_0}{\epsilon \rho}\right)$
& $\displaystyle \frac{c_2(e)}{\chi(0)} \log\left(\frac{\xi_0}{\epsilon \rho}\right)$
& \ref{ItcomplexCInexact}(ii)\\
\hline
SC-N & $\displaystyle \frac{N}{\mu_f(L)} \log\left(\frac{\xi_0}{\epsilon \rho}\right)$
& $\displaystyle \frac{N}{\chi(0) \delta}\log\left(\frac{\xi_0}{\epsilon \rho}\right)$
& \ref{ItcomplexSCInexact} \\
\hline
C-S & $\displaystyle c_2(L) \log\left(\frac{\xi_0}{\epsilon \rho}\right)$
&  $\displaystyle \frac{c_2(e)}{\vartheta(0)} \log\left(\frac{\xi_0}{\epsilon \rho}\right)$
& \ref{ItcomplexCInexactSmooth} \\
\hline
SC-S & $\displaystyle \frac{N}{\mu_f(L)} \log\left(\frac{\xi_0}{\epsilon \rho}\right)$
& $\displaystyle \frac{N}{\vartheta(0) \mu_f} \log\left(\frac{\xi_0}{\epsilon \rho}\right)$
& \ref{iterationcomplexity_SCS}\\
\hline
\end{tabular}
\caption{Comparison of the iteration complexity results for coordinate descent methods using an inexact update and using an exact update (C=Convex, SC=Strongly Convex, N=Nonsmooth, S = Smooth).}\vspace{-5mm}
\label{Table_Comparison}
\end{table}

\paragraph{Case C-N(i).} In this case, the difference between the complexity rates for UCDC and FCD appears in the constants
       $c_1(L)$ and $c_1(e)/\chi(0)$.
Clearly, if ${\cal R}_{L}^2(x_0)\leq \xi_0^F$ then $c_1(L) = c_1(e)$, so that the the complexity rate of FCD is simply $1/\chi(0)$ times worse than UCDC. (Recall Table~\ref{Table_chi} for a comparison of $\chi(\cdot)$ values.)
On the other hand, if ${\cal R}_{L}^2(x_0) > \xi_0^F$, then it is difficult to compare the complexity rates directly. Notice that the dependence of FCD on the Lipschitz constants is explicit (the constants $L_i$ appear in $\chi(0)$), whereas the dependence of UCDC on the Lipschitz constants is implicit in the weighted term $\mathcal{R}_L^2(x_0)$.
However, consider the case when $L_i=L_j$ for all $i,j$. Then
$\|y-x\|_L^2 = \sum_{i=1}^N L_i \|y_i-x_i\|_2^2 = L_i\|y-x\|_2^2,$ so that ${\cal R}_{L}^2(x_0) = L_i{\cal R}^2(x_0)$ in this case. Then we simply compare $L_i$ vs $1/\chi(0)$. Now suppose we choose $H_k^i = L_i$ for all $k$, so that $\chi(0) = (1-\theta)/2$. If $\theta \approx 0$, then FCD has a better complexity rate than UCDC whenever $L_i > 2$. Of course, in other circumstances, UCDC may have a better complexity rate than FCD --- this is simply one particular example.

\paragraph{Case C-N(ii).} Here we compare $c_2(L)$ and $c_2(\mathbf{1})$. Both rates contain the  term $2N/\epsilon$, so it remains to compare
${\cal R}_{L}^2(x_0)$ and ${\cal R}^2(x_0)/ \chi(0)$. Thus, see the discussion for C-N(i).

\paragraph{Case SC-N.} Substituting $\mu_f=\mu_F$ into \eqref{delta},  shows that $\delta = \mu_f/ \Lambda_{\max}$, so in this case we compare $1/\mu_f(L)$ with $\Lambda_{\max}/\chi(0)\mu_f$. Moreover, the Lipschitz constants appear explicitly in $\chi(0)$ in the analysis in this work, while they are again implicit in the term $\mu_f(L)$ in \cite{Richtarik14}. Again, this makes it difficult to compare the rates directly. However, consider the following example, where $\mu_f= \Lambda_{\max}$. Then $\lambda_S = L_S$ (because $\lambda_S\leq \Lambda_{\max} = \mu_f \leq L_S$), so $\Lambda_{\max}/\chi(0)\mu_f = 2/(1-\theta)$. Then, if $\theta \approx 0$, FCD has a better complexity rate than UCDC, whenever $\mu_f(L) \leq 1/2$. (Note that it always holds that $\mu_f(L)\leq 1$ \cite[(12)]{Richtarik14}, and this does not necessarily imply that $(\mu_f(e)\equiv)\mu_f \leq 1/2$).

\paragraph{Case C-S.} Similar to Case C-N(ii), we compare ${\cal R}_{L}^2(x_0)$ and ${\cal R}^2(x_0)/ \vartheta(0)$. Suppose we have the same set up as Case C-N(i), so that we compare $L_i$ with $1/\vartheta(0)$. Choosing $\lambda_S = \Lambda_S$ and $\theta\approx 1/2$ shows that $1/\vartheta(0) \approx 2 L_i$, so that the complexity rates for FCD is approximately twice that for UCDC.

\paragraph{Case SC-S.} Here we compare the quantities $1/\mu_f(L)$ and $1/\mu_f\vartheta(0)$. All that can be said for UCDC is that $1/\mu_f(L) \geq 1$ because $\mu_f(L)\leq 1$. Sometimes it may be the case that $1/\mu_f(L) \approx 1$, whereas in other cases, we may have $1/\mu_f(L) \gg 1$. Note that, for FCD, if we choose $\lambda_S = \Lambda_S$, then $1/\mu_f \vartheta(0) = L_S/\theta \mu_f$. For a well conditioned problem, (i.e., $\mu_f\approx L_S$), choosing $\theta\approx 1$ gives $1/\mu_f \vartheta(0)\approx 1$. On the other hand, if $L_S \gg \theta\mu_f$, then $1/\mu_f \vartheta(0) \gg 1$. So, while the rates for FDC and UCDC cannot be compared directly, the complexity rates are similar, when taking into account the problem conditioning.

\subsubsection{Comparison of complexity results with PCDM \cite{Richtarik15}}
Now we compare FCD with PCDM \cite{Richtarik15}. PCDM is a parallel coordinate descent method, that uses \emph{exact updates}, so in this case we suppose that $|S| = \tau \geq 1$, and that FCD also uses exact updates. Furthermore, we will suppose that the coordinates are sampled according to Definition~\ref{Def_Sampling} (which is called a $\tau$-nice sampling in \cite{Richtarik15}). Define $\beta = 1+ \frac{(\omega -1)(\tau-1)}{\max\{1,N-1\}}$, where $\omega$ is a measure of the `separability' of the objective function. Then, the iteration complexity for PCDM is presented in Table~\ref{TablePCDM}.
\begin{table}[h!]\centering
  \begin{tabular}{|c| c |c|}
\hline
& C-N & SC-N\\
\hline
  PCDM \cite{Richtarik15}
  & $\displaystyle \tfrac{2 N}{\tau \epsilon}\max\{\beta \mathcal{R}^2_L(x_0),F(x_0)-F^*\}\left(1+\log\frac1{\rho} - \frac{\epsilon}{\xi_0}\right)+2$
  & $\displaystyle \tfrac{\beta N}{\tau \mu_f(L)}\log\left(\frac{\xi_0}{\epsilon\rho}\right)$\\
  \hline
  FCD & $\displaystyle\tfrac{2N}{\tau\epsilon \chi(0)} \max\{{\cal R}^2(x_0),F(x_0)-F^*\}\left(1+\log\frac1{\rho} - \frac{\epsilon}{\xi_0}\right)+2$
   &
    $\displaystyle\frac{N}{\tau\delta \chi(0)}\log\left(\frac{\xi_0}{\epsilon\rho}\right)$\\
    \hline
\end{tabular}
\caption{Complexity results for PCDM \cite{Richtarik15}.}
\label{TablePCDM}
\end{table}

\paragraph{C-N.} Note that $\beta\geq 1$, so if $\beta \mathcal{R}^2_L(x_0)\leq F(x_0)-F^*$, then the complexity rate for FCD is $1/\chi(0)$ times worse than for PCDM. On the other hand, suppose that $\mathcal{R}^2_L(x_0)> F(x_0)-F^*$. Then we must compare the quantities $\beta \mathcal{R}^2_L(x_0)$ and $\mathcal{R}^2(x_0)/\chi(0)$. Similar arguments to the Case C-N(i) in Section~\ref{CompareUCDC} can be made here. However, we note that $\beta$ depends upon $\tau$, $N$ and $\omega$. While $\tau$ and $N$ can be controlled (they are parameters of the algorithm), $\omega$ depends upon the problem data, and could be very large, implying large $\beta$.


\paragraph{SC-N.} In this case we compare the quantities $\beta/\mu_f(L)$ and $1/\delta\chi(0)$. The case for FCD and $1/\delta\chi(0)$ is made in  Section~\ref{CompareUCDC}, Case SC-N. Now consider PCDM. For a well conditioned problem, and separable problem, we may have $\mu_f(L)\approx 1 \approx \omega$, so that $\beta/\mu_f(L) \approx 1$. However, for a poorly conditioned problem, $\omega$ may be very large, and $\mu_f(L)$ may be very small, so that $\beta / \mu_f(L)\gg 1$. Again, while the rates for FDC and PCDM cannot be compared directly, the complexity rates are similar, when taking into account the problem conditioning.



\section{Numerical Experiments} \label{S_Numerical}
In this section we compare the performance of FCD with UCDC \cite{Richtarik14}
on the $\ell_1$- and $\ell_2$-regularized logistic regression problem, which has the form  \eqref{Def_F} with
\begin{equation}
\label{logistic_regression}
f(x) = \sum_{j=1}^m\log(1+e^{-b_jx^T a_j}) \quad \mbox{and} \quad \Psi(x) =c \|x\|_1 \ \mbox{or} \ \Psi(x) =c \|x\|_2.
\end{equation}
Here $c>0$, $a_j\in\mathbb{R}^N$ $\forall j=1,2,\ldots,m$ are the training samples and $b_j\in\{-1,+1\}$ are the corresponding labels. Problems of the form \eqref{logistic_regression} are important in machine learning and are used for training a linear classifier $x\in\mathbb{R}^N$ that separates input data into two distinct clusters; see, for example, \cite{yuanho}. We present the performance of FCD and UCDC on six real-world large-scale data sets from the LIBSVM data collection \cite{Chang11}.  Details of the data sets are given in Table~\ref{LRprobs}, where $A\in\mathbb{R}^{m\times N}$ is a matrix whose rows represent the training samples.
\begin{table}[h!]
	\centering
	\caption{Properties of the $\ell_1$-regularized logistic regression data sets considered here. Note that \rm{ nnz($A$)} denotes the number of nonzeros in matrix $A$; consequently, the fourth column gives the (relative) sparsity of $A$.}
\begin{tabular}{|l|c|c|c|}
\hline
		\textbf{Problem}& $m$	& $N$		 & \rm{nnz($A$)}/$mN$  \\
\hline
kdd2010 (algebra)	        & $8,407,752$   & $20,216,830$	& $1.79e$-$6$  \\
kdd2010 (bridge to algebra)	& 19,264,097	& 29,890,095	&  $9.83e$-$7$ \\
news20.binary	            & 19,996		& 1,355,191	    &    $3.35e$-$4$ \\
rcv1.binary	                & 20,242	    &	47,236	    &   $1.16e$-$3$ \\
url	                        & 2,396,130		& 3,231,961	    &  $3.57e$-$5$  \\
webspam	                    & $350,000$	    & $16,609,143$	& $2.24e$-$4$  \\
\hline
	\end{tabular}
	\label{LRprobs}
\end{table}


\subsection{Implementations of FCD and UCDC}

\begin{table}[h!]
	\centering
	\caption{Algorithm parameters for FCD~v.1, FCD~v.2, UCDC~v.1, UCDC~v.2. For FCD~v.2 we set $\rho=10^{-6}$, $L_j$ is the Lipschitz constant for the $j$-th coordinate (see \eqref{S2_Lipschitz}), and $\bar L^S = \sum_{j\in S}L_j$.}
\begin{tabular}{|l|c|c|c|}
\hline
		Algorithm & $\tau$	& $H_k^S$	($\forall k,S$)	   \\
\hline
FCD~v.1 & $\ceil{0.001 N}$ & $\text{diag}(\nabla_S^2 f(x_k))$,   \\
FCD~v.2 & $\ceil{0.001 N}$ & $\nabla_S^2 f(x_k) + \rho I_{N_i}$ \\
UCDC~v.1 & $1$& $L_j$  \\
UCDC~v.2 & $\ceil{0.001 N}$& $\bar L^S I$  \\
\hline
	\end{tabular}
	\label{AlgTable}
\end{table}

\paragraph{FCD.} At every iteration of FCD, $\tau$ coordinates are sampled uniformly at random without replacement, i.e., Assumption~\ref{Assume_Sampling} holds. FCD is implemented with two different choices for $H_k^S$, as shown in Table~\ref{AlgTable}.
All other parameters are the same for FCD~v.1 and FCD~v.2. Clearly, for FCD~v.1, subproblem \eqref{eq_subproblem} is separable
and for $\ell_1$-regularization it has the closed form solution
\begin{equation}\label{tupdateSFT}
  t^S_k = \mathcal{S}(x^S_k - (H_k^S)^{-1}\nabla_S f(x^S_k), c\, \text{diag}((H_k^S)^{-1})),
\end{equation}
where
\begin{equation}
\label{soft_thresholding}
 \mathcal{S}(u,v) = \mbox{sign}(u) \max(|u| - v,0)
\end{equation}
is the well-known soft-thresholding operator, which is applied component wise when $u$ and $v$ are vectors. Since subproblem \eqref{eq_subproblem} is solved exactly via \eqref{tupdateSFT}, there is no need to verify the stopping conditions \eqref{Def_stoppingconditions_2}.
For $\ell_2$-regularization and FCD~v.1 subproblem \eqref{eq_subproblem} has a trivial and inexpensive solution.

For FCD~v.2, parameter $\rho$ ensures that $H_k^S$ is positive definite, and consequently, subproblem \eqref{eq_subproblem} is well defined.\footnote{Clearly, there is a trade-off when selecting $\rho$. The larger $\rho$ is, the smaller the condition number of $H_k$, so that \eqref{eq_subproblem} will be solved quickly using an iterative solver. However, if $\rho$ is too large then $H_k^S \approx \rho I $ so that essential second order information from $\nabla^2 f(x_k)$ may be lost.} Note that the matrix $H_k^S$ is never explicitly formed, we only perform matrix-vector products with it in a matrix-free manner. Moreover, for FCD~v.2, subproblem \eqref{eq_subproblem} is solved iteratively using an Orthant Wise Limited-memory Quasi-Newton (OWL) method in the case of $\ell_1$-regularization, which can be downloaded from \url{http://www.di.ens.fr/~mschmidt/Software/L1General.html}. OWL was chosen because it has been shown in \cite{sqa} to result in a robust and efficient deterministic version of FCD, i.e. $\tau=N$ (one block of size $N$). While for $\ell_2$-regularization and FCD~v.2 an iterative method like Conjugate Gradients is used.

\paragraph{UCDC.}
We compare FCD with the current state-of-the-art coordinate descent method, namely the Uniform Coordinate Descent for Composite functions (UCDC) method, \cite[Algorithm $2$]{Richtarik14}.  For UCDC, the block size $\tau$ and the decomposition of $\mathbb{R}^N$ into $\ceil{N/\tau}$ blocks must be \emph{fixed a-priori}, and at each iteration of UCDC the blocks are selected with uniform probability. Again, UCDC is also implemented with two different choices of $H_k^S$, as seen in Table~\ref{AlgTable}. The reasons for this set-up is as follows. For UCDC, the block Lipschitz constants are explicitly required in the algorithm, and, while the Lipschitz constants for \emph{coordinates} can be computed with relative ease, the Lipschitz constants for \emph{blocks of coordinates} can be far more expensive to compute. To this end, UCDC~v.1 is the case of coordinate updates where the Lipschitz constants are straightforward to compute, and UCDC~v.2 considers blocks of $\tau$ coordinates, where $\bar{L}^S \eqdef \sum_{j\in S}L_j$ is used as an \emph{inexpensive overapproximation} of the true block Lipschitz constant.  Here, the (coordinate) Lipschitz constants for UCDC are computed via \cite[Table~10]{Richtarik14}. Note that, for both UCDC~v.1 and v.2, the choice of $H_k^S$ leads to a separable subproblem \eqref{eq_subproblem}, which is solved exactly using it's closed for solution.

\begin{remark}
   Notice that Algorithm $2$ in \cite{Richtarik14} is a special case of FCD where the subproblem \eqref{eq_subproblem} is solved exactly. A line search is unnecessary in this case because, for $H_k^S$ as stated, \eqref{eq_subproblem} is an overapproximation of $F$ along block coordinate direction $t^S$, and minimizing this overapproximation exactly will lead to a decrease in the objective function; see \cite{Richtarik14}.
\end{remark}

\subsection{Termination Criteria and Parameter Tuning}
In these experiments FCD and UCDC are terminated when their running time exceeds an allowed maximum. Note that, using subgradients as a measure of the distance from optimality or any other operation of similar cost are considered to be too expensive a task for large scale problems, and are therefore not appropriate here. Here, all methods were terminated after the relative error of the objective function dropped below $10^{-8}$ or after $28$ hours of wall-clock time.

Furthermore, for FCD we set parameter $\eta^S_k$ in \eqref{Def_stoppingconditions_2} equal to $0.9$ $\forall i,k$.
The maximum number of backtracking line search iterations is set to $200$, $\theta=10^{-3}$ and each iteration the step-size is halved. For all instances we computed $c$ after performing fivefold cross validation as proposed in \cite{svmguid}. For UCDC the coordinate Lipschitz constants $L_j$ $\forall j$ are calculated once at the beginning of the algorithm and this task is included in the overall running time. Finally, all methods are initialized with the zero solution.

\subsection{Results for $\ell_1$-regularized logistic regression}\label{subsec:l1regexper}

The results of this experiment is shown in the following figures. Figure~\ref{Fig_small} shows the results of FCD and UCDC on several of the smaller data sets, and Figure~\ref{Fig_big} shows the results of FCD and UCDC on several of the larger data sets
for $\ell_1$-regularized logistic regression. Note that \textit{the calculation of $F(x)$ is included in the running time of FCD when it is required by line search}. The reported running time of UCDC includes the time required to compute the Lipschitz constants at the beginning of the algorithm.
\begin{figure}[h!]\centering
\subfloat[news20.b, $\frac{F(x) - F^*}{F^*}$ vs iterations]{\label{fig_rw_webspam_its}\includegraphics[scale=0.27]{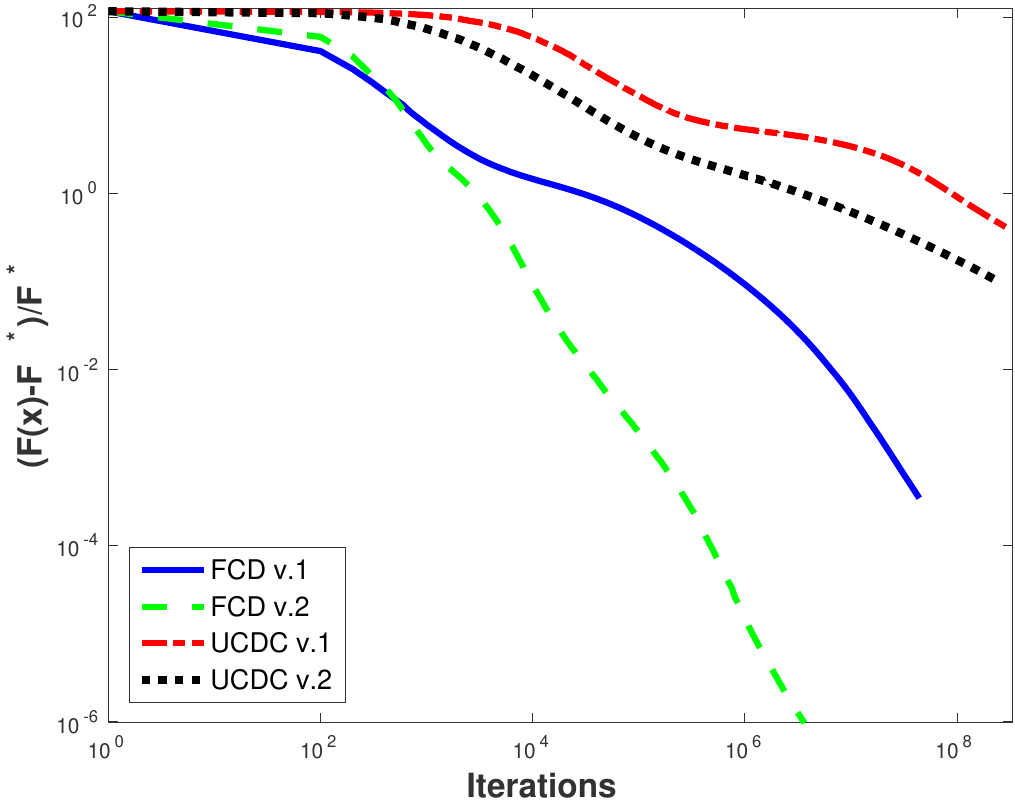}}
\subfloat[rcv1,          $\frac{F(x) - F^*}{F^*}$ vs iterations]{\label{fig2c}\includegraphics[scale=0.27]{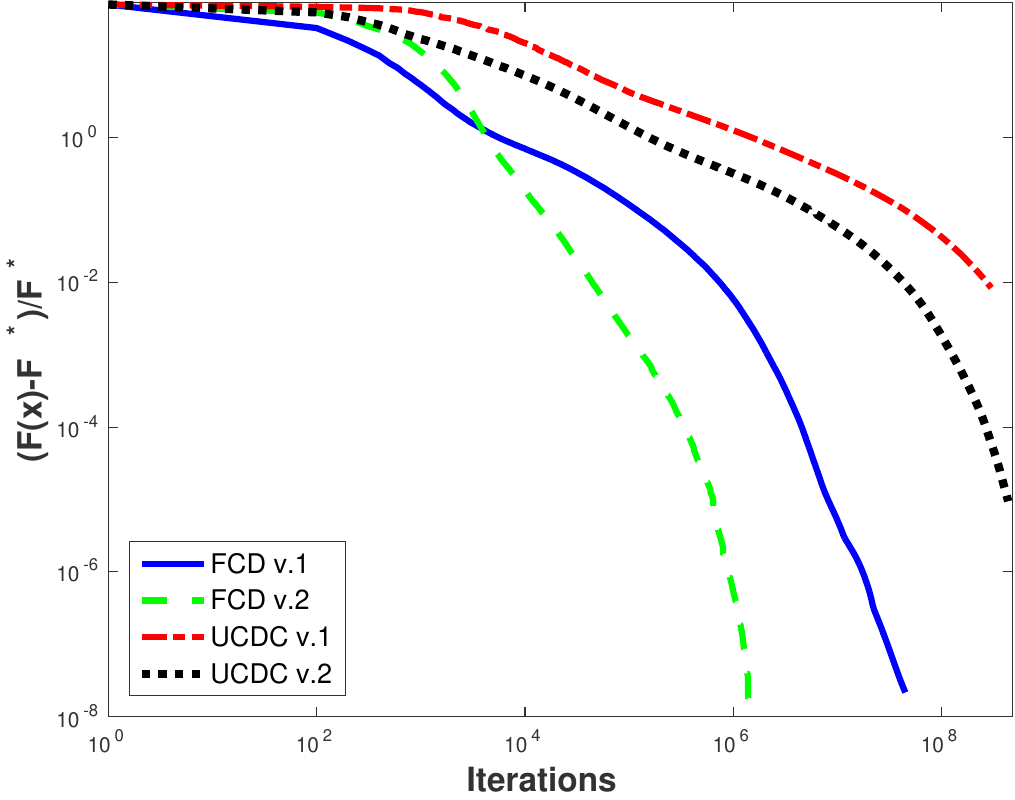}}
\subfloat[url,            $\frac{F(x) - F^*}{F^*}$ vs iterations]{\label{fig2a}\includegraphics[scale=0.27]{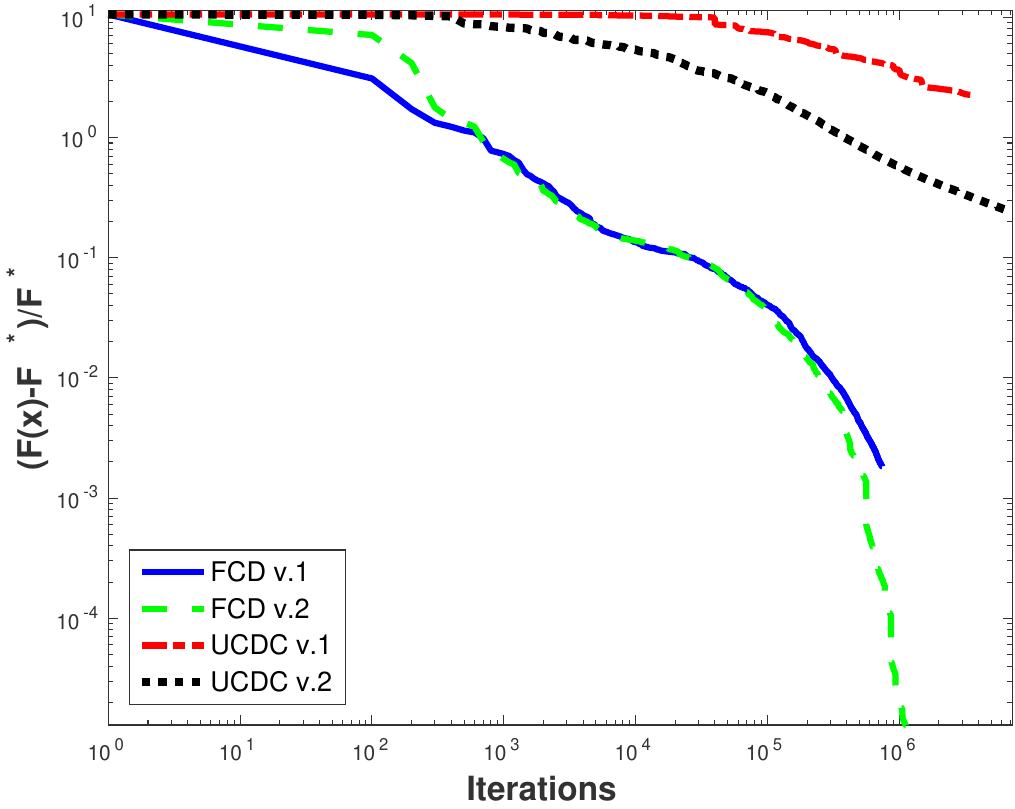}}\\
\subfloat[news20.b, $\frac{F(x) - F^*}{F^*}$ vs time]{\label{fig_rw_webspam_fn}\includegraphics[scale=0.27]{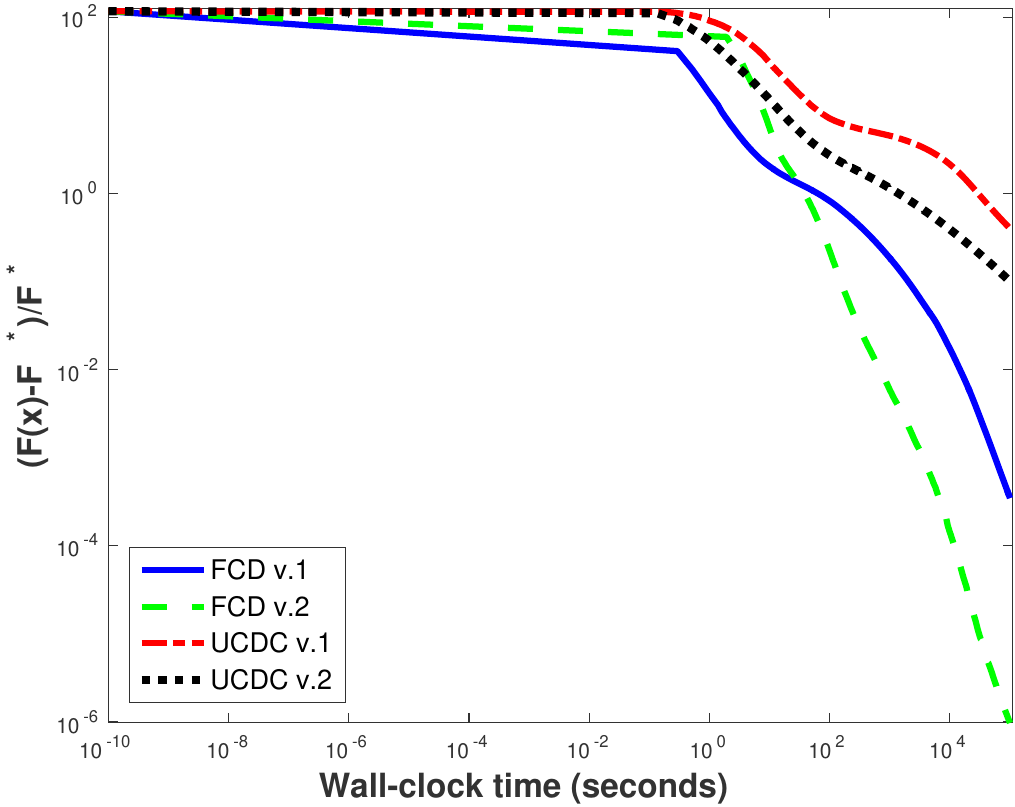}}
\subfloat[rcv1,          $\frac{F(x) - F^*}{F^*}$ vs time]{\label{fig2d}\includegraphics[scale=0.27]{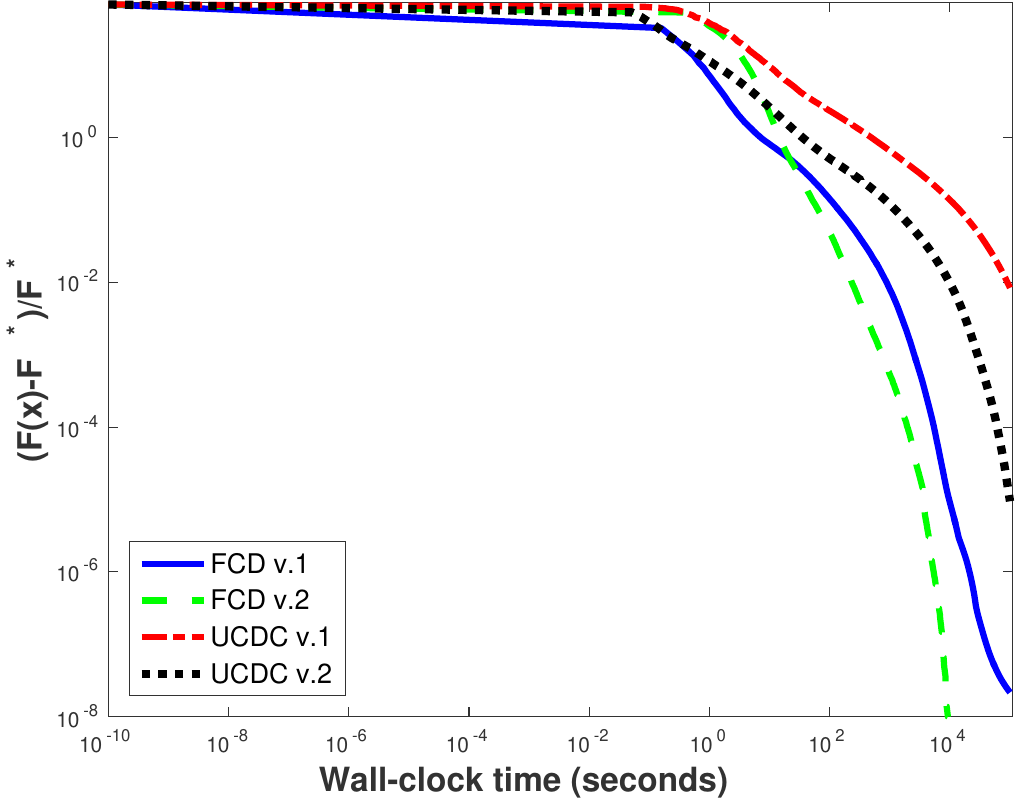}}
\subfloat[url,           $\frac{F(x) - F^*}{F^*}$ vs time]{\label{fig2b}\includegraphics[scale=0.27]{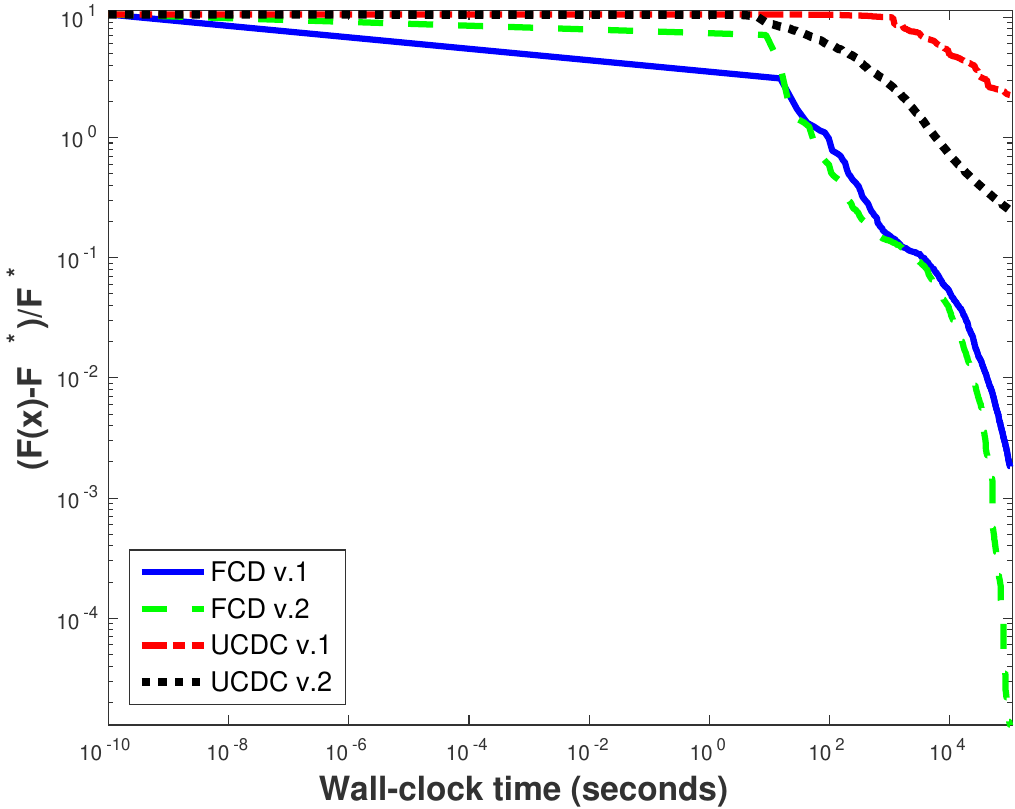}}\\
\caption{Performance of FCD and UCDC on 3 large scale $\ell_1$-regularized logistic regression problems.
The first row shows the function value $\frac{F(x) - F^*}{F^*}$ vs the number of iterations, while the second row shows the function value $\frac{F(x) - F^*}{F^*}$ vs time. Columns 1, 2 and 3 correspond to data sets news20.binary, rcv1 and url, respectively. All plots are in log-scale. For some figures, i.e., Figures \ref{fig_rw_webspam_fn}, \ref{fig2d} and \ref{fig2b}, we measured $\frac{F(x) - F^*}{F^*}$ every $100$ iterations of the algorithms, which explains the initial rapid decrease in $\frac{F(x) - F^*}{F^*}$.}
\label{Fig_small}
\end{figure}
\begin{figure}[h!]
\centering
\subfloat[webspam, $\frac{F(x) - F^*}{F^*}$ vs iterations]{\label{fig_rw_webspam_its}\includegraphics[scale=0.27]{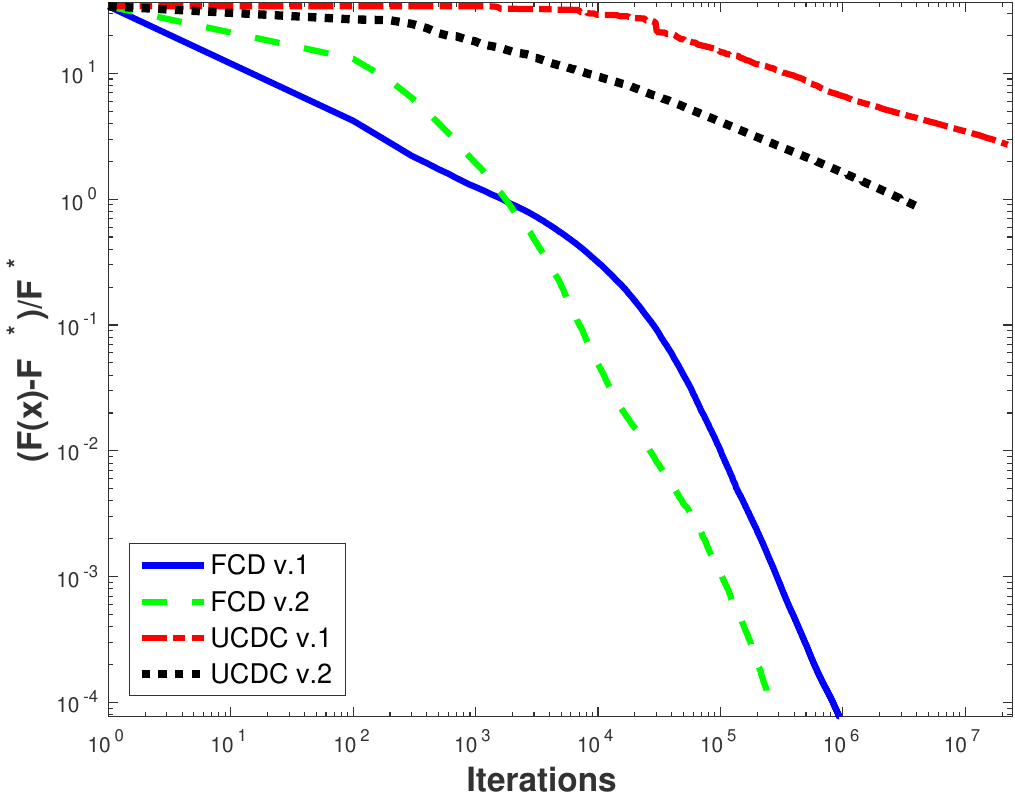}}
\subfloat[kdda,    $\frac{F(x) - F^*}{F^*}$ vs iterations]{\label{fig2c}\includegraphics[scale=0.27]{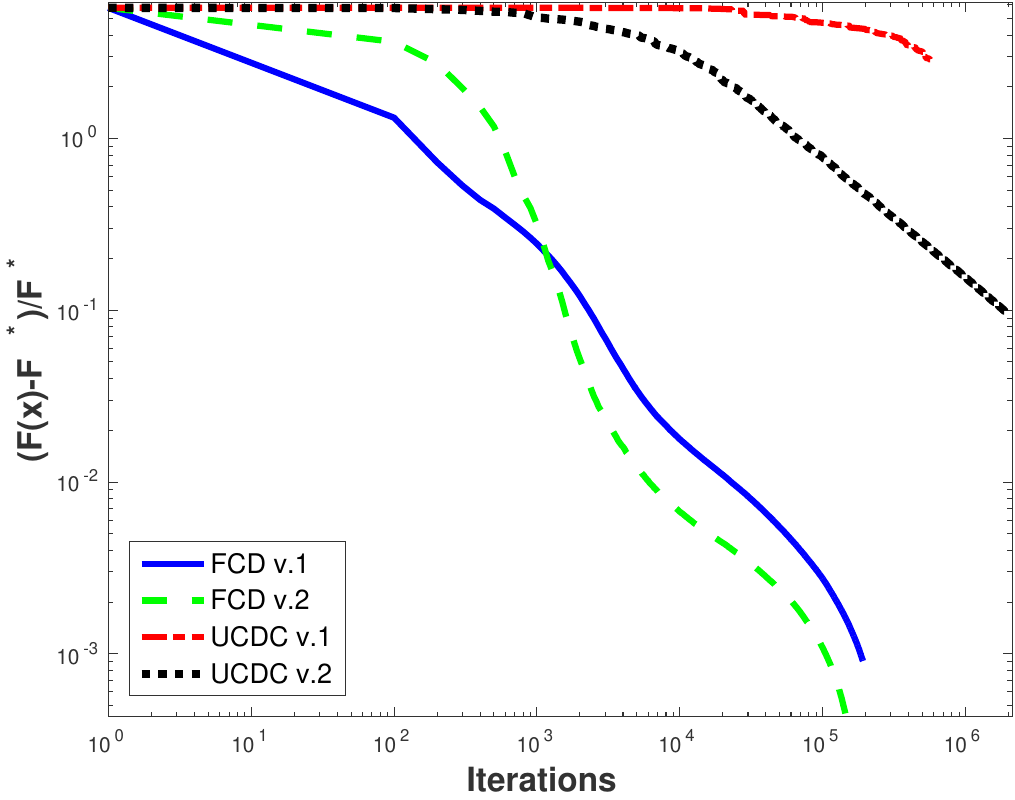}}
\subfloat[kddb,    $\frac{F(x) - F^*}{F^*}$ vs iterations]{\label{fig2a}\includegraphics[scale=0.27]{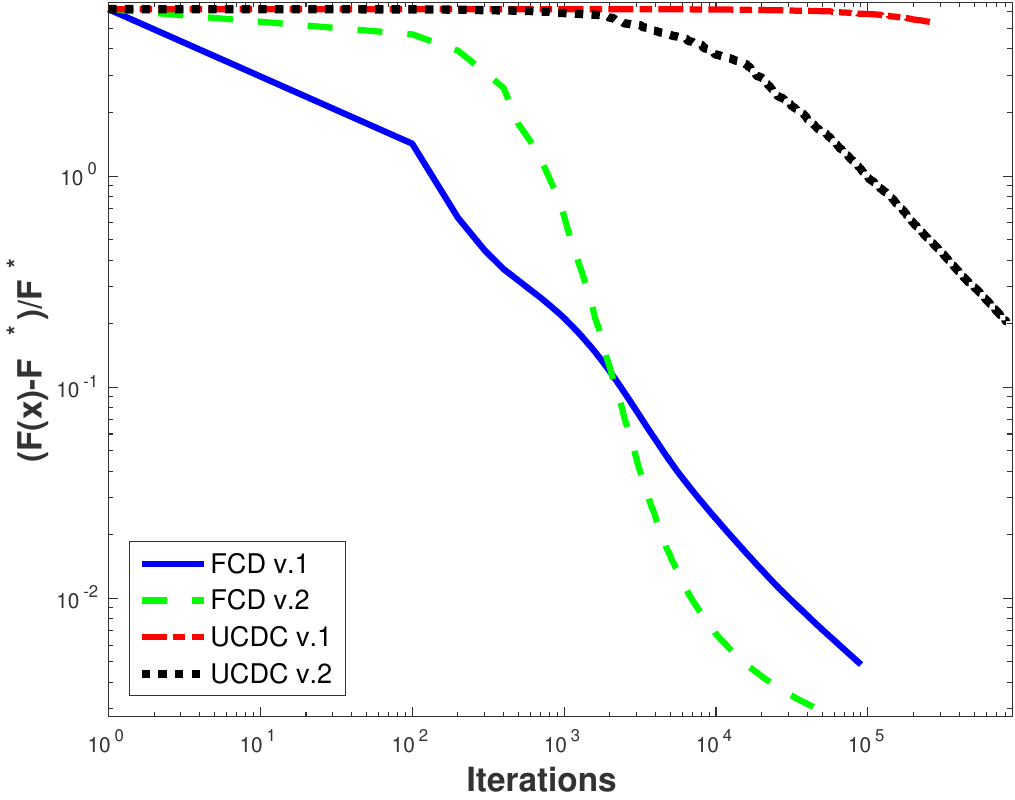}}\\
\subfloat[webspam, $\frac{F(x) - F^*}{F^*}$ vs time]{\label{fig_rw_webspam_fn_2}\includegraphics[scale=0.27]{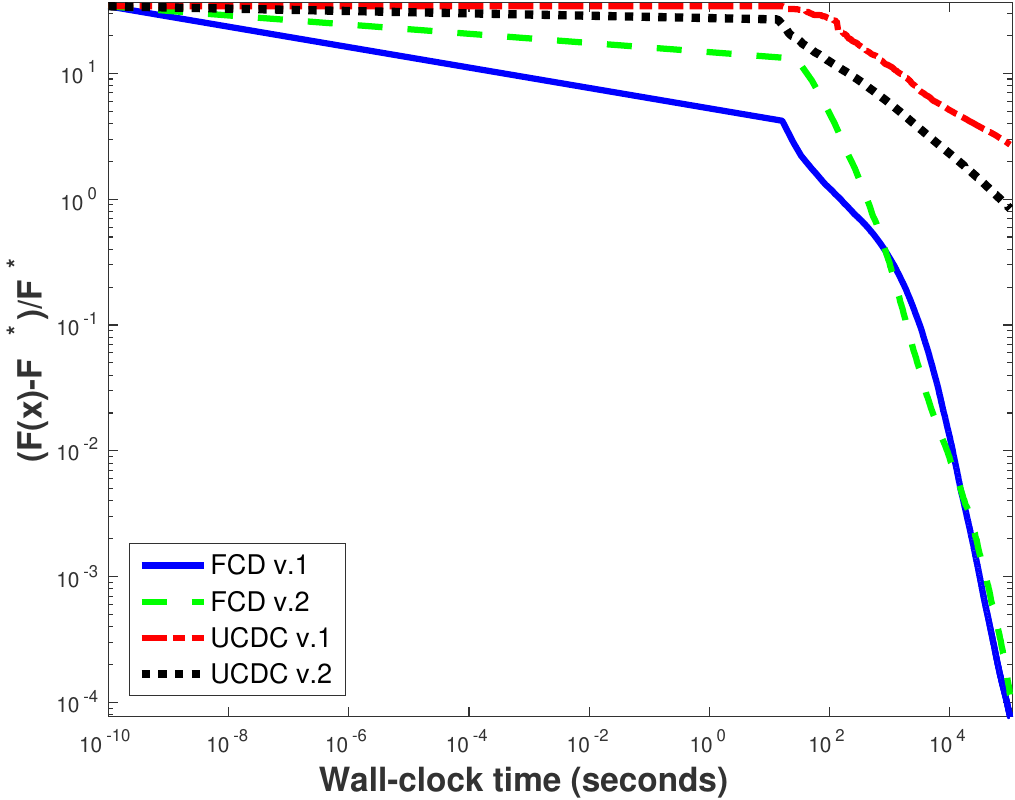}}
\subfloat[kdda,    $\frac{F(x) - F^*}{F^*}$ vs time]{\label{fig2d_2}\includegraphics[scale=0.27]{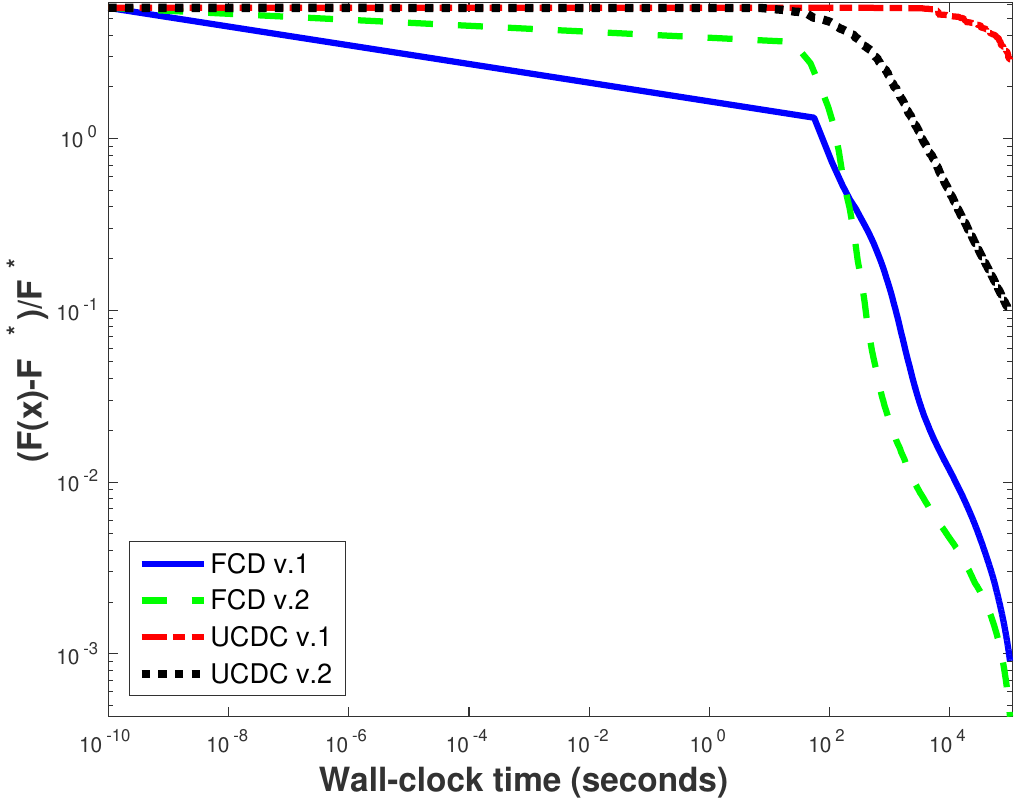}}
\subfloat[kddb,    $\frac{F(x) - F^*}{F^*}$ vs time]{\label{fig2b_2}\includegraphics[scale=0.27]{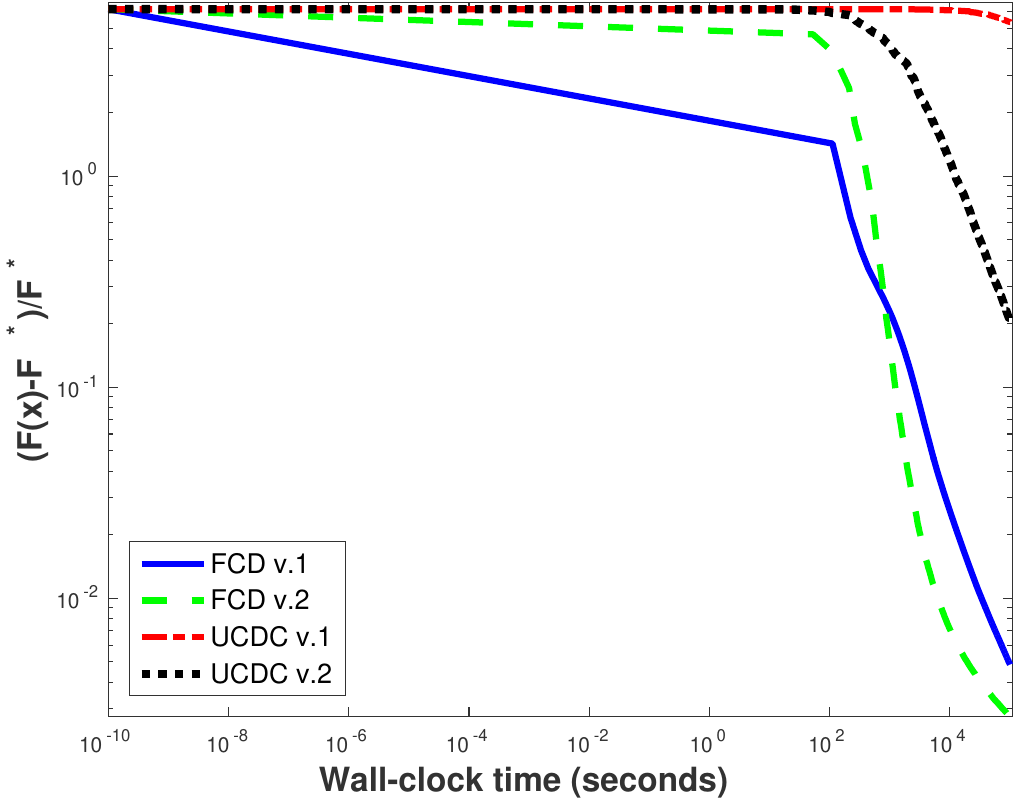}}\\
\caption{Performance of FCD and UCDC on 3 large scale $\ell_1$-regularized logistic regression problems.
The first row shows the function value $\frac{F(x) - F^*}{F^*}$ vs the number of iterations, while the second row shows the function value $\frac{F(x) - F^*}{F^*}$ vs time. Columns 1, 2 and 3 correspond to data sets webspam, `kdda' (kdd2010 (algebra)) and `kddb' (kdd2010 (bridge to algebra)), respectively. All plots are in log-scale.
For some figures, i.e., Figures \ref{fig_rw_webspam_fn_2}, \ref{fig2d_2} and \ref{fig2b_2}, we measured $\frac{F(x) - F^*}{F^*}$ every $100$ iterations of the algorithms, which explains the initial rapid decrease in $\frac{F(x) - F^*}{F^*}$.}
\label{Fig_big}
\end{figure}

FCD performs very well on all data sets. Notice that both versions of FCD achieve a lower objective function value than UCDC in the allocated time. The subfigures corresponding to function value vs iterations demonstrate that an iteration of FCD is slightly more expensive in terms of computation time than UCDC (i.e., fewer iterations of FCD are completed in the allocated time compared with UCDC). However, the expense of each iteration of FCD is offset by a greater reduction in function value. Clearly, including partial curvature information (as in FCD) is useful.

The results shown in Figure~\ref{Fig_big} are even more striking. Note that data sets kdda and kddb are very large. Despite this, FCD manages to decrease the objective function quickly. On the other hand, UCDC~v.1 barely manages to decrease the objective function at all in the allocated time, while UCDC~v.2 only does a little better. Again, this highlights the importance of partial second order information. The cost of including the curvature information is again offset by the rapid reduction in the objective function value. It is interesting to note that, on these data sets, FCD v.1 and v.2 perform similarly, which shows that even a little curvature information is beneficial. (Recall that the two variants only differ in the choice of $H_k^S$; see Table~\ref{AlgTable}.) Clearly, FCD outperforms UCDC on these problems.

Overall, FCD v.1 makes rapid progress decreasing the objective value initially, whereas FCD v.2 is better at decreasing the objective function closer to the solution. We believe that FCD v.2 is faster for more accurate solutions due to the second-order information that is capture by using the Hessian to construct the subproblem of the algorithm.

We now comment on an important observation about the running time required by the line search for FCD.
It is often believed that a line search might be a task that is too expensive for large-scale problems. Therefore, constant or diminishing step-sizes are usually preferred in order to avoid extra function evaluations, which are often required by line search techniques.
However, Table~\ref{LSsmall} shows that, for the experiments corresponding to Figures \ref{Fig_small} and \ref{Fig_big}, the additional time required to perform the function evaluations required by line search for FCD are often \emph{negligible}.
In particular, notice that columns three and five report the percentage of iterations in which a step-size of $\alpha=1$ was accepted by the line-search step in FCD. Perhaps surprisingly, for the kdd2010(algebra) dataset, FCD v.1 and v.2 \emph{accepted a step-size of $\alpha=1$ for $99.99\%$ of iterations.} At worst, FCD~v.2 applied to rcv1.binary, accepted a step size of 1 for 81.26\% of the iterations.
Moreover, the second and fourth columns of Table \ref{LSsmall} state the percentage of total time that was spent on line search computations (including all necessary function evaluations. At worst and only for one experiment, FCD~v.1 spent 55.38\% of the total time on line search computations for the rcv1.binary problem, whereas FCD~v.2  spent only 3.62\% of the total time on line search computations for the same problem. On average, over both algorithm variants, and each of the 6 problems, 9.91\% of the total running time was spent on line search computations.
We find it striking the fact that FCD~v.1, which uses only diagonal Hessian information, accepts unit step-sizes for at least $95.95\%$ of its iterations. We believe that this reveals the power of line-search and it shows that it shouldn't be discarded easily in practice.

\begin{remark}
Furthermore, the cost of the line-search in terms of worst-case function evaluations is not so high. Because we use a standard line search, it is straightforward to establish that, in the worst-case, one requires $\log(\alpha_S)$ function evaluations for the line search at each iteration, where $\alpha_S$ is defined in \eqref{Eq_alphamin}. In particular, because $\alpha$ decreases geometrically (for example, it is halved at every iteration) and $\alpha_S$ is the minimum $\alpha$ such that the termination criteria of line-search are satisfied, we simply have $(1/2)^s < \alpha_S$, which means that, in the worst case, we require $s> \log_2(\alpha_S)$ function evaluations for each line search.
\end{remark}

\begin{table}[h!]
	\centering
	\caption{Comparison of the line search costs for FCD for the experiments in Figures \ref{Fig_small} and \ref{Fig_big}. Columns 2 and 4 report the percentage of the overall running time that was spent on line search calculations (including all necessary function evaluations). Columns 3 and 5 report the percentage of iterations where a step-size equal to $1$ satisfied the line search termination conditions.}
\begin{tabular}{|l|c|c|c|c|}
\hline
\multirow{ 2}{*}{\hspace{1.2cm} Problem} & \multicolumn{2}{|c|}{FCD v.1} & \multicolumn{2}{|c|}{FCD v.2} \\
\cline{2-5}
& \% of time & \% of iterations &  \% of time & \% of iterations  \\
\hline
kdd2010 (algebra)	                 & $7.31\%$ &$99.99\%$ &$7.07\%$ &  $99.99\%$ \\
kdd2010 (bridge to algebra)	&$6.67\%$ &$99.99\%$ & $4.37\%$&$99.99\%$ \\
news20.binary	                 & $11.09\%$ &$99.97\%$ &$1.44\%$ & $95.95\%$ \\
rcv1.binary	                &$55.38\%$ &$81.26\%$ &$3.62\%$ &$99.46\%$ \\
url	                                 & $7.62\%$ & $99.92\%$ & $11.63\%$ &$99.81\%$ \\
webspam	                         & $2.23\%$ & $99.99\%$ & $0.56\%$& $99.89\%$\\
\hline
	\end{tabular}
	\label{LSsmall}
\end{table}

\subsection{Results for $\ell_2$-regularized logistic regression}

Figure~\ref{Fig_small_l2} shows the results of FCD and UCDC on several of the smaller data sets, and Figure~\ref{Fig_big_l2} shows the results of FCD and UCDC on several of the larger data sets
for $\ell_2$-regularized logistic regression. Similarly to the $\ell_1$-regularized experiments in Subsection \ref{subsec:l1regexper} FCD~v.1 or v.2 or both are faster in terms of running time compared to UCDC.
However, due to strong convexity imposed by $\ell_2$-regularization we observe that FCD~v.2 is clearly better than FCD~v.1, even for highly accurate solutions, despite that FCD~v.2 uses only diagonal Hessian information.
Strong convexity is also the reason that the performance gap between UCDC and FCD decreased for $\ell_2$-regularized logistic regression.
\begin{figure}[h!]\centering
\subfloat[news20.b, $\frac{F(x) - F^*}{F^*}$ vs iterations]{\label{fig_rw_webspam_its_l2}\includegraphics[scale=0.27]{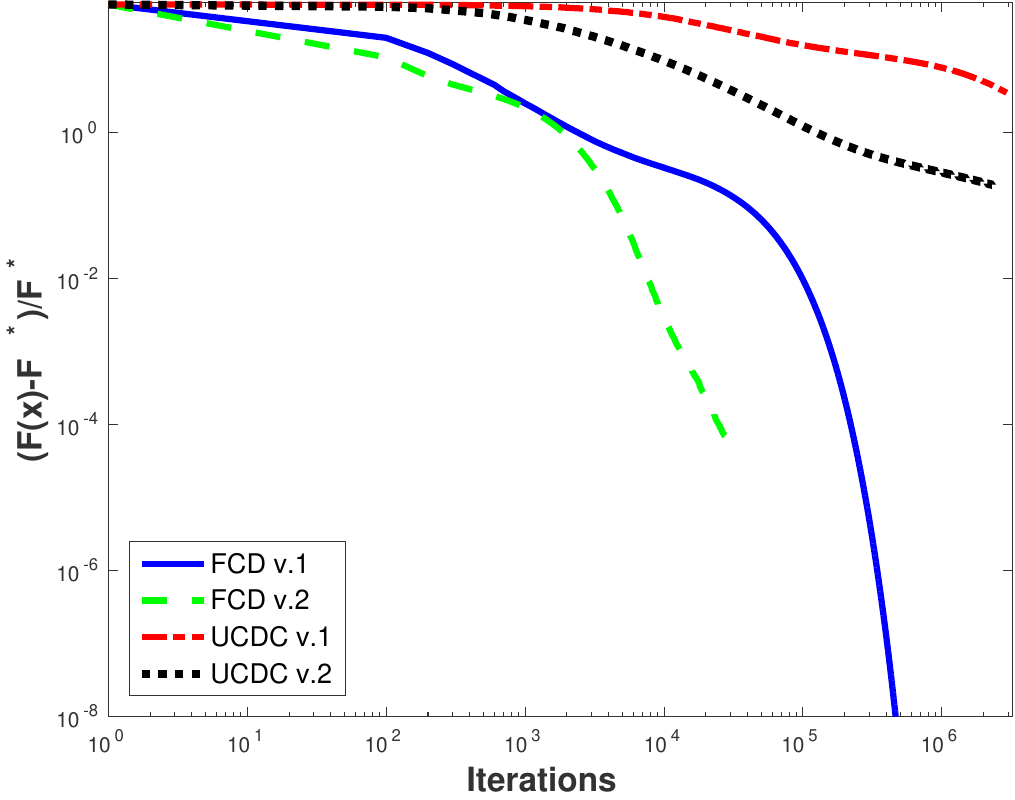}}
\subfloat[rcv1,          $\frac{F(x) - F^*}{F^*}$ vs iterations]{\label{fig2c_l2}\includegraphics[scale=0.27]{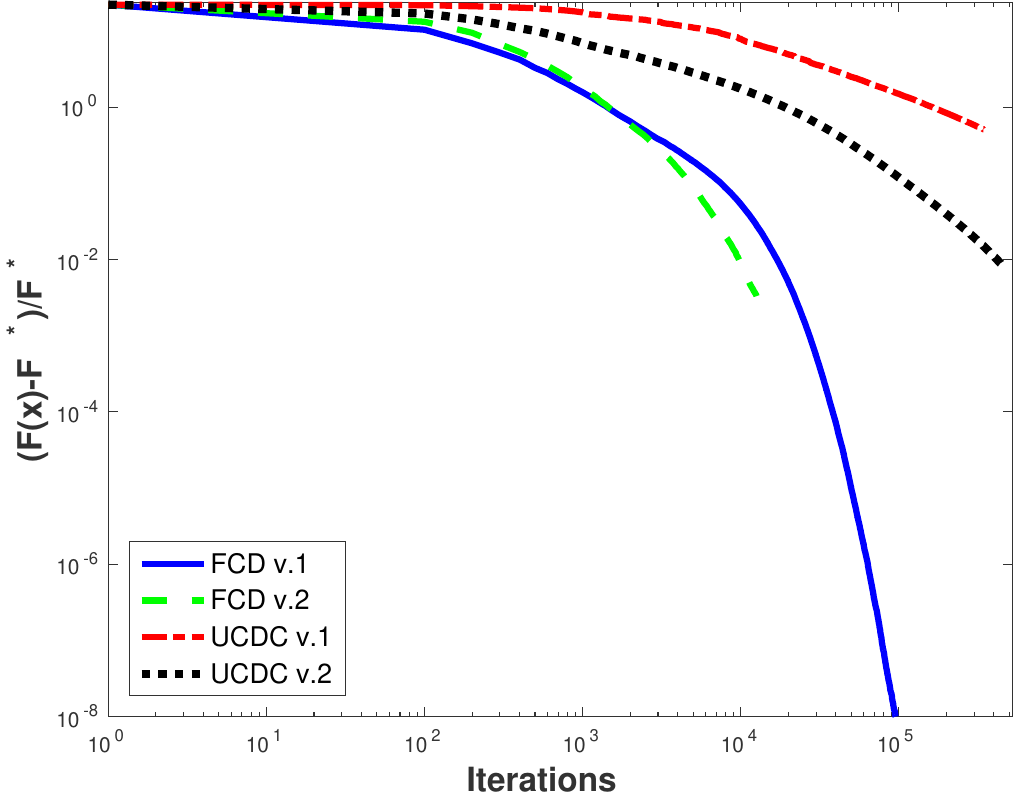}}
\subfloat[url,            $\frac{F(x) - F^*}{F^*}$ vs iterations]{\label{fig2a_l2}\includegraphics[scale=0.27]{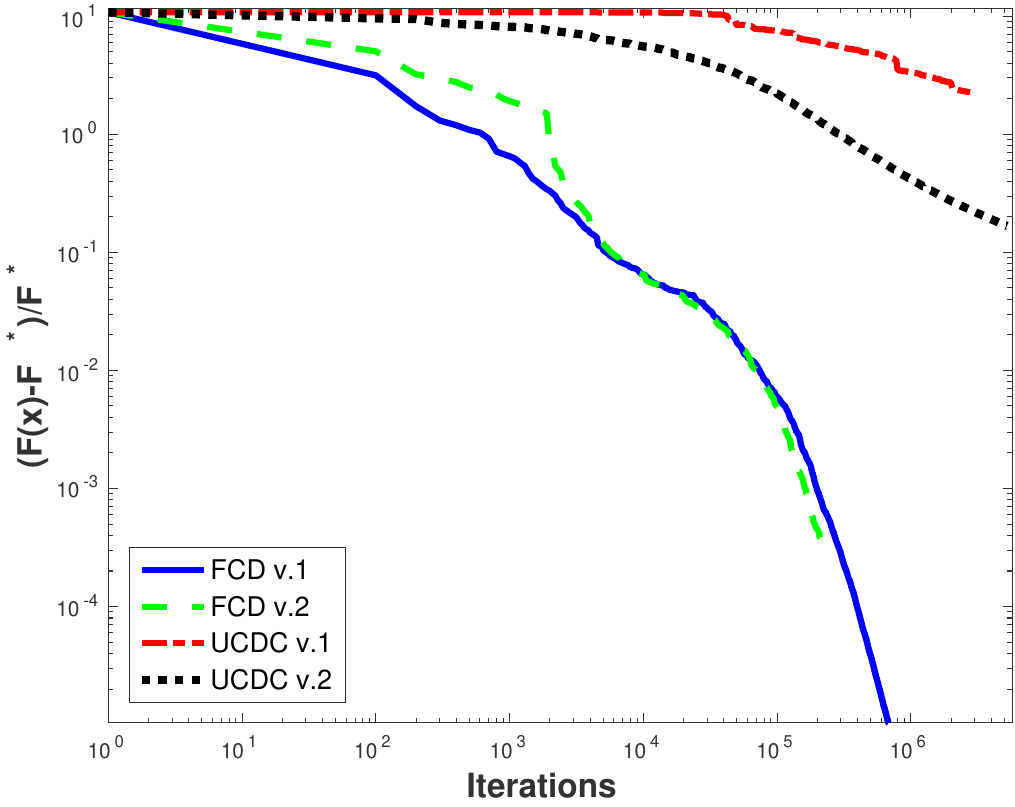}}\\
\subfloat[news20.b, $\frac{F(x) - F^*}{F^*}$ vs time]{\label{fig_rw_webspam_fn_l2}\includegraphics[scale=0.27]{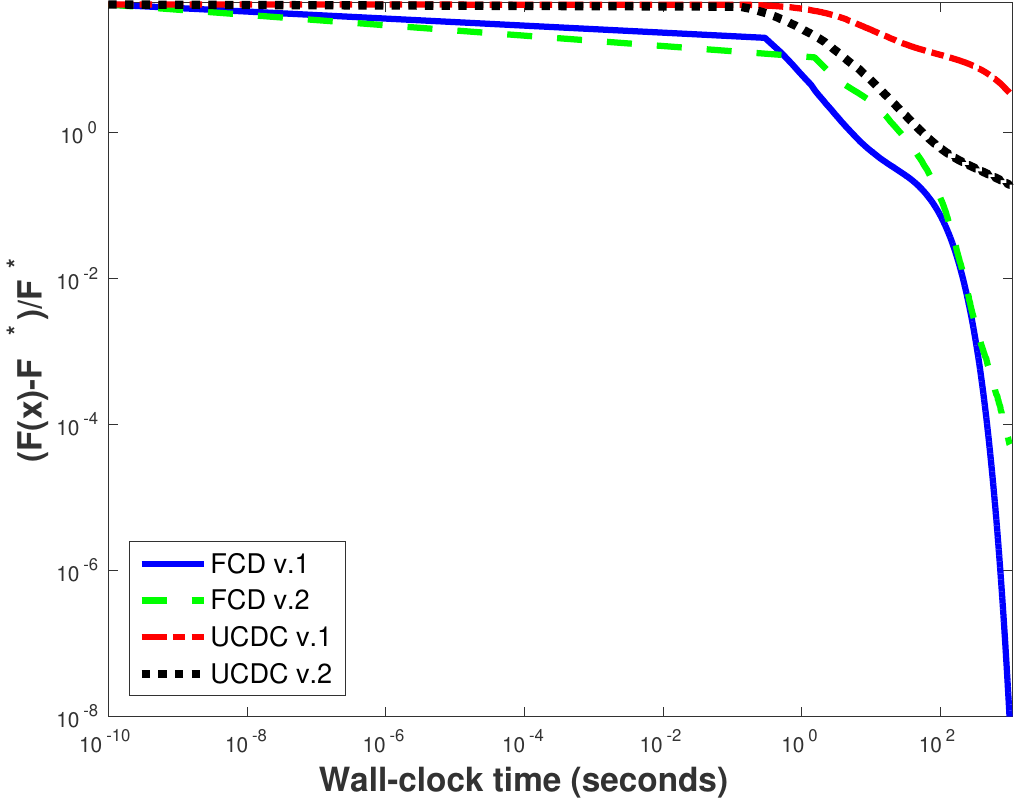}}
\subfloat[rcv1,          $\frac{F(x) - F^*}{F^*}$ vs time]{\label{fig2d_l2}\includegraphics[scale=0.27]{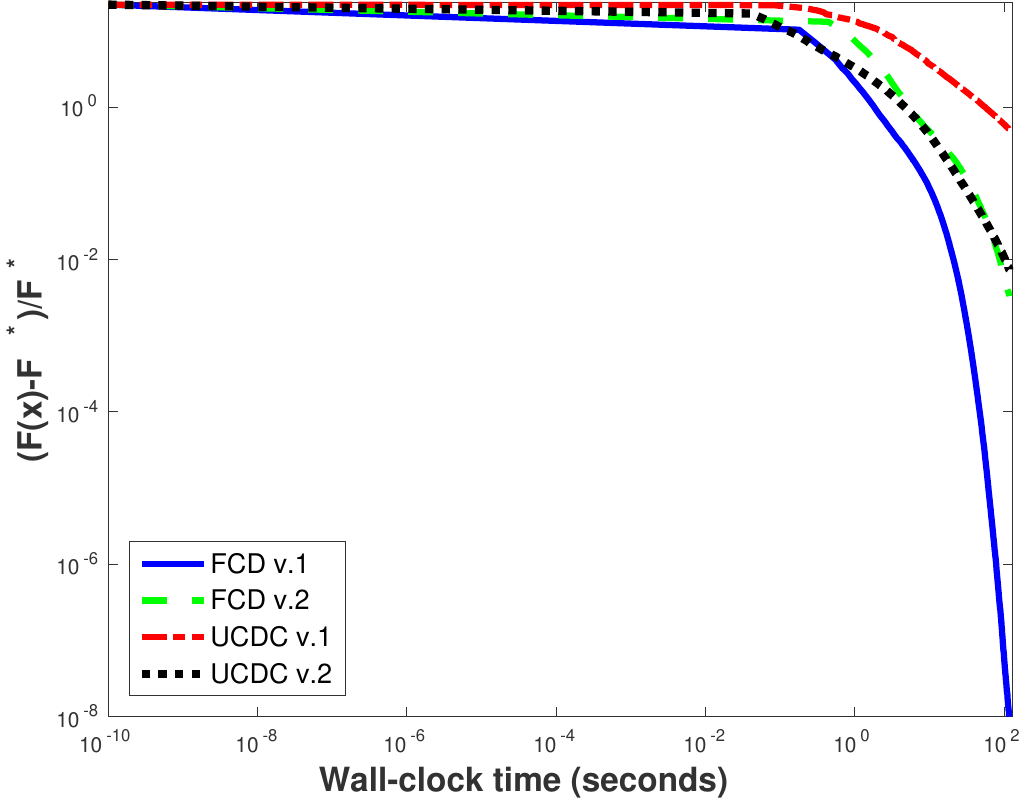}}
\subfloat[url,           $\frac{F(x) - F^*}{F^*}$ vs time]{\label{fig2b_l2}\includegraphics[scale=0.27]{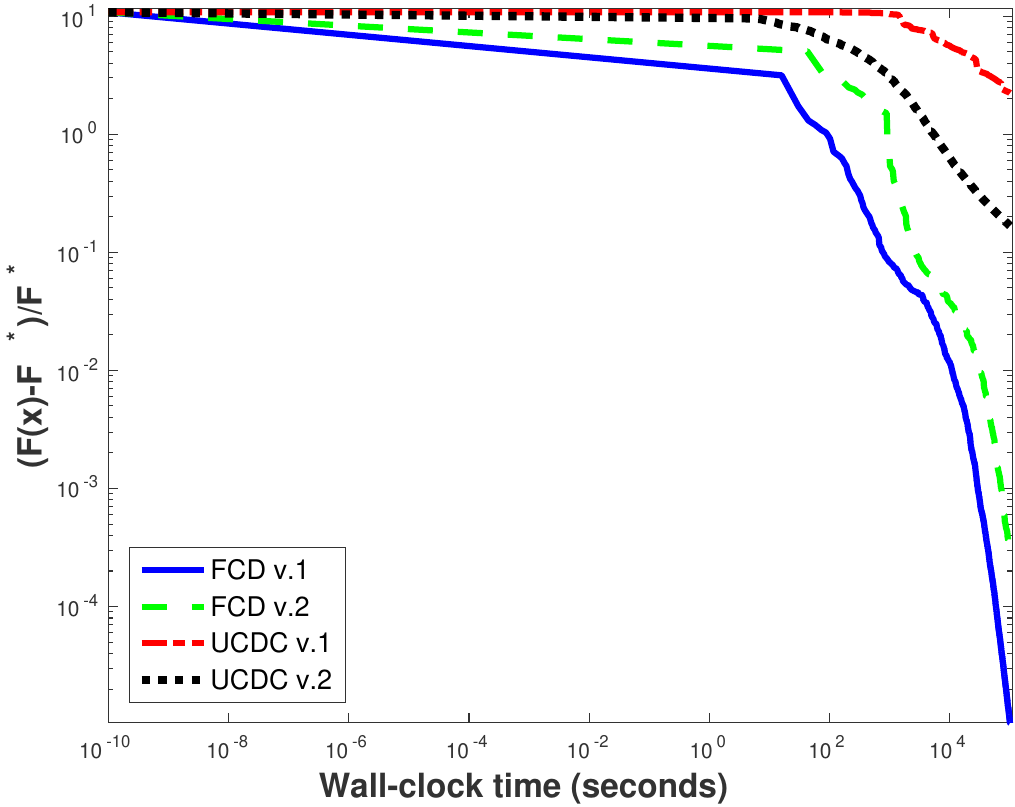}}\\
\caption{Performance of FCD and UCDC on 3 large scale $\ell_2$-regularized logistic regression problems.
The first row shows the function value $\frac{F(x) - F^*}{F^*}$ vs the number of iterations, while the second row shows the function value $\frac{F(x) - F^*}{F^*}$ vs time. Columns 1, 2 and 3 correspond to data sets news20.binary, rcv1 and url, respectively. All plots are in log-scale.
For some figures, i.e., Figures \ref{fig_rw_webspam_fn_l2}, \ref{fig2d_l2} and \ref{fig2b_l2}, we measured $\frac{F(x) - F^*}{F^*}$ every $100$ iterations of the algorithms, which explains the initial rapid decrease in $\frac{F(x) - F^*}{F^*}$.}
\label{Fig_small_l2}
\end{figure}
\begin{figure}[h!]
\centering
\subfloat[webspam, $\frac{F(x) - F^*}{F^*}$ vs iterations]{\label{fig_rw_webspam_its_l2}\includegraphics[scale=0.27]{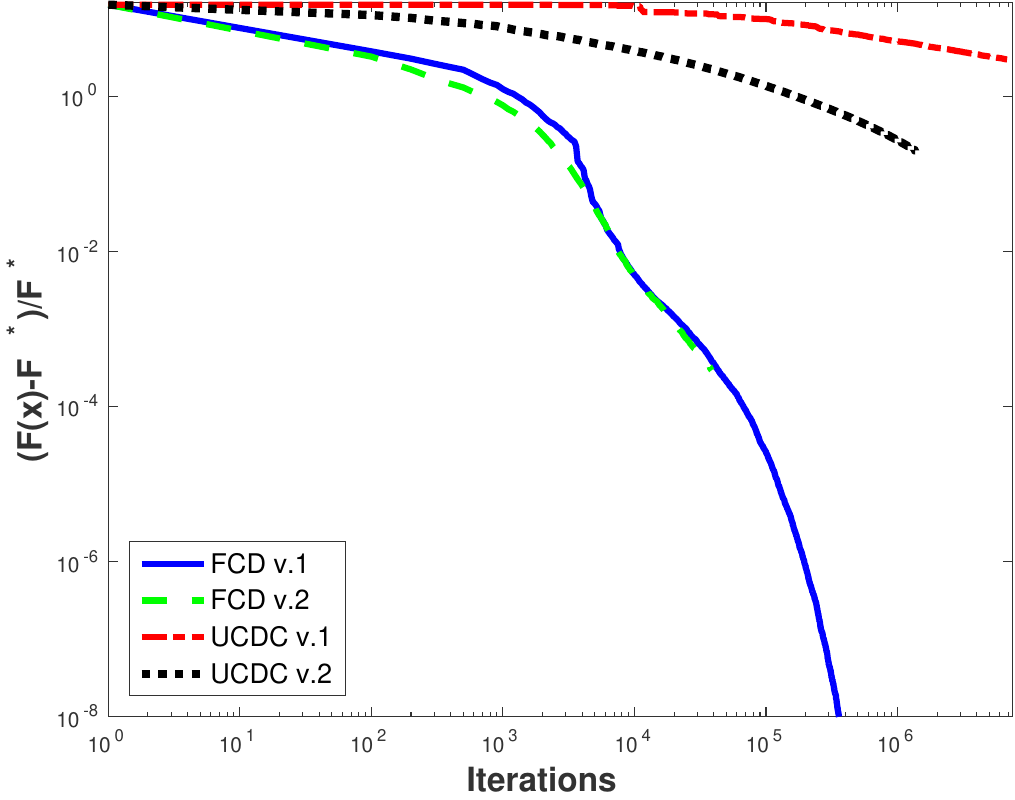}}
\subfloat[kdda,    $\frac{F(x) - F^*}{F^*}$ vs iterations]{\label{fig2c_l2}\includegraphics[scale=0.27]{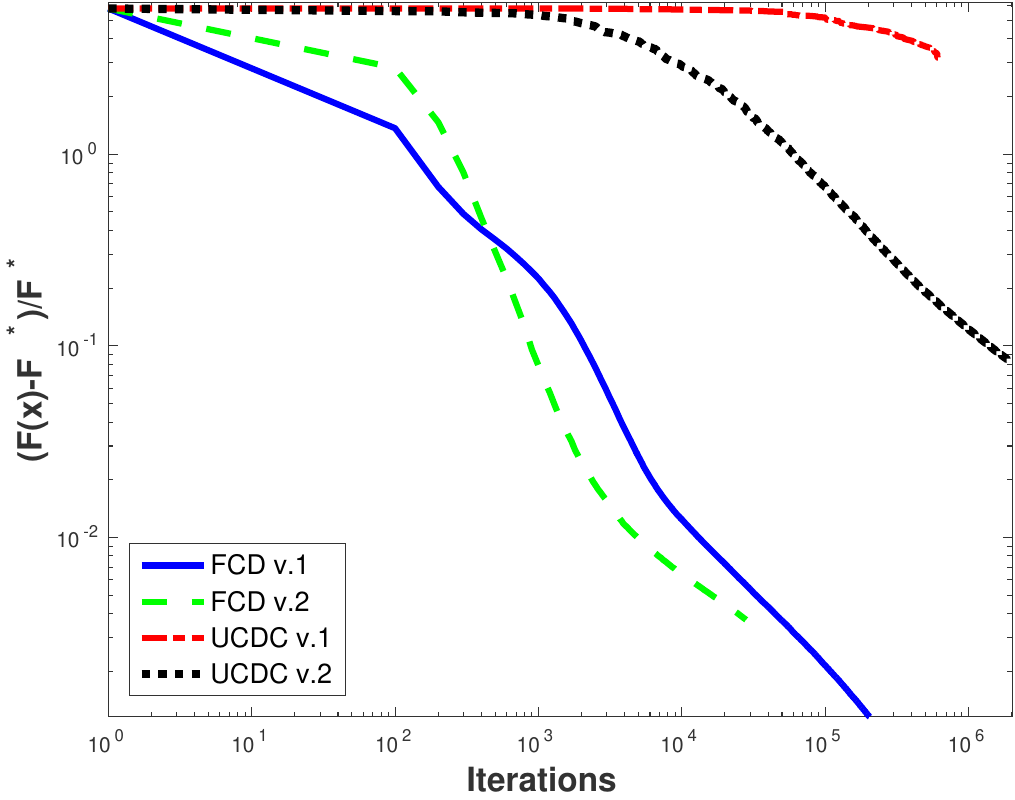}}
\subfloat[kddb,    $\frac{F(x) - F^*}{F^*}$ vs iterations]{\label{fig2a_l2}\includegraphics[scale=0.27]{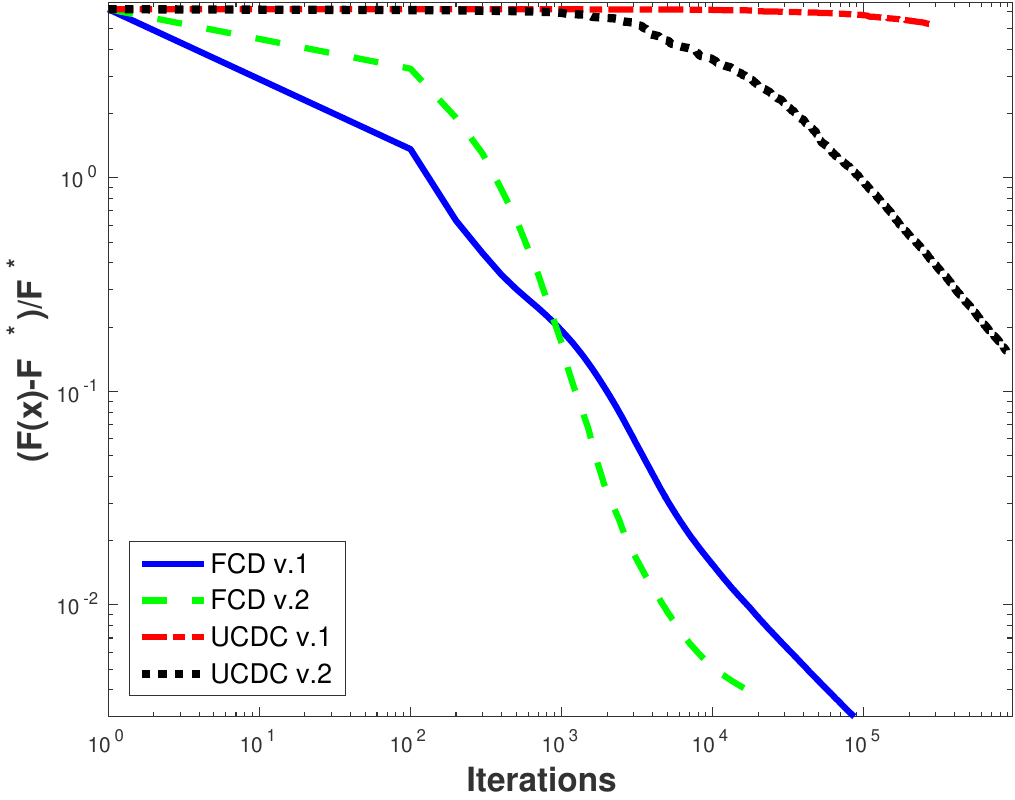}}\\
\subfloat[webspam, $\frac{F(x) - F^*}{F^*}$ vs time]{\label{fig_rw_webspam_fn_2_l2}\includegraphics[scale=0.27]{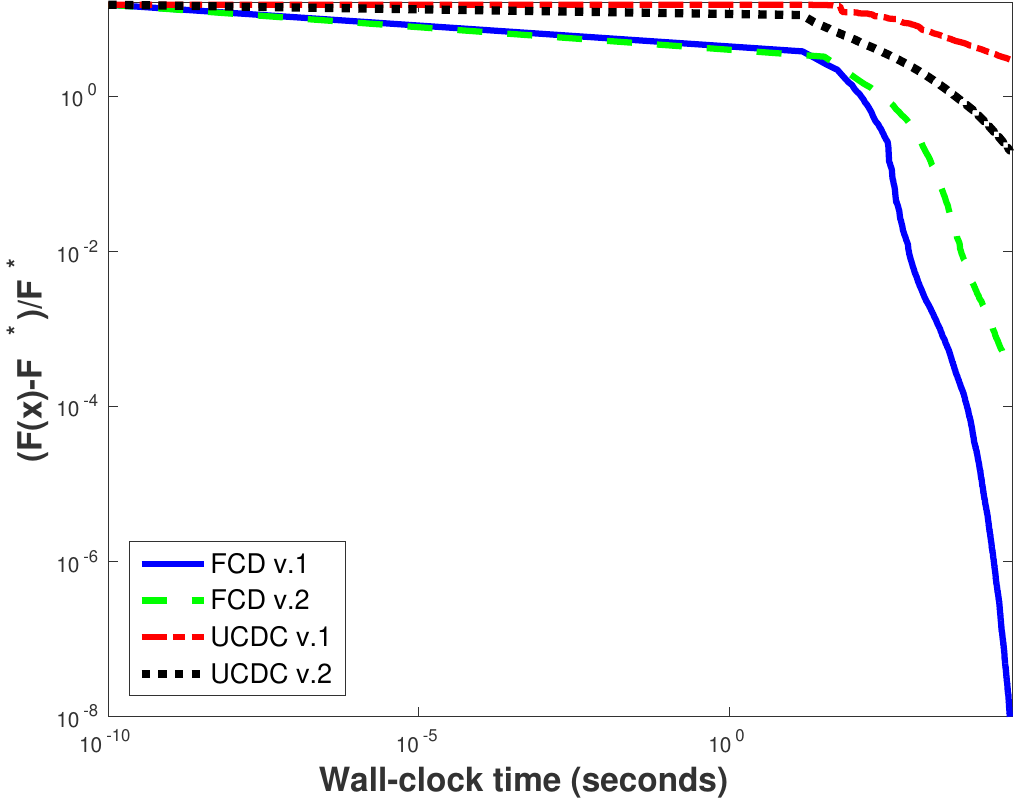}}
\subfloat[kdda,    $\frac{F(x) - F^*}{F^*}$ vs time]{\label{fig2d_2_l2}\includegraphics[scale=0.27]{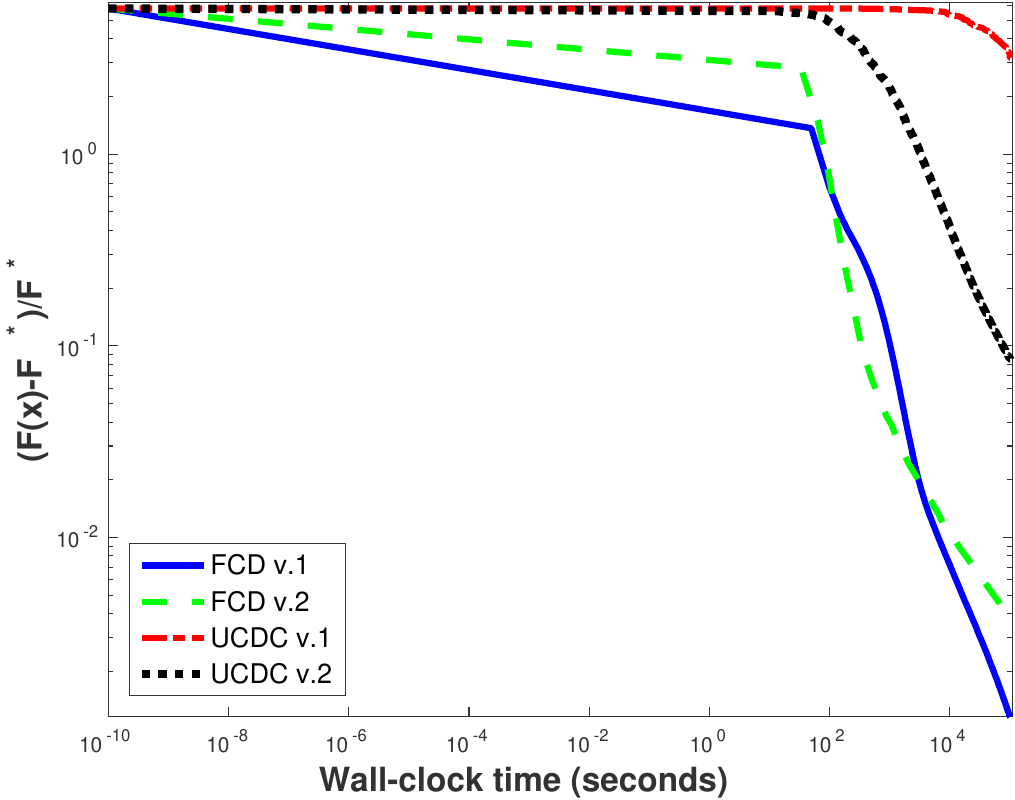}}
\subfloat[kddb,    $\frac{F(x) - F^*}{F^*}$ vs time]{\label{fig2b_2_l2}\includegraphics[scale=0.27]{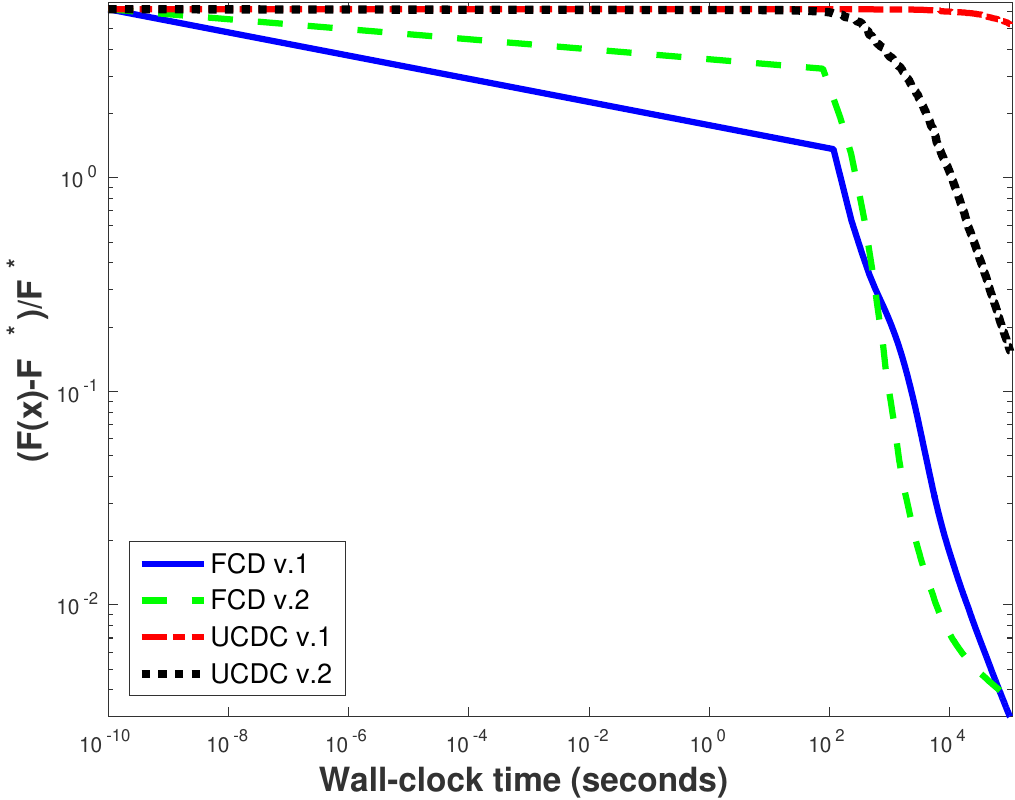}}\\
\caption{Performance of FCD and UCDC on 3 large scale $\ell_1$-regularized logistic regression problems.
The first row shows the function value $\frac{F(x) - F^*}{F^*}$ vs the number of iterations, while the second row shows the function value $\frac{F(x) - F^*}{F^*}$ vs time. Columns 1, 2 and 3 correspond to data sets webspam, `kdda' (kdd2010 (algebra)) and `kddb' (kdd2010 (bridge to algebra)), respectively. All plots are in log-scale.
For some figures, i.e., Figures \ref{fig_rw_webspam_fn_2_l2}, \ref{fig2d_2_l2} and \ref{fig2b_2_l2}, we measured $\frac{F(x) - F^*}{F^*}$ every $100$ iterations of the algorithms, which explains the initial rapid decrease in $\frac{F(x) - F^*}{F^*}$.}
\label{Fig_big_l2}
\end{figure}

Finally, Table \ref{LSsmall_l2} shows performance statistics for line-search in FCD for $\ell_2$-regularized logistic regression. Note that the performance of line-search is even better for $\ell_2$-regularized logistic regression compared to $\ell_1$-regularized logistic regression.
We believe that the reason is strong convexity that is guaranteed by $\ell_2$-regularization.

\begin{table}[h!]
	\centering
	\caption{Comparison of the line search costs for FCD for the experiments in Figures \ref{Fig_small_l2} and \ref{Fig_big_l2}. Columns 2 and 4 report the percentage of the overall running time that was spent on line search calculations (including all necessary function evaluations). Columns 3 and 5 report the percentage of iterations where a step-size equal to $1$ satisfied the line search termination conditions.}
\begin{tabular}{|l|c|c|c|c|}
\hline
\multirow{ 2}{*}{\hspace{1.2cm} Problem} & \multicolumn{2}{|c|}{FCD v.1} & \multicolumn{2}{|c|}{FCD v.2} \\
\cline{2-5}
& \% of time & \% of iterations &  \% of time & \% of iterations  \\
\hline
kdd2010 (algebra)	                 & $7.32\%$ &$99.99\%$ &$0.96\%$ &  $99.99\%$ \\
kdd2010 (bridge to algebra)	&$6.77\%$ &$99.99\%$ & $1.73\%$&$99.99\%$ \\
news20.binary	                 & $11.92\%$ &$99.99\%$ &$0.74\%$ & $99.98\%$ \\
rcv1.binary	                &$16.93\%$ &$99.99\%$ &$2.58\%$ &$99.90\%$ \\
url	                                 & $8.55\%$ & $99.99\%$ & $3.31\%$ &$99.99\%$ \\
webspam	                         & $3.05\%$ & $99.99\%$ & $0.36\%$& $99.98\%$\\
\hline
	\end{tabular}
	\label{LSsmall_l2}
\end{table}

\section{Conclusion}
In this work we have presented a randomized block coordinate descent method (FCD) that can be applied to convex composite functions of the form \eqref{Def_F}.
The proposed method allows the coordinates to be selected randomly via a non-fixed block structure, incorporates partial curvature information via a user defined matrix $H_k^S$, and allows inexact updates to be used. These features make the algorithm extremely flexible. The algorithm is supported by high probability iteration complexity results. Moreover, we presented several synthetic and real world, large scale problems where FCD is shown to outperform the current state-of-the-art coordinate descent method.

\section{Acknowledgments}
We would like to thank the anonymous reviewers for their helpful comments and suggestions, which led to improvements of an earlier version of this paper. 

\bibliographystyle{plain}
\begin{small}
  \bibliography{references}
\end{small}

\end{document}